%% file: Multi-level_Picard_56.tex
\documentclass[11pt,a4paper,reqno]{amsart}
\pdfoutput=1
\usepackage{times}
\usepackage[margin=1.0in,tmargin=0.9in,bmargin=0.80in]{geometry}
\usepackage[dvipsnames]{xcolor}
\usepackage[english]{babel} 

\definecolor{bred}{rgb}{0.8,0,0}

\usepackage{amsmath,amssymb,graphicx,epsfig}
\usepackage{mathtools}
\usepackage{enumerate,amsthm,epstopdf}
\usepackage[flushleft]{threeparttable}
\usepackage{multirow}
\usepackage{longtable}
\usepackage[utf8]{inputenc} 
\usepackage[T1]{fontenc}    
\usepackage{bbm}
\usepackage{xargs}
\usepackage{latexsym}
\usepackage{wasysym}
\usepackage{float}
\usepackage{hyperref}       
\usepackage{url}            
\usepackage{booktabs}       
\usepackage{amsfonts}       
\usepackage{nicefrac}       
\usepackage{microtype}      
\usepackage{cleveref}
\usepackage[textsize=footnotesize]{todonotes}
\hypersetup{colorlinks,linkcolor={blue},citecolor={bred},urlcolor={blue}}
\usepackage[numbers]{natbib}
\usepackage{array}
\usepackage{subfigure}
\usepackage{dsfont}

\usepackage{stmaryrd}
\usepackage{algorithm,algpseudocode}
\usepackage{tabularx}
\usepackage{algorithmicx}
\allowdisplaybreaks

\numberwithin{equation}{section}

\makeatletter
\newcommand{\multiline}[1]{%
  \begin{tabularx}{\dimexpr\linewidth-\ALG@thistlm}[t]{@{}X@{}}
    #1
  \end{tabularx}
}
\makeatother

\newtheorem{theorem}{Theorem}[section]
\newtheorem{proposition}[theorem]{Proposition}
\newtheorem{lemma}[theorem]{Lemma}
\newtheorem{corollary}[theorem]{Corollary}
\newtheorem{remark}[theorem]{Remark}
\newtheorem{definition}[theorem]{Definition}

\newtheorem{assumption}[theorem]{Assumption}

\newcommand{\R}{\mathbb{{R}}}

\newcommand{\N}{\mathbb{{N}}}

\newcommand\bE{\mathbb{E}}
\newcommand\bF{\mathbb{F}}

\newcommand\bN{\mathbb{N}}
\newcommand\bR{\mathbb{R}}
\newcommand\bP{\mathbb{P}}
\newcommand\bZ{\mathbb{Z}}

\newcommand\bS{\mathbb{S}}

\newcommand\cB{\mathcal{B}}

\newcommand\cF{\mathcal{F}}

\newcommand\cJ{\mathcal{J}}

\newcommand\cL{\mathcal{L}}
\newcommand\cM{\mathcal{M}}

\newcommand\cR{\mathcal{R}}

\newcommand{\cal}{\curvearrowleft}

\makeatletter
 \newcommand{\sumstar}
 {\operatornamewithlimits{\sum@\kern-.2em\raise1ex\hbox{*}}}
 \makeatother

\newcommand{\<}{\langle}
\renewcommand{\>}{\rangle}

\allowdisplaybreaks

\usepackage{etoolbox}
\AtBeginEnvironment{tabular}{\small}
\newcolumntype{R}[1]{>{\raggedleft\let\newline\\\arraybackslash\hspace{0pt}}m{#1}}

\begin{document}

\title[]{Multilevel Picard approximation algorithm for semilinear  partial integro-differential equations and its complexity analysis}

\author[A.Neufeld]{Ariel Neufeld}
\address{Nanyang Technological University, Division of Mathematical Sciences, Singapore}
\email{ariel.neufeld@ntu.edu.sg}

\author[S. Wu]{Sizhou Wu}
\address{Shanghai University of Finance and Economics, School of Mathematics, Shanghai, China}
\email{wusizhou@sufe.edu.cn}

\date{}
\thanks{
	Financial support by the Nanyang Assistant Professorship Grant (NAP Grant) \textit{Machine Learning based Algorithms in Finance and Insurance} is gratefully acknowledged.}
\keywords{Multilevel Picard approximation, Nonlinear PIDE, overcoming the curse of dimensionality, Monte Carlo methods, Feynman-Kac representation}

\subjclass[2010]{}

\begin{abstract}
In this paper we introduce a multilevel Picard approximation algorithm for  semilinear parabolic partial integro-differential equations (PIDEs). We prove that the numerical approximation scheme converges to the unique viscosity solution of the PIDE under consideration. To that end, we derive a Feynman-Kac representation for the unique viscosity solution of the semilinear PIDE, extending the classical Feynman-Kac representation for linear PIDEs.
Furthermore, we show that the algorithm  does not suffer from the curse of dimensionality, i.e.\ the computational complexity of the algorithm is bounded polynomially in the dimension $d$ and the reciprocal of the prescribed accuracy $\varepsilon$. We also provide a numerical example in up to 10'000 dimensions to demonstrate its applicability.
\end{abstract}
\maketitle

\section{\textbf{Introduction}}
Nonlinear partial integro-differential equations (PIDEs) have 
many important applications in finance, 
physics,
biology,
economics,
and engineering; we refer to 
 \cite{ContTankov,CV2005_0,delong2013backward,oksendal2005stochastic} 
and the references therein, as well as \cite{Jentzen2022deepPIDE} 
for an excellent survey on applications of nonlocal PDEs in various fields.
 
In the last years, several papers proposed modern numerical methods to solve
 nonlinear PDEs  involving either neural networks
 \cite{deepsplitting,han2017deep,E2018deepRitz, germain2022approximation, 
 han2018solving, hure2020deep, LuMengMaoKarniadakis21,
  RaissiPerdikarisKarniadakis19,SirignanoSpiliopoulos18}, 
or multilevel Monte Carlo methods
 \cite{becker2020numericalMLP,beck2020overcomingElliptic,
beck2020overcoming,hutzenthaler2021multilevel,giles2019generalised,
hutzenthaler2019multilevel,HJKN2020,hutzenthalerkruse2020multilevel}, 
which are able to approximately solve high-dimensional nonlinear PDEs 
in up to 10'000 dimensions. 
Moreover, it has been proven in
 \cite{becker2020numericalMLP,beck2020overcoming,hutzenthaler2021multilevel,
 hutzenthaler2019multilevel,
HJKN2020,hutzenthalerkruse2020multilevel}, 
under some (typically Lipschitz-) continuity conditions on the corresponding
coefficients of the PDE under consideration, 
that multilevel Picard approximations do overcome the curse of dimensionality 
in the sense that 
the computational complexity of the algorithms is bounded polynomially 
in the dimension $d$ and the reciprocal of the prescribed  
accuracy $\varepsilon$.
 
However, the development of numerical methods and their complexity analysis 
for PIDEs is only at its infancy and is still emerging.
\cite{gonon2021deep,GS2021} proposed both machine learning based and Monte Carlo based methods for solving \textit{linear} PIDEs and showed that their algorithms do not suffer from the curse of dimensionality. In  \cite{al2019applications} the deep Galerkin algorithm of \cite{SirignanoSpiliopoulos18} has been extented to PIDEs. In \cite{yuan2022pinn} physics informed neural networks for solving nonlinear PIDEs have been constructed.
Deep neural network based algorithms for solving nonlinear PIDEs have been presented in \cite{castro2021deep, frey2022deep}. However, none of these papers provide a complexity analysis of their algorithms.
Moreover, very recently, \cite{Jentzen2022deepPIDE} introduced for non-local nonlinear PDEs with Neumann boundary conditions both a deep splitting algorithm as well as a multilevel Picard approximation algorithm to numerically solve the PDEs under consideration. However, no convergence analysis nor a complexity analysis of their algorithms has been provided. Moreover, the nonlocal term is expressed through a probability measure, which corresponds to a normalized L\'evy measure of \textit{finite activity.}

In this paper, we develop a multilevel Picard approximation (MLP) algorithm for semilinear PIDEs, and provide both a convergence analysis and a complexity analysis of the methodology. We show that the algorithm  does not suffer from the curse of dimensionality, i.e.\ the computational complexity of the algorithm is bounded polynomially in the dimension $d$ and the reciprocal of the prescribed  accuracy $\varepsilon$. 
We also provide a numerical example in up to 10'000 dimensions to demonstrate its applicability.
 Similar to the ideas in \cite{becker2020numericalMLP,beck2020overcoming,hutzenthaler2021multilevel,hutzenthaler2019multilevel,HJKN2020,hutzenthalerkruse2020multilevel}, the approach of developing such an approximation method for PIDEs is the following.
Consider the following nonlinear PIDE on $[0,T]\times \R^d$ satisfying $u^d(T,x)=g^d(x)$ and
\begin{equation}\label{PDE}
	\begin{split}
\frac{\partial}{\partial t}u^d(t,x)
&+\langle\nabla u^d(t,x),\mu^d(t,x)\rangle
+\frac{1}{2}\operatorname{Trace}
(\sigma^{d}(t,x)[\sigma^{d}(t,x)]^T\operatorname{Hess}_x(u^d)(t,x))
+f^d(t,x,u^d(t,x))\\
&
+\int_{\R^{d}}\left(u^d(t,x+\eta^{d}_t(z,x))-u^d(t,x)
-\langle\nabla u^d(t,x),\eta^{d}_t(z,x)\rangle\right)\,\nu^{d}(dz)
=0.
\end{split}
\end{equation}
In this paper, we restrict ourselves to semilinear parabolic PIDEs which are only nonlinear in their solutions $u(t,x)$ but not in their derivatives.
To be able to obtain a converge result of our numerical method, we first prove that under suitable conditions on the coefficients (see Section \ref{section setting}), the PIDE has a unique viscosity solution, and its unique solution admits a
 Feynman-Kac representation 
\begin{equation}                                                   \label{FK}
	u^d(t,x)=\bE\left[g^d(X^{d,t,x}_{T})\right]
	+\int_t^T\bE\left[f^d(s,X^{d,t,x}_{s},u^d(s,X^{d,t,x}_{s}))\right]\,ds, 
\end{equation}
where
\begin{equation}                                               \label{SDE intro}                              
	dX^{d,t,x}_{s}
	=\mu^d(s,X^{d,t,x}_{s-})\,ds
	+\sigma^{d}(s,X^{d,t,x}_{s-})\,dW^{d}_s
	+\int_{\R^{d}}\eta^{d}_s(X^{d,t,x}_{s-},z)
	\,\tilde{\pi}^{d}(dz,ds),
\end{equation}
with $\tilde{\pi}^{d}(dz,dt):=\pi^{d}(dz,dt)-\nu^{d}(dz)\otimes dt$
being the compensated Poisson random measure of $\pi^{d}(dz,dt)$. 
The comparison result is obtained by an application of the result in \cite{JK}.
The existence of a viscosity solution and its Feynman-Kac representation is established by first obtaining the existence of smooth solutions and  their Feynman-Kac representations for PIDEs  with smooth coefficients \cite{DGW2020} and then  applying a limit argument involving mollifiers as well as a stochastic fixed point argument based on \cite{BHJ2020,HJKN2020}; see Section \ref{section PIDE}.

Motivated by the Feynman-Kac representation \eqref{FK}, one can then define the following map (from a suitable space of functions)
\begin{equation*}
(\Phi^d \circ v^d)(t,x):=\bE\left[g^d(X^{d,t,x}_{T})\right]
+\int_t^T\bE\left[f^d(s,X^{d,t,x}_{s},v^d(s,X^{d,t,x}_{s}))\right]\,ds
\end{equation*}
for which the solution $u^d$ of the PIDE is a fixed point.
This motivates us to define a sequence of Picard
iterations 
\begin{equation*}
u^d_k(t,x)=(\Phi^d\circ u^d_{k-1})(t,x).
\end{equation*}
The above sequence of Picard iterations is well studied, 
see e.g. \cite{HJKN2020}, for the case without jumps, 
where it was shown that under appropriate assumptions on the coefficients 
of the PDE under consideration,
$\Phi^d$ forms a contraction mapping on
some Banach space and
$\lim_{k\to\infty}u^d_k(t,x)=u^d(t,x)$ for all $(t,x)\in[0,T]\times\R^d$
(see, e.g., Proposition 2.2 in \cite{HJKN2020} for details).  
Furthermore, notice that 
\begin{align}
	u^d_k(t,x)&
	=u^d_1(t,x)+\sum_{l=1}^{k-1}\left[u^d_{l+1}(t,x)-u^d_{l}(t,x)\right]
	\nonumber\\
	&
	=(\Phi^d\circ u^d_0)(t,x)+\sum_{l=1}^{k-1}
	\left[(\Phi^d\circ u^d_l)(t,x)-(\Phi^d\circ u^d_{l-1})(t,x)\right]\nonumber\\
	&
	=\bE\Big[g^d(X^{d,t,x}_{T})\Big]+\sum_{l=0}^{k-1}\int_t^T\bE\Big[
	f^d(s,X^{d,t,x}_{s},u^d_l(s,X^{d,t,x}_{s}))\nonumber\\
	&
	\quad
	-\mathbf{1}_{\N}(l)f^d(s,X^{d,t,x}_{s},u^d_{l-1}(s,X^{d,t,x}_{s}))
	\Big]ds                                     \label{MLP 1}
\end{align}
with $u^d_0\equiv 0$. 
Then, as  suggested by Hutzenthaler et al.\ (see Theorem 1.1 in \cite{HJKN2020}
for PDEs without non-local operators),
one can introduce a multilevel Picard approximation scheme of $u^d(t,x)$
by replacing $(X^{d,t,x}_{s})_{s\in[t,T]}$ in \eqref{MLP 1}
with its (time-) discretized approximation, and by approximating 
the expectations and time integrals in \eqref{MLP 1} by Monte Carlo methods. 
To be able to simulate the jumps of the SDE in \eqref{SDE intro}, 
we propose to use the Monte Carlo approximation method of the compensator integral 
as in \cite[Section~4.4]{GS2021}. 
Since Euler-discretizations and Monte Carlo approximations 
do not suffer from the curse of dimensionality, 
we are then able to show that the corresponding complexity 
of the algorithm only grows polynomially in the dimension $d$ 
and the reciprocal of the prescribed  accuracy~$\varepsilon$, 
extending the results in
 \cite{becker2020numericalMLP,beck2020overcoming,
 hutzenthaler2021multilevel,hutzenthaler2019multilevel,
 HJKN2020,hutzenthalerkruse2020multilevel} to PIDEs.

The remainder of this paper is organized as follows. In Sections 
\ref{section setting} and \ref{section results}, 
we introduce the settings of the MLP approximations,
and formulate the main results of this paper 
(see Theorems \ref{MLP conv} and \ref{MLP complexity}). 
The pseudocode of our MLP 
algorithm is provided in Section~\ref{section code}, whereas the numerical example is presented in Section~\ref{section example}.
In Section~\ref{section FP}, some useful results for stochastic fixed 
point equations are presented.
In Section~\ref{section PIDE},
we investigate the existence and uniqueness of viscosity solutions of semilinear
PIDE \eqref{PDE}, and establish a Feynman-Kac formula for semilinear PIDEs
(see Proposition \ref{proposition PIDE existence}).
Section~\ref{section no Euler} presents an error bound for MLP approximations
not involving Euler scheme which is taken from \cite{HJKN2020}.
The proof of Theorems \ref{MLP conv} 
and \ref{MLP complexity} are given in Section \ref{section MLP}.
Moreover, in Appendix \ref{section pre}
we present some lemmas for elementary dimension-depending estimates 
and properties of SDEs with jumps
and their Euler approximations, which we use in Sections 
\ref{section PIDE} and \ref{section MLP}.

\subsection*{Notation}
In conclusion, we introduce some notations used throughout this paper.
We denote by $\bN$ and $\bN_0$ the set of all positive integers 
and the set of all natural numbers, respectively,
and denote by $\bZ$ the set of all integers.
For each $d\in\bN$, the $d\times d$ identity  matrix is denoted by $\mathbf{I}_d$.
For each $d\in\bN$,
we use $\bS^d$ to denote the space of $d\times d$
symmetric matrices, and for matrices $A,B\in\bS^d$ the notation $A\geq B$
means $A-B$ is positive semi-definite. 
For each $d\in\bN$ and any vectors $a,b\in\bR^d$, 
we denote by $\<a,b\>$ the scalar
product of $a$ and $b$, and denote by $\|a\|$ the euclidean norm of $a$.
For each $d\in\bN$ and every matrix $A\in\bR^{d\times d}$, we
denote by $\|A\|_F$ the Frobenius norm of $A$, 
and we use $A^{ij}$ to denote the element on the $i$-th row and $j$-th column
of $A$ for $i,j=1,...,d$,
For each $d\in\bN$ and every
matrix $A\in\bR^{d\times d}$, we denote by $A^T$ the transpose of $A$. 
For any metric spaces $(E,d_E)$ and $(F,d_F)$,
we denote by $C(E,F)$ the
set of continuous functions from $E$ to $F$. For every topological space $E$,
denote by $\cB(E)$ the Borel $\sigma$-algebra of $E$. 
For all measurable spaces $(X_1,\Sigma_1)$ and $(X_2,\Sigma_2)$, we denote by
$\cM(\Sigma_1,\Sigma_2)$ the set of $\Sigma_1/\Sigma_2$-measurable functions
from $X_1$ to $X_2$.
For all $a,b\in\bR$, we
use the notations $a\wedge b:=\min\{a,b\}$ and $a\vee b:=\max\{a,b\}$.
For any set $B$, 
we use $\mathbf{1}_B$ to denote the indicator function of $B$.

Let $\mathbb{T}$ denote either $(0,T)$ or $[0,T]$,
and denote by $LSC(\mathbb{T}\times\bR^d)$ 
($USC(\mathbb{T}\times\bR^d)$)
the space of lower (upper) semicontinuous functions 
$u:\mathbb{T}\times\bR^d\to\bR$.
We denote by
$LSC_1(\mathbb{T}\times\bR^d)$ 
($USC_1(\mathbb{T}\times\bR^d)$)
the space of functions $u\in LSC(\mathbb{T}\times\bR^d)$ 
$(USC(\mathbb{T}\times\bR^d))$
satisfying the linear growth condition  
\begin{equation}                                               \label{p growth}
\sup_{(t,x)\in\mathbb{T}\times\bR^d}\frac{|u(t,x)|}{1+\|x\|}<\infty,
\end{equation}
and we denote by $C_1(\mathbb{T}\times\bR^d)$ the space of continuous functions
$u:\mathbb{T}\times\bR^d\to\bR$ satisfying \eqref{p growth}.
Moreover, 
$SC(\mathbb{T}\times\bR^d):=LSC(\mathbb{T}\times\bR^d)
\cup USC(\mathbb{T}\times\bR^d)$,
and $SC_1(\mathbb{T}\times\bR^d):=LSC_1(\mathbb{T}\times\bR^d)\cup USC_1(\mathbb{T}\times\bR^d)$.
We use the notation $C^{1,2}(\mathbb{T}\times\bR^d)$ to denote the space of once
in time $t\in\mathbb{T}$ and twice in space $x\in\bR^d$ continuously differentiable
functions $u:\mathbb{T}\times\bR^d\to\bR$, and we also use the notation 
$C^{1,2}_1(\mathbb{T}\times\bR^d)
:=C^{1,2}(\mathbb{T}\times\bR^d)\cap C_1(\mathbb{T}\times\bR^d)$.
Moreover,  
we denote by $C^2_1(\bR^d)$ the space of twice continuously 
differentiable functions $u:\bR^d\to\bR$ satisfying \eqref{p growth}
(without dependence on $t$).
The notation $C^\infty_c(\bR^d)$ means the space of 
infinitely differentiable real-valued
functions on $\bR^d$ with compact support.

\section{\textbf{Multilevel Picard approximation scheme: 
setting and main results}}           
\label{section Picard}

\subsection{Setting}                                      \label{section setting}
Let $T>0$ be a fixed constant, and let $(\Omega,\cF,P)$ be a complete probability space
equipped with a filtration $\bF:=(\cF_t)_{t\in[0,T]}$ 
satisfying the usual conditions.
For each $d\in \bN$ we are given an $\bR^d$-valued standard $\bF$-Brownian motion
denoted by $(W^d_t)_{t\in[0,T]}$.
For each $d\in\bN$ let $\nu^{d}(dz)$ be a L\'evy measure on $\bR^{d}$,
and let $\pi^{d}(dz,dt)$ be a $\bF$-Poisson random measure on
$([0,T]\times\bR^{d},\cB([0,T])\otimes\cB(\bR^{d}))$ 
with intensity measure $\nu^{d}(dz)\otimes dt$,
and we denote by 
$\tilde{\pi}^{d}(dz,dt):=\pi^{d}(dz,dt)-\nu^{d}(dz)\otimes dt$
the compensated Poisson random measure of $\pi^{d}(dz,dt)$.
Moreover, for each $d\in\bN$ let
$\mu^d=(\mu^{d,1},...,\mu^{d,d})
\in C([0,T]\times\bR^d,\bR^d)$
$,
\sigma^{d}=(\sigma^{d,ij})_{i,j\in\{1,2,...,d\}}
\in C([0,T]\times\bR^d,\bR^{d\times d})
$, and
$\eta^{d}=(\eta^{d,1},...,\eta^{d,d})\in C([0,T]\times\bR^d\times \bR^{d},\bR^d)$.
Then for each $d\in\bN$ and $(t,x)\in[0,T]\times\bR^d$ 
consider the following stochastic differential equation (SDE) on $[t,T]$
\begin{equation}                                               \label{SDE}                              
dX^{d,t,x}_{s}
=\mu^d(s,X^{d,t,x}_{s-})\,ds
+\sigma^{d}(s,X^{d,t,x}_{s-})\,dW^{d}_s
+\int_{\bR^{d}}\eta^{d}_s(X^{d,t,x}_{s-},z)
\,\tilde{\pi}^{d}(dz,ds)
\end{equation}
with initial condition $X^{d,t,x}_{t}=x$.
Let $g^d\in C(\bR^d,\bR)$ and $f^d\in C([0,T]\times \bR^d \times\bR,\bR)$.
Then we make the following assumptions for the coefficient functions 
and L\'evy measures.

\subsubsection*{Regularity and growth conditions}

\begin{assumption}[Lipschitz and linear growth conditions]
\label{assumption Lip and growth}
There exists constants $L,p\in(0,\infty)$ satisfying 
for all $d\in\bN$, $x,y\in\bR^d$, $v,w\in\bR$, and $t\in[0,T]$ that
\begin{equation}                                    \label{assumption Lip f g}
T|f^d(t,x,v)-f^d(t,y,w)|+|g^d(x)-g^d(y)|
\leq L^{1/2}\big(T|v-w|+T^{-1/2}\|x-y\|\big),
\end{equation}
\begin{equation}                              \label{assumption Lip mu sigma eta}
\|\mu^d(t,x)-\mu^d(t,y)\|^2
+\|\sigma^{d}(t,x)-\sigma^{d}(t,y)\|_F^2
+\int_{\bR^{d}}\big\|\eta^{d}_t(x,z)-\eta^{d}_t(y,z)\big\|^2
\,\nu^{d}(dz)
\leq L\|x-y\|^2,
\end{equation}
\begin{equation}                                        \label{assumption growth}
\|\mu^d(t,x)\|^2+\|\sigma^{d}(t,x)\|_F^2
+\int_{\bR^{d}}\big\|\eta^{d}_t(x,z)\big\|^2\,\nu^{d}(dz)\leq Ld^p(1+\|x\|^2),
\end{equation}
and
\begin{equation}                                  \label{assumption growth f g}
|T\cdot f^d(t,x,0)|^2+|g^d(x)|^2\leq L(d^p+\|x\|^2).
\end{equation}
\end{assumption}
 
\begin{assumption}[Pointwise Lipschitz and integrability conditions]
\label{assumption pointwise}
For each $d\in \bN$ there exits a constant $C_{d}\in(0,\infty)$ such that
for all $x,x'\in \bR^d$, $t\in[0,T]$, and $z\in\bR^{d}$ 
\begin{equation}                                        \label{assumption eta}                                     
\|\eta^{d}_t(x,z)\|^2\leq C_{d}(1\wedge \|z\|^2), 
\quad \|\eta^{d}_t(x,z)-\eta^{d}_t(x',z)\|^2
\leq C_{d}\|x-x'\|^2(1\wedge \|z\|^2).
\end{equation} 
\end{assumption}

\begin{assumption}                                     \label{assumption jacobian}
For all $d\in\bN$, $(t,x)\in[0,T]\times\bR^d$, and $z\in\bR^d$ we assume that
$D_x \eta^d_t(x,z)$ exists,
where $D_x\eta^d_t(x,z)$ denotes the Jacobian matrix of $\eta^d_t(x,z)$
with respect to $x$.
Moreover, for each $d\in\bN$ assume that there exists a constant
$\lambda_d>0$ such that for all 
$(t,x)\in[0,T]\times\bR^d$, $z\in\bR^d$,
and $\gamma\in[0,1]$  
\begin{equation}                                     \label{jacobian cond}
\lambda_d\leq |\det(\mathbf{I}_d+\gamma D_x \eta^d_t(x,z))|.
\end{equation}
\end{assumption}

\begin{remark}
Note that Assumption \ref{assumption jacobian} is needed to 
show the existence and
uniqueness of linear PIDEs whose coefficient functions
are infinitely differentiable and compactly supported 
as functions of $x\in\bR^d$
(see Proposition \ref{proposition linear PIDE}). 
Moreover, note that Assumption \ref{assumption jacobian} is satisfied
in the L\'evy case where $\eta^d_t(x,z)=z$ or 
$\eta^d_t(x,z)=z\mathbf{1}_{\{\|z\|\leq 1\}}$ 
for all $d\in\bN$, $(t,x)\in[0,T]\times\bR^d$
and $z\in\bR^d$.
\end{remark}

\begin{assumption}
\label{assumption small jumps}
There exist constants $K,q\in(0,\infty)$ satisfying
for all $d\in \bN$, $\delta\in(0,1)$, and $(t,x)\in[0,T]\times\bR^d$ that
\begin{equation}                                       \label{int z & q bound}
\int_{\bR^{d}}(1\wedge\|z\|^2)\,\nu^{d}(dz)\leq Kd^p
\quad \text{and} \quad
\int_{\|z\|\leq \delta}\|\eta^{d}_t(x,z)\|^2
\,\nu^{d}(dz)\leq Kd^p\delta^q(1+\|x\|^2).
\end{equation}
\end{assumption}
Note that Assumptions \ref{assumption Lip and growth} and 
\ref{assumption pointwise} guarantee that for each $d\in\bN$,
the SDE in \eqref{SDE} has an unique solution satisfying that
\begin{equation}                                          \label{SDE moment est}
\bE\Big[\sup_{s\in[t,T]}\big\|X^{d,t,x}_s\big\|^2\Big]<C(1+\|x\|^2),
\end{equation}
where $C$ is a constant only depending on $L$, $d$, $p$, and $T$
(see, e.g., Theorem 3.1 and Theorem 3.2 in \cite{Kunita}, or 
Lemma 114 and Theorem 117 in \cite{Situ}).
The next assumption is needed to obtain the rate of convergence of 
the Euler approximation of SDE \eqref{SDE}, 
in the case that the coefficients
$\mu^d, \sigma^{d}$ and $\eta^{d}$ depend on time $t$
(see Lemma \ref{Lemma Euler} in Section \ref{section pre}).

\begin{assumption}[Temporal $1/2$-H\"older condition]
\label{assumption time Holder}
There exist constants $L_1,L_2>0$ satisfying 
for all $d\in \bN$,
$x\in\bR^d$, $z\in\bR^{d}$, and $(s,s')\in[0,T]^2$ that 
\begin{equation}                                        \label{Holder mu sigma d}
\|\mu^d(s,x)-\mu^d(s',x)\|^2+\|\sigma^{d}(s,x)-\sigma^{d}(s',x)\|_F^2
\leq L_1|s-s'|,
\end{equation}
and
\begin{equation}                                           \label{Holder eta d}
\int_{\bR^{d}}\|\eta^{d}_s(x,z)-\eta^{d}_{s'}(x,z)\|^2\,\nu^{d}(dz)
\leq L_2|s-s'|.
\end{equation}
\end{assumption}
For each $d\in\bN$
let 
$$
F^d: \cM(\cB([0,T]\times \bR^d),\cB(\bR))\to \cM(\cB([0,T]\times \bR^d),\cB(\bR))
$$
be the operator such that
$$
[0,T]\times\bR^d\ni(t,x)\mapsto (F^d(v))(t,x)
:=f^d(t,x,v(t,x)),\quad v\in \cM(\cB([0,T]\times \bR^d),\cB(\bR)).
$$

\subsubsection*{Time Discretization Approximations}
Let $\Theta=\cup_{n\in \bN}\bZ^n$ be an index set which we will use for 
the families of independent random variables needed for the Monte Carlo
approximations. For each $d\in\bN$, let 
$W^{d,\theta}=(W^{d,\theta,1},...,W^{d,\theta,d})
:[0,T]\times \Omega\to \bR^d, \theta\in\Theta$, be independent 
$\bR^d$-valued standard $\bF$-Brownian motions. 
For each $d\in\bN$, let $\pi^{d,\theta}(dz,ds),\theta \in \Theta$, be
independent $\bF$-Poisson random measures on $\bR^{d}\times [0,T]$ 
with identical intensity measure $\nu^{d}(dz)\otimes dt$, and denote by
$
\tilde{\pi}^{d,\theta}(dz,ds)
:=\pi^{d,\theta}(dz,ds)-\nu^{d}(dz)\otimes dt
$ 
the compensated Poisson random measures of $\pi^{d,\theta}(dz,ds)$, respectively.
For each $d\in\bN$ and $(t,x)\in[0,T]\times\bR^d$, 
let $\big(X^{d,\theta,t,x}_s\big)_{s\in[t,T]}: 
[t,T]\times\Omega\to \bR^d$, $\theta\in\Theta$, be 
$\cB([0,T])\otimes\cF/\cB(\bR^d)$-measurable 
functions satisfying for all $\theta\in\Theta$ and $s\in[t,T]$ 
almost surely that $X^{d,\theta,t,x}_{t}=x$ and
\begin{equation}                                               \label{SDE d}                              
dX^{d,\theta,t,x}_{s}
=\mu^d(s,X^{d,\theta,t,x}_{s-})\,ds
+\sigma^{d}(s,X^{d,\theta,t,x}_{s-})\,dW^{d,\theta}_s
+\int_{\bR^{d}}\eta^{d}_s(X^{d,\theta,t,x}_{s-},z)
\,\tilde{\pi}^{d,\theta}(dz,ds).
\end{equation}
For each $d\in\bN$, $N\in\bN$, and $(t,x)\in[0,T]\times\bR^d$, 
let $\big(Y^{d,\theta,t,x,N}_s\big)_{s\in[t,T]}: 
[t,T]\times\Omega\to \bR^d$, $\theta\in\Theta$, be measurable 
functions satisfying for all 
$n\in \bN_0$, $s\in \Big[t+\frac{n(T-t)}{N},t+\frac{(n+1)(T-t)}{N}\Big]
\cap [t,T]$ that
$Y^{d,\theta,t,x,N}_{t}=x$ and
\begin{align}
Y^{d,\theta,t,x,N}_{s}= & Y^{d,\theta,t,x,N}_{t+\frac{n(T-t)}{N}}
+\mu^d\Big(t+\frac{n(T-t)}{N},Y^{d,\theta,t,x,N}_{t+\frac{n(T-t)}{N}}\Big)
\left[s-(t+\frac{n(T-t)}{N})\right]
\nonumber\\
&
+\sigma^{d}\Big(t+\frac{n(T-t)}{N},Y^{d,\theta,t,x,N}_{t+\frac{n(T-t)}{N}}\Big)
\left(W^{d,\theta}_s-W^{d,\theta}_{t+\frac{n(T-t)}{N}}\right)\nonumber\\
&
+\int_{t+\frac{n(T-t)}{N}}^s\int_{\bR^{d}}
\eta^{d}_{t+\frac{n(T-t)}{N}}\Big(Y^{d,\theta,t,x,N}_{t+\frac{n(T-t)}{N}},z\Big)
\,\tilde{\pi}^{d,\theta}(dz,dr).                    \label{discrete Euler}
\end{align}
To ease notations we define 
$$
\kappa_N(s):=t+\frac{\lfloor N(s-t)/(T-t)\rfloor \cdot(T-t)}{N},\quad s\in[t,T],
$$
where $\lfloor y\rfloor:=\max\{n\in\bN_0:n\leq y\}$ for $y\in[0,\infty)$.
Then for all $d,N\in\bN$, $\theta\in\Theta$, and $(t,x)\in[0,T]\times\bR^d$, 
\eqref{discrete Euler} can be written as
\begin{align}
dY^{d,\theta,t,x,N}_{s}=
&
\mu^d\Big({\kappa_N(s-),Y^{d,\theta,t,x,N}_{\kappa_N(s-)}}\Big)\,ds
+\sigma^{d}\Big({\kappa_N(s-),Y^{d,\theta,t,x,N}_{\kappa_N(s-)}}\Big)
\,dW^{d,\theta}_s
\nonumber\\
&
+\int_{\bR^{d}}\eta^{d}_{\kappa_N(s-)}\Big(Y^{d,\theta,t,x,N}_{\kappa_N(s-)},z\Big)
\,\tilde{\pi}^{d,\theta}(dz,ds), \quad s\in[t,T].
\label{Euler 1}
\end{align}

Although \eqref{Euler 1} provides a first approximation to \eqref{SDE d},
one is not able to directly simulate the items of the integrals with respect to 
the compensated Poisson random measures. To find an 
appropriate approximation of \eqref{SDE d}
that allows us to simulate it explicitly, we follow the idea of 
L.Gonon and C. Schwab (see, Sections 4.2 and 4.4 in \cite{GS2021}),
and introduce the following approximation procedures.
For each $d\in\bN$ and $\delta\in(0,1)$ we define the set $A^{d}_\delta$
by
$$
A^{d}_\delta:=\{z\in\bR^{d}:\|z\|\geq \delta\},
$$
and also define a probability measure $\nu_{\delta}^{d}$ 
on $(\bR^{d},\cB(\bR^{d}))$ by
$$
\nu_{\delta}^{d}(B)
:=\frac{\nu^{d}(B\cap A_\delta^{d})}{\nu^{d}(A_\delta^{d})},
\quad B\in\cB(\bR^{d}).
$$
Then, for each $d,N\in\bN$, $(t,x)\in[0,T]\times\bR^d$,
and $\delta\in(0,1)$, let
$\big(Y^{d,\theta,t,x,N,\delta}_{s}\big)_{s\in[t,T]}: 
[t,T]\times\Omega\to \bR^d$, $\theta\in\Theta$,  
be measurable functions satisfying that $Y^{d,\theta,t,x,N,\delta}_{t}=x$ 
and almost surely for all $s\in[t,T]$ 
\begin{align}
dY^{d,\theta,t,x,N,\delta}_{s}=
&
\mu^d\Big({\kappa_N(s-),Y^{d,\theta,t,x,N,\delta}_{\kappa_N(s-)}}\Big)\,ds
+\sigma^{d}\Big({\kappa_N(s-),Y^{d,\theta,t,x,N,\delta}_{\kappa_N(s-)}}\Big)
\,dW^{d,\theta}_s\nonumber\\
&
+\int_{A_\delta^{d}}
\eta^{d}_{\kappa_N(s-)}\Big(Y^{d,\theta,t,x,N,\delta}_{\kappa_N(s-)},z\Big)
\,\pi^{d,\theta}(dz,ds)\nonumber\\
&
-\int_{A_\delta^{d}}
\eta^{d}_{\kappa_N(s-)}\Big(Y^{d,\theta,t,x,N,\delta}_{\kappa_N(s-)},z\Big)
\,\nu^{d}(dz)\,ds.                          \label{def Y N delta}
\end{align}
Furthermore,
for each $d,N,\cM\in\bN$ and $\delta\in(0,1)$ 
let $V^{d,\theta,N,\delta,\cM}_{i,j}$,
$\theta\in\Theta$, $i=0,1,2,...,N$, $j=1,2...,\cM$, 
be independent $\bR^{d}$-valued random variables with
identical distribution $\nu_\delta^{d}$, independent of 
$(W^{d,\theta},\pi^{d,\theta})_{(\theta,d)\in \Theta\times\bN}$.
For each $d\in\bN$, $(t,x)\in[0,T]\times\bR^d$,
$N,\cM\in\bN$, and $\delta\in(0,1)$, let
$\big(Y^{d,\theta,t,x,N,\delta,\cM}_{s}\big)_{s\in[t,T]}: 
[t,T]\times\Omega\to \bR^d$, $\theta\in\Theta$,  
be measurable functions satisfying that $Y^{d,\theta,t,x,N,\delta,\cM}_{t}=x$ 
and almost surely for all $s\in[t,T]$
\begin{align*}
dY^{d,\theta,t,x,N,\delta,\cM}_{s}=
&
\mu^d\Big({\kappa_N(s-),Y^{d,\theta,t,x,N,\delta,\cM}_{\kappa_N(s-)}}\Big)\,ds
+\sigma^{d}\Big({\kappa_N(s-),Y^{d,\theta,t,x,N,\delta,\cM}_{\kappa_N(s-)}}\Big)
\,dW^{d,\theta}_s\\
&
+\int_{A_\delta^{d}}
\eta^{d}_{\kappa_N(s-)}\Big(Y^{d,\theta,t,x,N,\delta,\cM}_{\kappa_N(s-)},z\Big)
\,\pi^{d,\theta}(dz,ds)\nonumber\\
&
-\left(\frac{\nu^{d}(A_\delta^{d})}{\cM}\sum_{j=1}^{\cM}
\eta^{d}_{\kappa_N(s-)}\Big(Y^{d,\theta,t,x,N,\delta,\cM}_{\kappa_N(s-)},
V^{d,\theta,N,\delta,\cM}_{\frac{(\kappa_N(s-)-t)N}{T-t},j}\Big)
\right)ds.
\end{align*}

\subsubsection*{Multilevel Picard Approximation Scheme}

Next, we introduce the multi-level Picard approximations.
Let $\xi^\theta:\Omega\to [0,1], \theta\in\Theta$, be i.i.d. uniformly distributed 
random variables, and assume that $(\xi^\theta)_{\theta\in\Theta}$ is independent of 
$(W^{d,\theta},\pi^{d,\theta})_{(\theta,d)\in \Theta\times\bN}$.
Let $\cR^\theta:[0,T]\times \Omega\to [0,T]$, $\theta\in \Theta$,
satisfy that $\cR^\theta_t=t+(T-t)\xi^\theta$ for all $t\in[0,T]$ and 
$\theta\in\Theta$.
For each $d\in\bN$, $M,N\in \bN$, $\theta\in\Theta$, 
$(t,x)\in[0,T]\times\bR^d$, and $s\in[t,T]$, let 
$\big(Y^{(d,\theta,t,x,N,\delta,\cM,l,i)}_{s}\big)_{(l,i)\in\bN\times(\bZ/\{0\})}$
and $\big(\cR^{(\theta,l,i)}_t\big)_{(l,i)\in\bN\times\bN_0}$ be independent
copies of $Y^{d,\theta,t,x,N,\delta,\cM}_{s}$ and 
$\cR^{\theta}_t$, respectively.

For each $d\in\bN$, $n\in\bN_0$, $M,N,\cM\in \bN$,
$\theta\in\Theta$, $\delta\in(0,1)$, 
let 
$U^{d,\theta,\delta,\cM}_{n,M,N}:[0,T]\times \bR^d \times \Omega \to \bR$
be measurable functions satisfying for all 
$(t,x)\in\bR^d$ that
$$
U^{d,\theta,\delta,\cM}_{n,M,N}(t,x)
=\frac{\mathbf{1}_{\bN}(n)}{M^n}
\sum_{i=1}^{M^n}g^d\Big(Y^{(d,\theta,t,x,N,\delta,\cM,0,-i)}_{T}\Big)
$$
\begin{equation}                                              \label{MLP def}
+\sum_{l=0}^{n-1}\frac{(T-t)}{M^{n-l}}
\left[\sum^{M^{n-l}}_{i=1}\Big(F^d\Big(U^{(d,\theta,\delta,\cM,l,i)}_{l,M,N}\Big)
-\mathbf{1}_{\bN}(l)F^d\Big(U^{(d,\theta,\delta,\cM,-l,i)}_{l-1,M,N}\Big)\Big)
\Big(\cR^{(\theta,l,i)}_t,
Y^{(d,\theta,t,x,N,\delta,\cM,l,i)}_{\cR^{(\theta,l,i)}_t}\Big)
\right],
\end{equation}
where $\Big(U^{(d,\theta,\delta,\cM,l,i)}_{n,M,N}\Big)_{(l,i)\in\bZ\times\bN_0}$ 
are independent copies of $U^{d,\theta,\delta,\cM}_{n,M,N}$.

\subsubsection*{Viscosity solutions of PIDEs}
For every $d\in\bN$, let
$
G^d:(0,T)\times \bR^d\times\bR\times\bR^d\times\bS^d\times C^2(\bR^d)\to \bR
$
be a functional defined for all
$(t,x,r,y,A,\phi)
\in(0,T)\times\bR^d\times\bR\times\bR^d\times\bS^d\times C^2(\bR^d)$ by 
\begin{align}
G^d(t,x,r,y,A,\phi):=
&
-\langle y,\mu^d(t,x)\rangle
-\frac{1}{2}\operatorname{Trace}
(\sigma^d(t,x)[\sigma^d(t,x)]^TA)
-f^d(t,x,r)\nonumber\\
& 
-\int_{\bR^{d}}\left(\phi(x+\eta^d_t(x,z))-\phi(x)
-\langle\nabla_x \phi(x),\eta^d_t(x,z)\rangle\right)\,\nu^d(dz).
\label{def operator G}
\end{align}
Then for every $d\in\bN$
 we consider the semilinear PIDE of parabolic type
\begin{align}
&
-\frac{\partial}{\partial t}u^d(t,x)
+G^d(t,x,u^d(t,x),\nabla_x u^d(t,x)
,\operatorname{Hess}_xu^d(t,x),u^d(t,\cdot))=0
\quad \text{on $(0,T)\times \bR^d$},
\label{APIDE}
\\
&
u^d(T,x)=g^d(x) \quad \text{on $\bR^d$}.
\label{APIDE initial}
\end{align}
We use the following definition of viscosity solutions
of PIDE \eqref{APIDE} (cf.\,Definition 2.1 in \cite{JK}).

\begin{definition}
Let $d\in\bN$.
A function $u^d\in USC_1((0,T)\times\bR^d)$ 
($u^d\in LSC_1((0,T)\times\bR^d)$)
is called a viscosity subsolution (supersolution) of PIDE \eqref{APIDE}
if for every $(t,x)\in(0,T)\times\bR^d$ 
and $\varphi\in C^{1,2}((0,T)\times\bR^d)$ such that 
$\varphi(t,x)=u^d(t,x)$, 
and $u^d\leq \varphi$ ($u^d\geq \varphi$), we have that
$$
-\frac{\partial}{\partial t}\varphi(t,x)
+G^d(t,x,\varphi(t,x),\nabla_x \varphi(t,x),
\operatorname{Hess}_x\varphi(t,x),\varphi(t,\cdot))
\leq 0\; (\geq 0).
$$
A function $u^d:(0,T)\times\bR^d\to\bR$ 
is said to be a viscosity solution of PIDE \eqref{APIDE} if $u$ is both a 
viscosity subsolution and a viscosity supersolution of \eqref{APIDE}.
\end{definition}

\begin{remark}                                       
\label{remark v solution}
Let $d\in\bN$.
If $u^d\in C_1([0,T]\times\bR^d)$ is a viscosity solution of
\eqref{APIDE} with $u^d(T,x)=g^d(x)$ for $x\in\bR^d$, 
then the function $v^d(t,x):=u^d(T-t,x)$, $(t,x)\in[0,T]\times\bR^d$,
satisfies the follows.
\begin{enumerate}[(i)]
\item{
$v^d(0,x)=g^d(x)$ for $x\in\bR^d$.
}
\item{
For every $(t,x)\in(0,T)\times\bR^d$ and
$\varphi\in C^{1,2}((0,T)\times\bR^d)$ such that $\varphi(t,x)=v^d(t,x)$ 
and $v^d\leq \varphi$ ($v^d\geq \varphi$), we have
$$
\frac{\partial}{\partial t}\varphi(t,x)
+G^d(t,x,\varphi(t,x),\nabla_x \varphi(t,x),\operatorname{Hess}_x\varphi(t,x),
\varphi(t,\cdot))
\leq 0\; (\geq 0).
$$
}
\end{enumerate}
The converse holds as well.
\end{remark}
\subsection{Main results}                              \label{section results}
The following theorem shows that there exists a unique viscosity solution of
PIDE \eqref{APIDE}, and that this solution satisfies a Feynman-Kac representation.
This allows to prove that our MLP approximation scheme converges to the
unique viscosity solution of PIDE \eqref{APIDE}.
\begin{theorem}                                                   \label{MLP conv}
Assume Setting \ref{section setting}.
Then the following holds.
\begin{enumerate}[(i)]
\item{
For each $d\in\bN$, $(t,x)\in[0,T]\times\bR^d$,
and $\theta\in\Theta$ there exists an unique $\bF$-adapted 
$\bR^d$-valued c\`adl\`ag process
$(X^{d,\theta,t,x}_{s})_{s\in[t,T]}$ such that \eqref{SDE d} holds.
}

\item{
For each $d\in\bN$ there exists an unique
Borel function $u^d:[0,T]\times\bR^d\to\bR$ satisfying for all
$(t,x)\in[0,T]\times\bR^d$ that

$$
\bE\Big[\big|g^d(X^{d,0,t,x}_{T})\big|\Big]
+\int_t^T\bE\Big[\big|f^d(s,X^{d,0,t,x}_{s},u^d(s,X^{d,0,t,x}_{s}))\big|\Big]\,ds
+\sup_{y\in\bR^d}\sup_{s\in[0,T]}\left(\frac{|u^d(s,y)|}{\sqrt{d^p+\|y\|^2}}
\right)<\infty,
$$
and
$$
u^d(t,x)=\bE\Big[g^d(X^{d,0,t,x}_{T})\Big]
+\int_t^T\bE\Big[f^d(s,X^{d,0,t,x}_{s},u^d(s,X^{d,0,t,x}_{s}))\Big]\,ds.
$$
}

\item{
For each $d\in\bN$ there exists a unique viscosity solution
$v^d\in C_1([0,T]\times\bR^d)$ of PIDE \eqref{APIDE} with
$v^d(T,x)=g^d(x)$ for all $x\in\bR^d$.
}

\item{
It holds for all $d\in\bN$ and $(t,x)\in[0,T]\times\bR^d$ that
$v^d(t,x)=u^d(t,x)$.
}

\item{
There exists a positive constant $c=c(T,L,L_1,L_2,K)$ satisfying
for all $d\in\bN$,
$(t,x)\in[0,T]\times\bR^d$, $\delta\in(0,1)$, $n\in\bN_0$, 
and $M,N,\cM\in\bN$ with $\cM\geq \delta^{-2}Kd^p$ that
\begin{align}
&
\left(\bE\left[\big|U^{d,0,\delta,\cM}_{n,M,N}(t,x)-u^d(t,x)\big|^2\right]\right)^{1/2}
\nonumber\\
&
\leq
c(d^p+\|x\|^2)\left[e^{M/2}M^{-n/2}e^{2nTL^{1/2}}
+(N^{-1}T+\delta^qd^p+\delta^{-2}Kd^p\cM^{-1})^{1/2}\right].
\end{align}
}
\end{enumerate}
\end{theorem}
The next theorem shows that the MLP approximation scheme 
defined by \eqref{MLP def} 
can overcome the curse of dimensionality. 
To describe the computational complexity, 
for each $d\in\bN$,
$n\in\bN_0$, 
and $M\in\bN$ 
we introduce a natural number $\mathfrak{C}^{(d)}_{n,M}$ to denote the sum 
of the number of function evaluations of $g^d$,
of the number of function evaluations of $\mu^d$,
of the number of function evaluations of $\sigma^{d}$,
of the number of function evaluations of $\eta^{d}$,
and of the number of realizations of scalar random variables used to obtain
one realization of the MLP approximation algorithm in \eqref{MLP def}. 
Moreover, for each $d\in\bN$ we use
$\mathfrak{g}^{(d)}$ to denote the number of function evaluations of $g^d$, 
$\mathfrak{f}^{(d)}$ to denote the number of function evaluations of $f^d$, 
$\mathfrak{e}^{(d)}$ to denote the sum 
of the number of realizations of scalar random variables generated,
of the number of function evaluations of $\mu^d$,
of the number of function evaluations of $\sigma^{d}$,
and of the number of function evaluations of $\eta^{d}$.

\begin{theorem}                                            \label{MLP complexity}
Assume the settings in Section \ref{section setting}. 
Moreover, for each $d\in\bN$, let 
$\mathfrak{e}^{(d)}$, $\mathfrak{g}^{(d)}$, $\mathfrak{f}^{(d)}$,
$n,M\in\bN$, $\mathfrak{C}^{(d)}_{n,M}\in\bN$,
satisfy for all $n,M\in\bN$ that
\begin{equation}                                            \label{cc 1}
\mathfrak{C}^{(d)}_{n,M}\leq M^n(M^M\mathfrak{e}^{(d)}+\mathfrak{g}^{(d)})\mathbf{1}_{\bN}(n)
+\sum_{l=0}^{n-1}[M^{n-l}(M^M\mathfrak{e}^{(d)}+\mathfrak{f}^{(d)}+\mathfrak{C}^{(d)}_{l,M}
+\mathfrak{C}^{(d)}_{l-1,M})].
\end{equation}
Then, for all $d\in\bN$ and $n\in\bN$ we have
\begin{equation}                                               \label{cc 2}
\sum_{k=1}^{n+1}\mathfrak{C}^{(d)}_{k,k}
\leq 12(3\mathfrak{e}^{(d)}+\mathfrak{g}^{(d)}+2\mathfrak{f}^{(d)})36^nn^{2n}.
\end{equation}
Moreover, for each $d\in\bN$, $\varepsilon\in(0,1]$, and $x\in\bR^d$
there exists a positive integer $\mathbf{n}^d(x,\varepsilon)\geq 2$ 
such that for all 
$\gamma\in(0,1]$ and $t\in[0,T]$,
\begin{equation}                                            \label{error epsilon}
\sup_{n\in[\mathbf{n}^d(x,\varepsilon),\infty)\cap \bN}
\left(\bE\left[\big|U^{d,0,n^{-n/q},n^{n+2n/q}Kd^p}_{n,n,n^n}(t,x)
-u^d(t,x)\big|^2\right]
\right)^{1/2}<\varepsilon,
\end{equation}
and
\begin{align}                                             \label{cc 3}
\left(\sum_{n=1}^{\mathbf{n}^d(x,\varepsilon)}\mathfrak{C}^{(d)}_{n,n}\right)
\varepsilon^{\gamma+4} 
\leq & 12(3\mathfrak{e}^{(d)}+\mathfrak{g}^{(d)}+2\mathfrak{f}^{(d)})
\left[c(d^p+\|x\|^2)\right]^{\gamma+4}
\nonumber\\
&
\cdot \sup_{n\in\bN}
\left(36^n
n^{-\gamma n/2}(e^{n/2}e^{2nTL^{1/2}}+(1+T+d^p)^{1/2})^{\gamma+4}\right)
<\infty,
\end{align}
where $c=c(T,L,L_1,L_2,K)$
is the constant introduced in (v) in Theorem \ref{MLP conv}. 
\end{theorem}

\subsection{Pseudocode}                                      \label{section code}
In this subsection we provide a pseudocode to show how the multilevel Picard
approximations \eqref{MLP def} can be implemented. 
\begin{algorithm}
\scriptsize{
\caption{Multilevel Picard Approximation}
\begin{algorithmic}[1]
\Function{MLP}{$t,x,n,M,N,\delta,\cM$}

\State $t_1(j)\leftarrow (T-t)/N$ for all $j\in\{0,...,N-1\}$;
\Comment the length of time intervals in the time discretization
 
\For{$i\leftarrow 1$ to $M^n$}

    \State $X(i)\leftarrow x$;

    \State \multiline{Generate $N$ realizations $W(j)\in\bR^d$,
    $j\in\{0,...,N-1\}$, of i.i.d.\,standard normal random vectors;}
    \State \multiline{Generate $N$ realizations $P(j)\in\bN_0$,
    $j\in\{0,...,N-1\}$, of i.i.d.\,Poisson random variables with
    parameter $\nu^d(A^d_\delta)(T-t)N^{-1}$;
    (see, e.g., (1.8.3) in \cite{PB2010})}
    \For{$k\leftarrow 0$ to $N-1$}
    \State \multiline{Generate $P(k)$ realizations $Z(j)\in\bR^d$,
     $j\in\{1,...,P(k)\}$,
    of i.i.d.\,random vectors with distribution 
    $\nu^d_\delta(dz)$;}
    \State \multiline{Generate $\cM$ realizations $V(j)\in\bR^d$, 
    $j\in\{1,2,...,\cM\}$, of i.i.d.\,random vectors 
    with distribution $\nu^d_\delta(dz)$;}
    \State \multiline{
    $
    X(i)\leftarrow  X(i)+\mu^d(t+k(T-t)/N,X(i))t_1(k)+\sigma^d(t+k(T-t)/N,X(i))
    \sqrt{t_1(k)}\cdot W(k)
    $\\
    $
    \quad\quad\quad\;\;+\sum_{j=1}^{P(k)}\eta^d_{t+k(T-t)/N}(X(i),Z(j))
    -\frac{t_1(k)\nu^d(A^d_\delta)}{\cM}
    \sum_{j=1}^{\cM}\eta^d_{t+(T-t)k/N}(X(i),V(j));
    $
    }
    \EndFor
    
\EndFor
\If{$n>0$}
       \State $u=\frac{1}{M^n}\sum_{i=1}^{M^n}g(X(i));$
       \Else 
       \State $u=0;$
    \EndIf

\For{$l \leftarrow 0$ to $n-1$}

    \State $Y(i)\leftarrow x$ for all $i\in\{1,...,M^{n-l}\}$;
    \State \multiline{
    Generate $M^{n-l}$ realizations $\cR(i)\in[t,T]$, $i\in{1,...,M^{n-l}}$,
    of i.i.d.\,random variables uniformly distributed on $[t,T]$;
    }
    \For{$i \leftarrow 1$ to $M^{n-l}$}
      \State $S(i)\leftarrow \lfloor (\cR(i)-t)N/(T-t)\rfloor+1$;
      \Comment total time points of the time discretization of SDE
      \State $t_2(j)\leftarrow (T-t)/N$ for all $j\in\{0,...,S(i)-2\}$;
      \Comment the length of the first $S(i)-1$ time intervals
      \State $t_2(S(i)-1)\leftarrow \cR(i)-[t+N^{-1}(T-t)(S_2(i)-1)]$;
      \Comment the length of the last time interval
      \State \multiline{Generate $S(i)$ realizations $W(j)\in\bR^d$,
      $j\in\{0,...,S(i)-1\}$, of independent standard normal random vectors;}
      \State \multiline{Generate $S(i)-1$ realizations $P(j)\in\bN_0$,
      $j\in\{0,...,S(i)-2\}$, of i.i.d.\,Poisson random variables with
      parameter $\nu^d(A^d_\delta)(T-t)N^{-1}$;}
      \State \multiline{Generate one realization $P(S_2(i)-1)\in\bN_0$ 
      of Poisson random variable with parameter $\nu^d(A^d_\delta)t_2(S(i)-1)$;}
      
        \For{$k\leftarrow 0$ to $S(i)-1$}
          \State \multiline{Generate $P(k)$ realizations $Z(j)\in\bR^d$,
          $j\in\{1,...,P(k)\}$,
          of i.i.d.\,random vectors with distribution 
          $\nu^d_\delta(dz)$;}
          \State \multiline{Generate $\cM$ realizations $V(j)\in\bR^d$, 
          $j\in\{1,2,...,\cM\}$, of i.i.d.\,random vectors 
          with distribution $\nu^d_\delta(dz)$;}
          \State \multiline{
          $
          Y(i)\leftarrow  Y(i)+\mu^d(t+k(T-t)/N,Y(i))t_2(k)
          +\sigma^d(t+k(T-t)/N,Y(i))\sqrt{t_2(k)}\cdot W(k)
          $\\
          $
          \quad\quad\quad\;\;+\sum_{j=1}^{P(k)}\eta^d_{t+k(T-t)/N}(Y(i),Z(j))
          -\frac{t_2(k)\nu^d(A^d_\delta)}{\cM}\sum_{j=1}^{\cM}\eta^d_{t+k(T-t)/N}
          (Y(i),V(j));
          $
          }
        \EndFor
    \EndFor
    \State $u\leftarrow u+\frac{T-t}{M^{n-l}}\sum_{i=1}^{M^{n-l}}
    f\big(\cR(i),Y(i),\operatorname{MLP}(\cR(i),Y(i),l,M,N,\delta,\cM)\big)$;
    \If{$l>0$}
    \State $u\leftarrow u-\frac{T-t}{M^{n-l}}\sum_{i=1}^{M^{n-l}}
    f\big(\cR(i),Y(i),\operatorname{MLP}(\cR(i),Y(i),l-1,M,N,\delta,\cM)\big)$;
    \EndIf
\EndFor\\
\quad\:   \Return{$u$;}
\EndFunction
\end{algorithmic}
}
\end{algorithm}
\clearpage

\subsection{A numerical example 
of pricing with counterparty credit risk in Vasi\v{c}ek jump model}                           
\label{section example}
In this subsection, we present a numerical example\footnote{All numerical experiments have been implemented 
in \texttt{Python} on an average laptop 
(Lenovo ThinkPad X13 Gen2a with Processor AMD Ryzen 7 PRO 5850U and Radeon Graphics,
1901 Mhz, 8 Cores, 16 Logical Processors). 
The code can be found under the following link: 
\url{https://github.com/psc25/MLPJumps}} 
to demonstrate the applicability
of the MLP algorithm~\eqref{MLP def}.
We consider the problem of pricing a financial derivative 
with counterparty credit risk, which was already considered in \cite{EHJK2019}, 
see also \cite{BurgardKjaer2011,HernyLabordere2012} for the corresponding PDE. 
For this purpose, we assume that the stock prices 
can be described by the Vasi\v{c}ek jump model (see also \cite{WuLiang2018}), i.e.,
~for every $d\in\bN$ and $(t,x)\in [0,T]\times\bR^d$, the stock price
$(X^{d,0,x}_{t})_{t\in[0,T]}:\Omega\times[0,T]\to\bR^d$ 
follows the SDE
	\begin{equation}
		\label{EqJVasicek}
		\begin{split}
			X^{d,0,x}_{t} 
			& = x + \alpha \int_0^t 
			\left( \mu_0\mathds{1}_d - X^{d,0,x}_{s} \right) \, ds + 
			\sigma_0 W^d_t + \left( \sum_{n=1}^{N_t} Z^d_n - \frac{\lambda t}{2} 
			\mathds{1}_d \right) \\
			& = x + \alpha \int_0^t 
			\left( \mu_0\mathds{1}_d - X^{d,0,x}_{s} \right) \, ds 
			+ \sigma_0 
			\int_0^t \, dW^d_{s} 
			+ \int_0^t\int_{\mathbb{R}^d} z\mathbf{1}_{[0,1]^d}(z) 
			\, \tilde{\pi}^d(dz,ds).
		\end{split}
	\end{equation}
	Hereby, $x \in \mathbb{R}^d$ is the initial value, 
	$\alpha, \mu_0, \sigma_0 > 0$ are parameters,
	$\mathds{1}_d:=(1,1,\dots,1)^T\in \R^d$,
	$(N_t)_{t \in [0,\infty)}$ is a Poisson process with intensity parameter
	$\lambda > 0$, and $(Z^d_n)_{n \in \mathbb{N}_0} \sim \mathcal{U}_d([0,1]^d)$ 
	is a sequence of i.i.d.~multivariate uniformly distributed random variables which are independent of $(N_t)_{t \in [0,\infty)}$.
Moreover, $\tilde{\pi}^d(dz,dt):=\pi^d(dz,dt)-\nu^d(dz)~\otimes~dt$ 
is the compensated Poisson random measure whose corresponding Poisson point
measure is of the form
$$
\pi^d((0,t]\times A):=\#\big\{s\in(0,t]:X^{d,0,x}_s-X^{d,0,x}_{s-}\in A\big\},
\quad t\in(0,T], \;\; A\in\bR^d/\{0\},
$$
with compensator $\nu^d(dz) \otimes dt$ satisfying 
$\nu^d(dz) = \lambda\mathbf{1}_{[0,1]^d}(z) dz$, where 
here $\#B$ denotes the number of elements in some set $B$.
Therefore, for every $d\in\bN$ 
the stochastic process $(X^{d,0,x}_t)_{t \in [0,T]}$ is of the form 
\eqref{SDE} with $\mu^d(s,x) = \alpha (\mu_0\mathds{1}_d - x)$, 
$\sigma^d(s,x) = \sigma_0\mathbf{I}_d$, 
$\eta^d_s(x,z) = z \mathbf{1}_{[0,1]^d}(z)$, 
and $\nu^d(dz) = \lambda \mathbf{1}_{[0,1]^d}(z) dz$, which satisfy 
Assumptions \ref{assumption Lip and growth}, \ref{assumption pointwise},
\ref{assumption jacobian}, \ref{assumption small jumps},
and \ref{assumption time Holder}. 
For the numerical example, 
we choose $\alpha = 0.01$, $\mu_0 = 100$, $\sigma_0 = 2$, $\lambda = 0.5$, 
and $x = (100,...,100)^T \in \mathbb{R}^d$.
	
	Then, by following the derivations in \cite{BurgardKjaer2011,HernyLabordere2012}, 
	we consider 
	the pricing problem which consists of solving PIDE \eqref{PDE} 
	with $f^d(t,x,v) = -\beta \min(v,0)$ 
	and 
	$g^d(x) = \max\left( \min_{i=1,...,d} x_i - K_1, 0 \right) 
	- \max\left( \min_{i=1,...,d} x_i - K_2, 0 \right) - L_0$, 
	where we choose the parameters $\beta = 0.03$, $K_1 = 80$, $K_2 = 100$, 
	and $L_0 = 5$.
	Note that all the assumptions in Setting~\ref{section setting} are satisfied, in particular Theorem~\ref{MLP conv} and Theorem~\ref{MLP complexity} hold.
	
	For different dimensions $d \in \mathbb{N}$ 
	and levels $M = n \in \lbrace 1,...,5 \rbrace$, 
	we run the algorithm $10$ times to approximate the solution $u(T,100,...,100)$ 
	of the PIDE \eqref{PDE}, 
	where we choose $T = 1/2$, $N = 12$, $\mathcal{M} = 200$, and $\delta = 0.1$ 
	such that the condition $\mathcal{M} \geq \delta^{-2} K d^p$ 
	in (v) of Theorem \ref{MLP conv} is satisfied 
	with $K = \lambda = 0.5$ and $p = 0$. 
	The results are reported in Table~\ref{TabCounterpartyJVasicek}.
	
	\begin{table}[h!]
		\begin{tabular}{rl|R{2cm}R{2cm}R{2cm}R{2cm}R{2cm}|}
			& & \multicolumn{5}{c|}{Level} \\
			$d$ & & $M = n = 1$ & $M = n = 2$ & $M = n = 3$ & 
			$M = n = 4$ & $M = n = 5$ \\
			\hline
			\input{table.txt}
		\end{tabular}
		\caption{MLP solution of the pricing problem under counterparty credit risk 
		with the Vasi\v{c}ek jump model \eqref{EqJVasicek}, 
		for different $d \in \mathbb{N}$ and $M = n \in \lbrace 1,...,5 \rbrace$. 
		The average solution (Avg.~Sol.), the standard deviation (Std.~Dev.), 
		the average time (Avg.~Time [in seconds]), 
		and the number of function evaluations 
		$\mathfrak{C}^{(d)}_{n,M}$ (Avg.~Eval., where the evaluation 
		of $\mu^d$, $\sigma^d$, $\eta^d$, $f^d$, and $g^d$, 
		and the generation of a one-dimensional random variable 
		is counted as one unit) are computed over the 10 runs of the algorithm.}
		\label{TabCounterpartyJVasicek}
	\end{table}

\section{\textbf{Stochastic fixed-point equations}}
\label{section FP}
The following results from \cite{HJKN2020} provide existence and uniqueness of
solutions of stochastic fixed point equations. 
These results will play an important role in
the proof of Theorem \ref{MLP conv} in Section \ref{section MLP}.
\begin{proposition}                                        \label{Prop FP}
Let $d\in \bN$, $a,b,c,p,T\in(0,\infty)$, and let
$Z^{t,x}_s:\Omega\to \bR^d$ be a random variable for each $t\in[0,T]$,
$s\in[t,T]$, and $x\in\bR^d$. For every nonnegative Borel function
$\varphi:[0,T]\times\bR^d\to [0,\infty)$, we assume that the mapping
$
\{(t,s)\in[0,T]^2:t\leq s\}\times \bR^d \ni (t,s,x)
\mapsto \bE[\varphi(s,Z^{t,x}_s)]\in[0,\infty] 
$
is measurable. Let $f:[0,T]\times \bR^d \times \bR\to \bR$
and $g:\bR^d\to\bR$ be Borel functions, and assume for all
$t\in[0,T]$, $s\in[t,T]$, $x\in\bR^d$, and $y,y'\in\bR$ that
$|f(t,x,0)|\leq a(d^p+\|x\|^2)^{1/2}$, $|g(x)|\leq a(d^p+\|x\|^2)^{1/2}$,
$\bE[(d^p+\|Z^{t,x}_s\|^2)^{1/2}]\leq b(d^p+\|x\|^2)^{1/2}$
and
$|f(t,x,y)-f(t,x,y')|\leq c|y-y'|$.

Then there exists a unique Borel function 
$u:[0,T]\times \bR^d\to\bR$ satisfying for all $(t,x)\in[0,T]\times\bR^d$ that
$$
\bE\left[\big|g(Z^{t,x}_T)\big|\right]
+\int_t^T\bE\Big[\big|f(s,Z^{t,x}_s,u(s,Z^{t,x}_s))\big|\Big]\,ds
+\sup_{y\in\bR^d}\sup_{s\in[0,T]}\left(\frac{|u(s,y)|}{(d^p+\|y\|^2)^{1/2}}
\right)<\infty
$$
and
$$
u(t,x)=\bE\left[g(Z^{t,x}_T)\right]+\int_t^T
\bE\left[f(s,Z^{t,x}_s,u(s,Z^{t,x}_s))\right]\,ds.
$$
Moreover, it holds for all $t\in[0,T]$ that
$$
\sup_{x\in \bR^d}\left( \frac{|u(t,x)|}{(d^p+\|x\|^2)^{1/2}} \right)
\leq \left[  \sup_{x\in \bR^d} \left( \frac{|g(x)|}{(d^p+\|x\|^2)^{1/2}} \right) 
+\sup_{x\in\bR^d}\sup_{s\in[t,T]}
\left( \frac{T|f(s,x,0)|}{(d^p+\|x\|^2)^{1/2}}\right)\right]e^{c(T-t)}<\infty.
$$
\end{proposition}
\begin{proof}
For $b=1$, 
the above proposition is proved in \cite{HJKN2020} as Proposition 2.2 
(with $X^x_{t,s}\cal Z^{t,x}_s$, $\psi \cal \varphi$,
$\mathcal{O}\cal \bR^d$, $T\cal T$, $c\cal a$, $1\cal b$, $L\cal c$,
$V\cal (\bR^d\ni x\mapsto (d^p+\|x\|^2)^{1/2}\in\bR)$,
$u \cal u$, $f \cal f$, $g \cal g$
in the notation of Proposition 2.2 in \cite{HJKN2020}).
For $b\in(0,\infty)$, one can precisely follow the proof of Proposition 2.2
in \cite{HJKN2020} to obtain the desired results.
\end{proof}

\begin{lemma}                                           \label{lemma perturbation}
Let $d\in\bN$, $c,\rho,\eta,L\in[0,\infty)$, $T,\delta\in(0,\infty)$, 
let $(X^{t,x,k}_{s})_{s\in[t,T]}:[t,T]\times\Omega\to\bR^d$,
$t\in[0,T]$, $x\in\bR^d$, $k\in\{1,2\}$, be 
$\cB([0,T])\otimes\cF/\cB(\bR^d)$-measurable,
let $f:[0,T]\times\bR^d\times\bR\to\bR$, 
$g:\bR^d\to\bR$,
and 
$u_k:[0,T]\times\bR^d\to\bR$, $k\in\{1,2\}$, be Borel functions.
For all $t\in[0,T]$, $s\in[t,T]$, and all Borel functions 
$h:\bR^d\times\bR^d\to[0,\infty)$ assume that
$
\bR^d\times\bR^d\ni(y_1,y_2)\mapsto\bE[h(X^{t,y_1,1}_{s},X^{t,y_2,1}_{s})]
\in[0,\infty]
$ 
is measurable. Moreover, for all $x,y\in\bR^d$, $t\in[0,T]$ and $s\in[t,T]$,
$r\in[s,T]$, $v,w\in\bR$, $k\in\{1,2\}$, and all Borel functions 
$h:\bR^d\times\bR^d\to[0,\infty)$ assume that
$$
X^{t,x,k}_{t}=x,\quad \bE\Big[d^p
+\big\|X^{t,x,k}_{s}\big\|^2\Big]\leq \eta(d^p+\|x\|^2),
$$
$$
\bE\left[\bE\Big[h(X^{s,x',1}_{r},X^{s,y',1}_{r})\Big]
\Big|_{(x',y')=(X^{t,x,1}_{s},X^{t,y,1}_{s})}\right]
=\bE\Big[h(X^{t,x,1}_{r},X^{t,y,1}_{r})\Big],
$$
$$
\max\{T|f(t,x,v)-f(t,y,w)|,|g(x)-g(y)|\}
\leq L\Big(T|v-w|+T^{-1/2}(2d^p+\|x\|^2+\|y\|^2)^{1/2}\|x-y\|\Big),
$$
$$
\bE\left[\bE\left[\big\|X^{s,x',1}_{r}-X^{s,y',1}_{r}\big\|^2\right]
\Big|_{(x',y')=(X^{t,x,1}_{s},X^{t,x,2}_{s})}\right]\leq\delta^2(d^p+\|x\|^2),
$$
$$
\bE\left[\big|g(X^{t,x,k}_{T})\big|\right]
+\int_t^T\bE\left[\big|f(s,X^{t,x,k}_{s},u_k(s,X^{t,x,k}_{s}))\big|
\right]\,ds<\infty,
$$
$$
u_k(t,x)=\bE\left[g(X^{t,x,k}_{T})
+\int_t^Tf(s,X^{t,x,k}_{s},u_k(s,X^{t,x,k}_{s}))\,ds\right],
$$
and
\begin{equation}                                      \label{cond max}
\max\left\{|u_k(t,x)|^2,
e^{\rho(t-s)}\bE\Big[d^p+\big\|X^{t,x,k}_{s}\big\|^2\Big]\right\}
\leq c(d^p+\|x\|^2).
\end{equation}
Then for all $(t,x)\in[0,T]\times\bR^d$ it holds that
\begin{equation}                                     \label{u difference}
|u_1(t,x)-u_2(t,x)|\leq 4c\delta T^{1/2}L(1+LT)(d^p+\|x\|^2)
\exp\big\{(L+\rho/2+(\eta c)^{1/2}L)(T-t)\big\}.
\end{equation}
\end{lemma}
\begin{proof}
The above lemma is basically a special version of Lemma 2.3 in \cite{HJKN2020}
(with
$(X^{x,k}_{t,s})_{s\in[t,T]}\cal (X^{t,x,k}_s)_{s\in[t,T]}$,
$\eta\cal \eta$, $L\cal L$, $T\cal T$,
$\psi\cal(\bR^d\ni x\mapsto (d^p+\|x\|^2)\in\bR)$, 
$T^{1/2}$ in (27) in \cite{HJKN2020} $\cal LT^{1/2}$, $h\cal h$,
$V\cal(\bR^d\ni x\mapsto (d^p+\|x\|^2)\in\bR)$, $p\cal 2$,
$q\cal 2$, $\delta\cal \delta$, $g\cal g$, $f \cal f$,
$u_k\cal u_k$, $ce^{\rho(t-s)}$ in the last condition of Lemma 2.3 
in \cite{HJKN2020} $\cal e^{\rho(t-s)}$ in the notation of Lemma 2.3 
in \cite{HJKN2020}).
One can follow the proof of Lemma 2.3 
in \cite{HJKN2020} step by step to get \eqref{u difference}. 
\end{proof}

\section{\textbf{Existence and uniqueness of viscosity solutions of PIDE
\eqref{APIDE} and its Feynman-Kac representation}}
\label{section PIDE}
There are many existence and uniqueness results
for viscosity solutions of nonlinear second-order PIDEs in the literature
(c.f. e.g., \cite{AT1996,BBP1997,BI2008,CIP1992,IY1993,JK}).
For the comparison principle, we verify that the conditions of Theorem 3.1 
in \cite{JK} are satisfied; see Proposition \ref{proposition uniqueness PIDE}.
Moreover, we show that there exists a viscosity solution of PIDE \eqref{APIDE} 
which satisfies a Feynman-Kac representation; 
see Proposition \ref{proposition PIDE existence}. 
The idea is to first show for a corresponding linear PIDE 
having smooth coefficients that it possesses a classical solution 
satisfying a Feynman-Kac representation. 
Then, combining a limit argument involving mollifiers 
as well as a stochastic fixed point argument based on \cite{BHJ2020,HJKN2020}
we can conclude the desired result.

We assume the settings in Section \ref{section setting},
fix $d$ and $\theta$, and omit  the superscript $d$ and $\theta$ in the notations. 
Then we introduce some notions and notations 
which will be used throughout this section.
To ease notation we introduce some functionals.
Let
$
H:(0,T)\times \bR^d\times\bR\times\bR^d\times\bS^d\to \bR
$
be a functional defined for each
$(t,x,r,y,A)\in(0,T)\times\bR^d\times\bR\times\bR^d\times\bS^d$ by
$$
H(t,x,r,y,A):= 
-\langle y,\mu(t,x)\rangle
-\frac{1}{2}\operatorname{Trace}
(\sigma(t,x)[\sigma(t,x)]^TA)
-f(t,x,r).
$$
Let 
$
G_0:(0,T)\times \bR^d\times\bR\times\bR^d\times\bS^d\times C^2(\bR^d)\to \bR
$
be a functional defined for all
$(t,x,r,y,A,\phi)
\in(0,T)\times\bR^d\times\bR\times\bR^d\times\bS^d\times C^2(\bR^d)$ by 
\begin{align*}
G_0(t,x,r,y,A,\phi):=G(t,x,r,y,A,\phi)+f(t,x,r)-h(t,x),
\end{align*}
where $G$ is defined by \eqref{def operator G},
and $h\in C([0,T]\times\bR^d)$ is a fixed function 
satisfying the follows. 
\begin{assumption}                                      \label{assumption h}
There exist a positive constant
$C_0$ satisfying for all $(t,x)\in[0,T]\times\bR^d$ that 
\begin{equation}                                   \label{cond h}
\quad |h(t,x)|^2\leq C_0(1+\|x\|^2).
\end{equation}
\end{assumption}
For every $\kappa\in(0,1)$, let
$
G_\kappa:
(0,T)\times \bR^d\times\bR\times\bR^d\times\bS^d\times 
SC((0,T)\times\bR^d)\times C^{1,2}((0,T)\times\bR^d)\to \bR
$
be a functional defined for each 
$(t,x,r,y,A,\psi,\varphi)\in(0,T)\times\bR^d\times\bR\times\bR^d\times\bS^d
\times SC((0,T)\times\bR^d)\times C^{1,2}((0,T)\times\bR^d)$ by
\begin{align*}
G_\kappa(t,x,r,y,A,\psi,\varphi):=
&
H(t,x,r,y,A)\\
&
-\int_{\|z\|\geq\kappa}\left(\psi(t,x+\eta_t(x,z))-\psi(t,x)
-\langle y,\eta_t(x,z)\rangle\right)\,\nu(dz)\\
&
-\int_{\|z\|<\kappa}\left(\varphi(t,x+\eta_t(x,z))-\varphi(t,x)
-\langle\nabla_x \varphi(t,x),\eta_t(x,z)\rangle\right)\,\nu(dz).
\end{align*}

\begin{lemma}                                \label{Lemma equiv def}
Let $u\in USC_1((0,T)\times\bR^d)$ be a viscosity subsolution 
of PIDE \eqref{APIDE}
Let $\varphi\in C^{1,2}((0,T)\times\bR^d)$, $(t_0,x_0)\in(0,T)\times\bR^d$,
and assume that $u-\varphi$ has a local maximum at $(t_0,x_0)$. Then
$$
-\frac{\partial}{\partial t}\varphi(t_0,x_0)
+G(t_0,x_0,u(t_0,x_0),\nabla_x \varphi(t_0,x_0),
\operatorname{Hess}_x\varphi(t_0,x_0),
\varphi(t,\cdot))
\leq 0.
$$
\end{lemma}
\begin{proof}
First, notice 
that the upper semi-continuity of $u$ ensures that there exist a non-empty
open set $\mathcal{O}\subseteq (0,T)\times\bR^d$ and 
$\phi\in C^{1,2}((0,T)\times\bR^d)$ such that $(t_0,x_0)\in\mathcal{O}$,
$\varphi(t,x)=\phi(t,x)$ for all $(t,x)\in\mathcal{O}$,
and $u-\phi$ has a global maximum at $(t_0,x_0)$.
Then we define the function
$$
\phi_0(t,x):=\phi(t,x)-\phi(t_0,x_0)+u(t_0,x_0),
\quad (t,x)\in(0,T)\times\bR^d.
$$
Note that $\phi_0\in C^{1,2}((0,T)\times\bR^d)$, 
$\phi_0(t_0,x_0)=u(t_0,x_0)$, and $u\leq\phi_0$.
Thus, the fact $u\in USC_1((0,T)\times\bR^d)$ is a 
viscosity subsolution of PIDE \eqref{APIDE} implies that 
$$
-\frac{\partial}{\partial t}\phi_0(t_0,x_0)
+G(t_0,x_0,\phi_0(t_0,x_0),\nabla_x \phi_0(t_0,x_0),
\operatorname{Hess}_x\phi_0(t_0,x_0),
\phi_0(t,\cdot))
\leq 0.
$$
Hence, by the definition of $\phi_0$, and the fact that
$\phi_0(t_0,x_0)=u(t_0,x_0)$ and 
$\phi_0(t,x)=\varphi(t,x)-\phi(t_0,x_0)+u(t_0,x_0)$ for all $(t,x)\in\mathcal{O}$
we have
$$
-\frac{\partial}{\partial t}\varphi(t_0,x_0)
+G(t_0,x_0,u(t_0,x_0),\nabla_x \varphi(t_0,x_0),
\operatorname{Hess}_x\varphi(t_0,x_0),
\varphi(t,\cdot))
\leq 0.
$$
This completes the proof of Lemma \ref{Lemma equiv def}.
\end{proof}

\begin{lemma}                                  \label{lemma continuity G 0}
Let Assumptions \ref{assumption Lip and growth}, \ref{assumption pointwise},
and \ref{assumption h}
hold. Then for every $\varphi\in C^{1,2}((0,T)\times\bR^d)$, the mapping
$$
(0,T)\times\bR^d\times\bR\times\bR^d\times\bS^d\ni(t,x,r,y,A)
\mapsto G_0(t,x,r,y,A,\varphi(t,\cdot))\in\bR
$$
is continuous.
\end{lemma}
\begin{proof}
By the definition of $G_0$, to prove this lemma we only need to show that
for every $\varphi\in C^{1,2}((0,T)\times\bR^d)$ the mapping 
\begin{equation}                                          \label{map nonlocal}
(0,T)\times\bR^d\ni(t,x)\mapsto\int_{\bR^d}\Big(
\varphi(t,x+\eta_t(x,z))-\varphi(t,x)
-\langle \nabla_x\varphi(t,x),\eta_t(x,z)\rangle
\Big)\,\nu(dz)\in\bR
\end{equation}
is continuous. To see this, note that by Taylor's formula we obtain for every
$\varphi\in C^{1,2}((0,T)\times\bR^d)$ and $(t,x)\in(0,T)\times\bR^d$ that
\begin{align}
&
\int_{\bR^d}\Big(
\varphi(t,x+\eta_t(x,z))-\varphi(t,x)
-\langle \nabla_x\varphi(t,x),\eta_t(x,z)\rangle
\Big)\,\nu(dz)\nonumber\\
&
=\int_{\bR^d}\int_0^1(1-\alpha)\sum_{i,j=1}^d\left[
\frac{\partial^2}{\partial x_i\partial x_j}
\varphi(t,x+\alpha\eta_t(x,z))\eta^i_t(x,z)\eta^j_t(x,z)\right]d\alpha\,\nu(dz)
\label{Taylor varphi}
\end{align} 
Moreover, by \eqref{assumption eta} and Cauchy-Schwarz inequality we have 
for all $\alpha\in[0,1]$, $z\in\bR^d$, $(t,x)\in (0,T)\times \bR^d$, 
and $(t',x')\in (0,T)\times \bR^d$ with $|t'-t|\leq [(T-t)\wedge t]/2$ 
and $\|x'-x\|\leq 1$ that
\begin{align}
&
\left|(1-\alpha)\sum_{i,j=1}^d\left[
\frac{\partial^2}{\partial x_i\partial x_j}
\varphi(t',x'+\alpha\eta_{t'}(x',z))
\eta^i_{t'}(x',z)\eta^j_{t'}(x',z)\right]\right|
\nonumber\\
&
\leq 
\sup_{s:|s-t|\leq \frac{(T-t)\wedge t}{2}}
\sup_{y:\|y-x\|\leq 1+C_d^{1/2}}\Big[\big\|
\operatorname{Hess}_y\varphi(s,y)\big\|_F\Big]
\cdot \sum_{i,j=1}^d\left|\eta^i_{t'}(x',z)\eta^j_{t'}(x',z)
\right|
\nonumber\\
&
\leq 
\sup_{s:|s-t|\leq \frac{(T-t)\wedge t}{2}}
\sup_{y:\|y-x\|\leq 1+C_d^{1/2}}\Big[\big\|
\operatorname{Hess}_y\varphi(s,y)\big\|_F\Big]
\cdot d\|\eta_{t'}(x',z)\|^2
\nonumber\\
&
\leq
\sup_{s:|s-t|\leq \frac{(T-t)\wedge t}{2}}
\sup_{y:\|y-x\|\leq 1+C_d^{1/2}}\Big[\big\|
\operatorname{Hess}_y\varphi(s,y)\big\|_F\Big]
\cdot dC_d(1\wedge\|z\|^2).                                 \label{est J}
\end{align}
Thus, by \eqref{Taylor varphi} and Lebesgue's dominated convergence theorem
we obtain that \eqref{map nonlocal} is continuous. The proof of this lemma
is therefore completed.
\end{proof}

\begin{proposition}                          \label{proposition uniqueness PIDE}
Let Assumptions \ref{assumption Lip and growth} and \ref{assumption pointwise}
hold, and let $u_1,u_2\in C_1([0,T]\times\bR^d)$ 
be two viscosity solutions of PIDE \eqref{APIDE} such that 
$u_1(T,x)=u_2(T,x)=g(x)$ for all $x\in\bR^d$.
Then we have for all
$(t,x)\in[0,T]\times\bR^d$ that 
$u_1(t,x)=u_2(t,x)$.
\end{proposition}

\begin{proof}
By \eqref{assumption Lip f g} and Rademacher's theorem 
(see, e.g., Theorem 3.1.6 in \cite{HF}) we first notice 
for almost all $x\in\bR^d$ that
$\nabla_x g(x)$ exists, and
\begin{equation}                                        \label{bbd Dg}
\|\nabla_x g(x)\|\leq (Ld)^{1/2}T^{-1/2}.
\end{equation}
To apply Theorem 3.1 in \cite{JK}, it remains to verify that Assumptions
(C1)-(C4) and (F0)-(F6) in \cite{JK} are satisfied.
Note that by Remark 2.2 in \cite{JK}, Assumptions (F0)-(F4) in \cite{JK} implies
Assumptions (C1)-(C4) in \cite{JK}. Thus, we only need to verify Assumptions
(F0)-(F6).
In the rest of this proof, we fix
$u,-v\in USC_1((0,T)\times\bR^d)$, $w\in SC_1((0,T)\times\bR^d)$,
$\varphi,\psi\in C^{1,2}_1((0,T)\times\bR^d)$.
To ease notation, for each $\phi\in C^{1,2}((0,T)\times\bR^d)$ 
and $(t,x)\in(0,T)\times\bR^d$ we define
\begin{align*}
J\phi(t,x):=
&
\int_{\bR^{d}}\left(\phi(t,x+\eta_t(x,z))-\phi(t,x)
-\langle\nabla_x \phi(t,x),\eta_t(x,z)\rangle\right)\,\nu(dz).
\end{align*}
By the definition of $G$ and $G_\kappa$, we first notice that  
(F0) obviously  holds (with  $F\cal G$ and $F_\kappa\cal G_\kappa$ 
in the notation of \cite{JK}). 
Next, by Lemma \ref{lemma continuity G 0}
and the continuity of $h$ and $f$, we notice that the mapping
\begin{equation}                                        \label{continuity G}
(0,T)\times\bR^d\times\bR\times\bR^d\times\bS^d\ni(t,x,r,y,A)
\mapsto G(t,x,r,y,A,\varphi(t,\cdot))\in\bR
\end{equation}
is continuous. 
Furthermore, for each $k\in\bN$ 
let $\varphi_k\in C^{1,2}((0,T)\times\bR^d)$, $(t_k,x_k)\in(0,T)\times\bR^d$,
and let $(t_0,x_0)\in(0,T)\times\bR^d$.
Assume that $(t_k,x_k)\to(t_0,x_0)$ as $k\to\infty$, and
\begin{equation}                               \label{conv phi k}
\varphi_k\to \varphi,
\;\; \nabla_x\varphi_k\to\nabla_x\varphi,
\;\; \operatorname{Hess}_x\varphi_k\to\operatorname{Hess}_x\varphi
\;\;
\text{locally uniformly on $(0,T)\times\bR^d$ as $k\to\infty$}.
\end{equation}
By \eqref{assumption eta}, \eqref{int z & q bound}, 
and the analogous argument to obtain
\eqref{est J}, for all $k\in\bN$ with 
$$
|t_k-t_0|\leq[(T-t_0)\wedge t_0]/2 \quad
\text{and} \quad \|x_k-x_0\|\leq 1
$$
we have that
\begin{align*}
&
\big|J\varphi_k(t_k,x_k)-J\varphi(t_k,x_k)\big|
\\
&
=\big|J(\varphi_k-\varphi)(t_k,x_k)\big|
\\
&
\leq 
\left(
\sup_{|s-t_0|\leq \frac{(T-t_0)\wedge t_0}{2}}
\sup_{\|y-x_0\|\leq 1+C_d^{1/2}}\Big[\big\|
\operatorname{Hess}_x(\varphi_k-\varphi)(s,y)\big\|_F\Big]
\right) 
\cdot
\int_{\bR^d}dC_d(1\wedge\|z\|^2)\,\nu(dz)
\\
&
\leq Kd^{p+1}C_d
\cdot
\left(
\sup_{|s-t_0|\leq \frac{(T-t_0)\wedge t_0}{2}}
\sup_{\|y-x_0\|\leq 1+C_d^{1/2}}\Big[\big\|
\operatorname{Hess}_x(\varphi_k-\varphi)(s,y)\big\|_F\Big]
\right).
\end{align*} 
Hence, \eqref{conv phi k} implies that
\begin{equation}                          \label{J conv 1}
\lim_{k\to\infty}\big|
J\varphi_k(t_k,x_k)-J\varphi(t_k,x_k)
\big|=0.
\end{equation}
Then,  by \eqref{continuity G}, \eqref{J conv 1}, and the triangle inequality
we have
\begin{align} 
&                         
\lim_{k\to\infty}
\big|
J\varphi(t_0,x_0)-J\varphi_k(t_k,x_k)
\big|
\nonumber\\
&
\leq
\lim_{k\to\infty}
\big|
J\varphi(t_0,x_0)-J\varphi(t_k,x_k)
\big|
+
\lim_{k\to\infty}
\big|
J\varphi_k(t_k,x_k)-J\varphi(t_k,x_k)
\big|=0.                                            \label{conv J phi k}
\end{align}
This together with \eqref{continuity G} ensure that (F1) holds
(with $F\cal G$ in the notation of \cite{JK}).
To verity (F2), let $(t_0,x_0)\in(0,T)\times\bR^d$, and assume that the functions
$(u-v)(t_0,\cdot)$ and $(\varphi-\psi)(t_0,\cdot)$ have global maxima at
$x_0$. Then for all $z\in\bR^d$ we have
$$
(u-v)(t_0,x_0+\eta_{t_0}(x_0,z))\leq (u-v)(t_0,x_0), \quad
(\varphi-\psi)(t_0,x_0+\eta_{t_0}(x_0,z))\leq (\varphi-\psi)(t_0,x_0),
$$
and
$$
\nabla_x(\varphi-\psi)(t_0,x_0)=0.
$$
Hence, by the definition of $G_\kappa$ we have for all
$r\in\bR$, $y\in\bR^d$, and $A,B\in\bS^d$ with $A\leq B$ that
\begin{align*}
&
G_\kappa\big(t_0,x_0,r,y,A,u(t_0,\cdot),\varphi(t_0,\cdot)\big)
-G_\kappa\big(t_0,x_0,r,y,B,v(t_0,\cdot),\psi(t_0,\cdot)\big)
\\
&
=-\frac{1}{2}\operatorname{Trace}\left(
\sigma(t_0,x_0)[\sigma(t_0,x_0)]^T(A-B)
\right)
\\
& \quad
-\int_{\|z\|\geq \kappa}\left[
(u-v)(t_0,x_0+\eta_{t_0}(x_0,z))-(u-v)(t_0,x_0)
\right]\nu(dz)
\\
& \quad
-\int_{\|z\|< \kappa}\left[
(\varphi-\psi)(t_0,x_0+\eta_{t_0}(x_0,z))-(\varphi-\psi)(t_0,x_0)
-\langle \nabla_x(\varphi-\psi)(t_0,x_0),\eta_{t_0}(x_0,z) \rangle
\right]\nu(dz)
\\
& 
\geq 
\frac{1}{2}\sum_{i,j=1}^d\sum_{k=1}^d
\sigma^{ik}(t_0,x_0)\sigma^{jk}(t_0,x_0)(B-A)^{ij}
\geq 0
\end{align*}
Thus, (F2) is verified (with $F_\kappa\cal G_\kappa$ 
in the notation of \cite{JK}).
Next, by \eqref{assumption Lip f g} we obtain for all $(t,x)\in(0,T)\times\bR^d$,
and $r,s\in\bR$ that
$$
|f(t,x,r)-f(t,x,s)|\leq L^{1/2}|r-s|.
$$
This implies that for all $(t,x)\in(0,T)\times\bR^d$, $y\in\bR^d$, $A\in\bS^d$,
and $r,s\in\bR$ with $r\leq s$
$$
G(t,x,r,y,A,\varphi(t,\cdot))-G(t,x,s,y,A,\varphi(t,\cdot))
\geq L^{1/2}(r-s).
$$
Thus, we verify (F3) (with $F\cal G$ and $\gamma\cal L^{1/2}$ in the notation of
\cite{JK}). Next, by the definition of $G_\kappa$ it is easy to see that
(F4) holds (with $F_\kappa \cal G_\kappa$ in the notation of \cite{JK}).
To see that (F5) is satisfied, fix $t_0\in(0,T)$ , and 
let $\phi_k\in C^{1,2}_1((0,T)\times\bR^d)$ for 
all $k\in\bN$ such that 
$\lim_{k\to\infty}\phi_k(t_0,x)=w(t_0,x)$ for almost all $x\in\bR^d$.
Moreover, assume that there is a constant $\mathfrak{C}>0$ such that
\begin{equation}                                   \label{growth phi k}
|\phi_k(t_0,x)|\leq \mathfrak{C}(1+\|x\|) \quad
\text{for all $x\in\bR^d$}.
\end{equation}
Note that Assumption \ref{assumption pointwise} ensures 
$\|\eta_{t_0}(x,z)\|^2\leq C_d$
for all $x,z\in\bR^d$.
This together with \eqref{growth phi k} and the assumption 
$w\in SC_1((0,T)\times\bR^d)$ imply that for all $x,z\in\bR^d$
\begin{equation}                                         \label{bound w phi k}
\big|(w-\phi_k)(t_0,x+\eta_{t_0}(x,z))-(w-\phi_k)(t_0,x)\big|
\leq \sup_{y:\|y-x_0\|\leq C_d^{1/2}}\big\{2|w(t_0,y)|\big\}
+2\mathfrak{C}(1+\|x\|+C_d^{1/2}).
\end{equation} 
Moreover, $\nu$ being a L\'evy measure on $\bR^d$ ensures that 
$\nu(\{z\in\bR^d:\|z\|\geq \kappa\})<\infty$ for all $\kappa\in(0,1)$.
Hence, \eqref{bound w phi k} allows us to apply Lebesgue's
dominated convergence theorem to show that for all
$(t,x)\in(0,T)\times\bR^d$ and $\kappa\in(0,1)$ 
$$
\left|
\lim_{k\to\infty}\int_{\|z\|\geq \kappa}\left[
(w-\phi_k)(t_0,x+\eta_{t_0}(x,z))-(w-\phi_k)(t_0,x)
\right]\nu(dz)
\right|
=0.
$$
This implies that for all $x\in\bR^d$, $r\in\bR$, $y\in\bR^d$,
$A\in\bS^d$, and $\kappa\in(0,1)$
$$
\lim_{k\to\infty}
G_\kappa\big(t_0,x,r,y,A,\phi_k(t_0\cdot),\varphi(t_0,\cdot)\big)
=
G_\kappa\big(t_0,x,r,y,A,w(t_0\cdot),\varphi(t_0,\cdot)\big).
$$
Therefore, we verify (F5) (with $F_\kappa\cal G_\kappa$ 
in the notation of \cite{JK}).\\
Next, to verify (F6) let $\alpha,\varepsilon,\beta\in(0,\infty)$, and define
$$
\phi(t,x,y):=e^{\beta t}\frac{\alpha}{2}\|x-y\|^2
+e^{\beta t}\frac{\varepsilon}{2}(\|x\|^2+\|y\|^2),
\quad (t,x,y)\in(0,T)\times\bR^d\times\bR^d.
$$
Let $(x_0,y_0,t_0)$ be a global maximum of 
\begin{equation}                                     \label{function u v phi}
(0,T)\times\bR^d\times\bR^d\ni(t,x,y)
\mapsto
u(t,x)-v(t,y)-\phi(t,x,y)\in\bR.
\end{equation}
Moreover, let $X,Y\in\bS^d$ satisfy
\begin{equation}                                 \label{matrix cond}
\begin{pmatrix}
X & 0 \\
0 & -Y
\end{pmatrix}
\leq 2e^{\beta t_0}\alpha
\begin{pmatrix}
\mathbf{I}_d & -\mathbf{I}_d\\
-\mathbf{I}_d & \mathbf{I}_d
\end{pmatrix}
+e^{\beta t_0}\varepsilon
\begin{pmatrix}
\mathbf{I}_d & 0\\
0 & \mathbf{I}_d
\end{pmatrix},
\end{equation}
where we recall that $\mathbf{I}_d$ denotes the $d\times d$ identity matrix.
By \eqref{assumption eta} we first notice 
for all $\kappa\in(0,1)$ that
\begin{align}
&
\int_{\|z\|<\kappa}\left[
\phi(t_0,x_0,y_0+\eta_{t_0}(y_0,z))-\phi(t_0,x_0,y_0)
-\langle
\nabla_y\phi(t_0,x_0,y_0),\eta_{t_0}(y_0,z)
\rangle
\right]\nu(dz)
\nonumber\\
& \quad
+\int_{\|z\|<\kappa}\left[
\phi(t_0,x_0+\eta_{t_0}(x_0,z),y_0)-\phi(t_0,x_0,y_0)
-\langle
\nabla_x\phi(t_0,x_0,y_0),\eta_{t_0}(x_0,z)
\rangle
\right]\nu(dz)
\nonumber\\
&
=\int_{\|z\|<\kappa}\int_0^1(1-\vartheta)\sum_{i,j=1}^d
\left(\frac{\partial^2}{\partial y_i\partial y_j}
\phi\right)(t_0,x_0,y_0+\vartheta\eta_{t_0}(y_0,z))
\eta^i_{t_0}(y_0,z)\eta^j_{t_0}(y_0,z)
\,d\vartheta\,\nu(dz)
\nonumber\\
& \quad
+ \int_{\|z\|<\kappa}\int_0^1(1-\vartheta)\sum_{i,j=1}^d
\left(\frac{\partial^2}{\partial x_i\partial x_j}
\phi\right)(t_0,x_0+\vartheta\eta_{t_0}(x_0,z),y_0)
\eta^i_{t_0}(x_0,z)\eta^j_{t_0}(x_0,z)
\,d\vartheta\,\nu(dz)
\nonumber\\
&
=\int_{\|z\|<\kappa}\int_0^1(1-\vartheta)\sum_{i=1}^d\left[
e^{\beta t_0}(-\alpha+\varepsilon)
\big|\eta^i_{t_0}(y_0,z)\big|^2
+e^{\beta t_0}(\alpha+\varepsilon)
\big|\eta^i_{t_0}(x_0,z)\big|^2\right]
\,d\vartheta\,\nu(dz)
\nonumber\\
&
\leq \int_{\|z\|<\kappa}e^{\beta T}(\alpha+2\varepsilon)C_d(1\wedge\|z\|^2)
\,\nu(dz).
\label{est G k 1}
\end{align}
By \eqref{int z & q bound}, we obtain 
$$
\int_{\bR^d}e^{\beta T}(\alpha+2\varepsilon)C_d(1\wedge\|z\|^2)
\,\nu(dz)<\infty.
$$
This together with Lebesgue's dominated convergence theorem imply that
\begin{equation}                       \label{conv z kappa}
\lim_{\kappa\to 0}
\int_{\|z\|<\kappa}e^{\beta T}(\alpha+2\varepsilon)C_d(1\wedge\|z\|^2)
\,\nu(dz)=0.
\end{equation} 
Furthermore, note that for all $t\in(0,T)$ and $x,y\in\bR^d$ we have
$$
\nabla_x\phi(t,x,y)=e^{\beta t}\alpha(x-y)+e^{\beta t}\varepsilon x
\quad \text{and} \quad
\nabla_y\phi(t,x,y)=-e^{\beta t}\alpha(x-y)+e^{\beta t}\varepsilon y.
$$
Hence, it holds for all $\kappa\in(0,1)$ that
\begin{align}
E_\kappa 
& 
:=-\int_{\|z\|\geq \kappa}\Big[
v(t_0,y_0+\eta_{t_0}(y_0,z))-v(t_0,y_0)
\nonumber\\
& \quad \quad \quad \quad \quad \quad
-\langle
e^{\beta t_0}\alpha(x_0-y_0)-e^{\beta t_0}\varepsilon y_0,
\eta_{t_0}(y_0,z)
\rangle
\Big]\nu(dz)
\nonumber\\
& \quad
+
\int_{\|z\|\geq \kappa}\Big[
u(t_0,x_0+\eta_{t_0}(x_0,z))-u(t_0,x_0)
\nonumber\\
& \quad \quad \quad \quad \quad \quad
-\langle
e^{\beta t_0}\alpha(x_0-y_0)+e^{\beta t_0}\varepsilon x_0,
\eta_{t_0}(x_0,z)
\rangle
\Big]\nu(dz)
\nonumber\\
& 
=-\int_{\|z\|\geq \kappa}\Bigg[
v(t_0,y_0+\eta_{t_0}(y_0,z))-v(t_0,y_0)
+\langle
\nabla_y\phi(t_0,x_0,y_0),
\eta_{t_0}(y_0,z)
\rangle
\Big]\nu(dz)
\nonumber\\
& \quad
+
\int_{\|z\|\geq \kappa}\Bigg[
u(t_0,x_0+\eta_{t_0}(x_0,z))-u(t_0,x_0)
-\langle
\nabla_x\phi(t_0,x_0,y_0),
\eta_{t_0}(x_0,z)
\rangle
\Big]\nu(dz).                                    \label{kfc}
\end{align}
Moreover, $(t_0,x_0,y_0)$ being a global maximum of the function
defined by \eqref{function u v phi} 
ensures that for all $z\in\bR^d$
\begin{align}
&
u(t_0,x_0)-v(t_0,y_0)-\phi(t_0,x_0,y_0)\nonumber
\\
&
\geq
u(t_0,x_0+\eta_{t_0}(x_0,z))-v(t_0,y_0+\eta_{t_0}(y_0,z))
-\phi(t_0,x_0+\eta_{t_0}(x_0,z),y_0+\eta_{t_0}(y_0,z)).
\label{maximum ineq}
\end{align}
Then by \eqref{assumption Lip mu sigma eta}, \eqref{assumption growth},
 \eqref{kfc}, \eqref{maximum ineq}, and Taylor's formula 
we have for every $\kappa\in(0,1)$ that
\begin{align}
E_\kappa
&\leq
\int_{\|z\|\geq\kappa}
\Big[
\phi(t_0,x_0+\eta_{t_0}(x_0,z),y_0+\eta_{t_0}(y_0,z))-\phi(t_0,x_0,y_0)
\nonumber\\
& \quad\quad\quad\quad\;\;
-\langle
\nabla_y\phi(t_0,x_0,y_0),
\eta_{t_0}(y_0,z)
\rangle
-\langle
\nabla_x\phi(t_0,x_0,y_0),
\eta_{t_0}(x_0,z)
\rangle
\Big]
\nu(dz)
\nonumber\\
&
=\int_{\|z\|\geq \kappa}\int_0^1
(1-\vartheta)\Bigg
[\sum_{i=1}^d(e^{\beta t_0}\alpha+e^{\beta t_0}\varepsilon)
\big|\eta^i_{t_0}(x_0,z)\big|^2
+
\sum_{i=1}^d(e^{\beta t_0}\alpha+e^{\beta t_0}\varepsilon)
\big|\eta^i_{t_0}(y_0,z)\big|^2\\
& \quad\quad\quad\quad\quad\;
-2\sum_{i=1}^de^{\beta t_0}\alpha\eta^i_{t_0}(x_0,z)\eta^i_{t_0}(y_0,z)
\Bigg]\,d\vartheta\,\nu(dz)
\nonumber\\
&
\leq 
\int_{\|z\|\geq \kappa}\left[
e^{\beta t_0}\varepsilon\left(\|\eta_{t_0}(x_0,z)\|^2+
\|\eta_{t_0}(y_0,z)\|^2\right)
+\sum_{i=1}^de^{\beta t_0}\alpha
\big|\eta^i_{t_0}(x_0,z)-\eta^i_{t_0}(y_0,z)\big|^2
\right]\nu(dz)
\nonumber\\
&
\leq 2e^{\beta t_0}\varepsilon Ld^p(1+\|x_0\|^2+\|y_0\|^2)
+e^{\beta t_0}\alpha L\|x_0-y_0\|^2.                   
\label{estimate E kappa}
\end{align}
Furthermore, by \eqref{assumption Lip mu sigma eta}, \eqref{assumption growth},
and the Cauchy-Schwarz inequality
it holds that
\begin{align}
&
-\langle
\mu(t_0,y_0),e^{\beta t_0} \alpha(x_0-y_0)-e^{\beta t_0}\varepsilon y_0
\rangle
+
\langle
\mu(t_0,x_0),e^{\beta t_0} \alpha(x_0-y_0)-e^{\beta t_0}\varepsilon x_0
\rangle
\nonumber\\
&
\leq e^{\beta t_0}\alpha \|x_0-y_0\|\cdot \|\mu(t_0,x_0)-\mu(t_0,y_0)\|
+e^{\beta t_0}\varepsilon \|y_0\|\cdot\|\mu(t_0,y_0)\|
+e^{\beta t_0}\varepsilon \|x_0\|\cdot\|\mu(t_0,x_0)\|
\nonumber\\
&
\leq L^{1/2}e^{\beta t_0}\alpha\|x_0-y_0\|^2
+4L^{1/2}d^{p/2}e^{\beta t_0}\varepsilon(1+\|x_0\|^2+\|y_0\|^2).
\label{est mu x 0 y 0}
\end{align}
Next, we use a trick by Ishii (see, e.g., Example 3.6 in \cite{CIP1992}) 
to get an estimate for 
$$
-\frac{1}{2}\operatorname{Trace}\Big(\Sigma_{y_0}\Sigma_{y_0}^TY\Big)
+\frac{1}{2}\operatorname{Trace}\Big(\Sigma_{x_0}\Sigma_{x_0}^TX\Big),
$$
where
$$
\Sigma_{x_0}:=\sigma(t_0,x_0)
\quad \text{and} \quad
\Sigma_{y_0}:=\sigma(t_0,y_0).
$$
More precisely, 
left multiplying both sides of \eqref{matrix cond} by the symmetric matrix
$$
\begin{pmatrix}
\Sigma_{x_0}\Sigma_{x_0}^T & \Sigma_{y_0}\Sigma_{x_0}^T \\
 & \\
\Sigma_{x_0}\Sigma^T_{y_0} & \Sigma_{y_0}\Sigma_{y_0}^T
\end{pmatrix}
$$
yields
\begin{align}
\begin{pmatrix}
\Sigma_{x_0}\Sigma_{x_0}^T X & -\Sigma_{y_0}\Sigma_{x_0}^T Y \\
 & \\
\Sigma_{x_0}\Sigma^T_{y_0} X & -\Sigma_{y_0}\Sigma_{y_0}^T Y
\end{pmatrix}
\leq
&
2e^{\beta t_0}\alpha
\begin{pmatrix}
\Sigma_{x_0}\Sigma_{x_0}^T-\Sigma_{y_0}\Sigma_{x_0}^T 
& 
-\Sigma_{x_0}\Sigma_{x_0}^T+\Sigma_{y_0}\Sigma_{x_0}^T\\
& 
\\
\Sigma_{x_0}\Sigma_{y_0}^T-\Sigma_{y_0}\Sigma_{y_0}^T 
& 
-\Sigma_{x_0}\Sigma_{y_0}^T+\Sigma_{y_0}\Sigma_{y_0}^T
\end{pmatrix}
\nonumber\\
&
+e^{\beta t_0}\varepsilon
\begin{pmatrix}
\Sigma_{x_0}\Sigma_{x_0}^T & \Sigma_{y_0}\Sigma_{x_0}^T \\
 & \\
\Sigma_{x_0}\Sigma^T_{y_0} & \Sigma_{y_0}\Sigma_{y_0}^T
\end{pmatrix}.                                 
\label{matrix ineq}
\end{align}
By Assumption \ref{assumption Lip and growth}, 
taking traces to \eqref{matrix ineq} yields 
\begin{align}
&
-\frac{1}{2}\operatorname{Trace}\Big(
\Sigma_{y_0}\Sigma_{y_0}^TY
\Big)
+\frac{1}{2}\operatorname{Trace}\Big(
\Sigma_{x_0}\Sigma_{x_0}^TX
\Big)
\nonumber\\
&
\leq 
e^{\beta t_0}\alpha\operatorname{Trace}
\left\{
\Big(\Sigma_{x_0}-\Sigma_{y_0}\Big)\Big(\Sigma_{x_0}-\Sigma_{y_0}\Big)^T
\right\}
+\frac{1}{2} e^{\beta t_0}\varepsilon \operatorname{Trace}
\Big\{
\Sigma_{x_0}\Sigma_{x_0}^T+\Sigma_{y_0}\Sigma_{y_0}^T
\Big\}
\nonumber\\
&
=e^{\beta t_0}\alpha\big\|\Sigma_{x_0}-\Sigma_{y_0}\big\|_F^2
+\frac{1}{2} e^{\beta t_0}\varepsilon
\left(\big\|\Sigma_{x_0}\big\|_F^2+\big\|\Sigma_{y_0}\big\|_F^2\right)
\nonumber\\
&
\leq Le^{\beta t_0}\alpha\|x_0-y_0\|^2
+Ld^pe^{\beta t_0}\varepsilon(1+\|x_0\|^2+\|y_0\|^2).
\label{trace est}
\end{align}
Moreover, by \eqref{assumption Lip f g} 
it holds for all $r\in\bR$ that
\begin{equation}                                               \label{est f r}
|f(t_0,x_0,r)-f(t_0,y_0,r)|\leq L^{1/2}T^{-3/2}\|x_0-y_0\|.
\end{equation}
Then using \eqref{est G k 1}, \eqref{estimate E kappa}, \eqref{est mu x 0 y 0},
\eqref{trace est}, and \eqref{est f r} 
we verify condition (F6) on page 287 of \cite{JK}
(with $m\cal 1$, $\alpha\cal \alpha$, $\varepsilon\cal \varepsilon$,
$\lambda\cal \beta$, $p\cal 2$ , $\bar{F}_\kappa\cal G_\kappa$, 
$F_\kappa\cal G_\kappa$,
$\eta_1=\eta_2=\eta_3=\eta_4=0$, $p_1=p_2=p_3=p_4=p_s=0$,
$K_1=L^{1/2}T^{-3/2}$, $K_2=e^{\beta t_0}L^{1/2}(1+2L^{1/2})$,
$K_3=L^{1/2}(3L^{1/2}d^p+4d^{p/2})$,
and $m_{\alpha,\varepsilon}$ being the integral term in \eqref{conv z kappa}
in the notation of condition (F6) in \cite{JK}, noting that 
$m_{\alpha,\varepsilon}\equiv m_{\alpha,\varepsilon,\lambda}$ is allowed to
depend on $\lambda$ in  the notation of condition (F6) in \cite{JK};
see p.\ 291-292 in \cite{JK}).
Hence, by Remark \ref{remark v solution} and 
Theorem 3.1 in \cite{JK} we obtain $u_1(t,x)=u_2(t,x)$ for all
$(t,x)\in[0,T]\times\bR^d$.
\end{proof}

To show the existence of viscosity solutions of \eqref{APIDE}, we first consider
the linear PIDE
\begin{align}
&
-\frac{\partial}{\partial t}u(t,x)
+G_0(t,x,u(t,x),\nabla_x u(t,x),\operatorname{Hess}_xu(t,x),u(t,\cdot))=0
\quad \text{on $(0,T)\times \bR^d$},
\label{linear PIDE}\\
&
u(T,x)=g(x) \quad \text{on $\bR^d$}.             \label{linear PIDE initial}
\end{align}

\begin{lemma}                                         \label{lemma smooth PIDE}
Let Assumptions \ref{assumption Lip and growth}, \ref{assumption pointwise},
\ref{assumption jacobian}, and \ref{assumption h} hold.
Moreover, for every $t\in[0,T]$ and $z\in\bR^d$ 
we assume that $\mu(t,\cdot)$, $\sigma(t,\cdot)$,
$\eta_t(\cdot,z)$, $g(\cdot)$, and $h(t,\cdot)$ 
are compactly supported and infinitely differentiable.
Then PIDE \eqref{linear PIDE}
has a unique classical solution $u\in C^{1,2}([0,T]\times\bR^d)$
with $u(T,x)=g(x)$ for all $x\in\bR^d$.
Moreover, we have for all $(t,x)\in[0,T]\times\bR^d$ that
\begin{equation}                                           \label{FK linear}
u(t,x)=\bE\Bigg[g(X^{t,x}_T)+\int_t^Th(s,X^{t,x}_s)\,ds\Bigg].
\end{equation}
\end{lemma}
To prove the above lemma, we recall some notions and introduce some notations.
For each $n\in\bN$, 
a finite list $\alpha=\alpha_1\alpha_2,...,\alpha_n$  of numbers $\alpha_i\in\{1,2,...,d\}$ 
is called a multi-number of length $|\alpha|:=n$, and we use the notation 
$$
D_{\alpha}:=D_{\alpha_1}D_{\alpha_2}...D_{\alpha_n},
$$ 
where for $i\in\{1,...,d\}$ and locally integrable functions $v: \bR^d\to\bR$,
$D_iv$ refers to the weak derivative 
(see, e.g., 1.62 in \cite{AF} for the definition) of $v$ with respect $x_i$.
We also use the multi-number $\epsilon$ of length $0$, 
and agree that $D_{\epsilon}$ means the identity operator. 
For a separable Banach space $V$ we use the notation $L_2([0,T],V)$ 
for the space of Borel
functions $v:[0,T]\to V$ such that 
$$
\int_0^T\|v(t)\|^2_V\,dt<\infty.
$$
We also denote by $L_2(\bR^d)$ the space 
Borel functions $\varphi:\bR^d\to\bR$ such that 
$$
\|\varphi\|^2_{L_2(\bR^d)}:=\int_{\bR^d}|\varphi(x)|^2\,dx<\infty. 
$$
For integers $m\in \bN_0$, the notation $W^m_2$ refers to the 
Sobolev space defined as the completion of $C_c^{\infty}(\bR^d)$ in the norm 
$$
\|\varphi\|_{W^m_2(\bR^d)}
:=\sum_{0\leq|\alpha|\leq m}\|D_{\alpha}\varphi\|_{L_2(\bR^d)}.  
$$
Moreover, for $\delta\in(0,1)$ we use the notation $C^2_\delta(\bR^d)$ to denote
the H\"older space defined as the set 
of all functions $v:\bR^d\to\bR$ with continuous 
derivatives up to order $2$ such that
$$
\|v\|_{C^2_\delta}:=\max_{0\leq|\alpha|\leq 2}\sup_{x\in\bR^d}|D_\alpha v(x)|
+\max_{0\leq|\alpha|\leq 2}\sup_{\substack{ x,y\in\bR^d \\ x\neq y}}
\frac{|D_\alpha v(x)- D_\alpha v(y)|}{|x-y|^\delta}<\infty.
$$

\begin{proof}[Proof of Lemma \ref{lemma smooth PIDE}]
We first introduce the notion of generalised solutions
(cf., Definition 2.1 in \cite{DGW2020}) of the linear PIDE
\begin{align}
&
\frac{\partial}{\partial t}v(t,x)
+G_0(t,x,v(t,x),\nabla_x v(t,x),\operatorname{Hess}_xv(t,x),v(t,\cdot))=0
\quad \text{on $(0,T)\times \bR^d$},
\label{linear PIDE 1}\\
&
v(0,x)=g(x) \quad \text{on $\bR^d$}.             \label{linear PIDE initial 1}
\end{align}
Note that by virtue of Remark 2.1 in \cite{DGW2020}, integrating by parts 
ensures for all
$\varphi\in C^\infty_c(\bR^d)$, $v\in W^1_2(\bR^d)$, and $t\in[0,T]$ that
\begin{align}
&
\int_{\bR^d}\varphi(x)\int_{\bR^d}\left[
v(x+\eta_t(x,z))-v(t,x)-\langle\nabla_xv(x),\eta_t(x,z)\rangle
\right]\nu(dz)\,dx\nonumber\\
&
=-\sum_{k=1}^d\int_{\bR^d}\Big[\mathcal{J}^k(t)v(x)\Big]\cdot
D_k\varphi(x)\,dx
+\int_{\bR^d}\Big[\mathcal{J}^0(t)v(x)\Big]\cdot\varphi(x)\,dx,
\label{int by parts}
\end{align}
where, by denoting $\eta_t(x,z)=(\eta_t^1(x,z),...,\eta_t^d(x,z))$
$$
\cJ^k(t)v(x):=\int_0^1\int_{\bR^d}\eta^k_t(x,z)(v(x+\vartheta\eta_t(x,z))-v(x))
\,\nu(dz)\,d\vartheta, \;\; (t,x)\in[0,T]\times\bR^d,\;\; k=1,...,d,
$$
and
\begin{align*}
\cJ^0(t)v(x):=
&
-\int_0^1\int_{\bR^d}\Bigg[\sum_{k=1}^dD_k\eta^k_t(x,z)
(v(x+\vartheta\eta_t(x,z))-v(x))\\
&
+\sum_{k,l=1}^d\vartheta\eta^k_t(x,z)(D_k\eta^l_t)(x,z)
(D_lv)(x+\vartheta\eta_t(x,z))
\Bigg]\,\nu(dz)\,d\vartheta, \quad (t,x)\in[0,T]\times\bR^d.
\end{align*}
Define $
a=(a^{i,j})_{i,j\in\{1,...,d\}}
\in C([0,T]\times\bR^d,\bS^d)$ by
$$
a(t,x):=\sigma(t,x)[\sigma(t,x)]^T,
\quad (t,x)\in[0,T]\times\bR^d.
$$
Then an $L_2(\bR^d)$-valued continuous function 
$v=(v(t))_{t\in[0,T]}$, is called a generalised solution
(see Definition 2.1 in \cite{DGW2020})
to equation \eqref{linear PIDE 1} with initial condition 
$v(0)=g$, if $v(t)\in W^1_2(\bR^d)$ 
for $dt$-almost every $t\in[0,T]$, $v\in L_2([0,T], W^1_2(\bR^d))$,  
and	for every $\varphi\in C^\infty_c(\bR^d)$ and $t\in[0,T]$
\begin{align}
\int_{\bR^d}v(t,x)\varphi(x)\,dx
&
=\int_{\bR^d}g(x)\varphi(x)\,dx
+\int_0^t\int_{\bR^d}\left(h(s,x)\varphi(x)
+\left[\mathcal{J}^0(s)v(s,x)\right]\cdot\varphi(x)\right)dx\,ds
\nonumber\\
&
\quad
+\int_0^t\int_{\bR^d}
\Bigg(
\sum_{i,j=1}^d\Big[
-\left(a^{ij}(s,x)D_jv(s,x)
+\mathcal{J}^i(s)v(s,x)\right)D_i\varphi(x)
\nonumber\\
& \quad \quad \quad \quad \quad \quad  \;
+[\mu^i(s,x)-D_ja^{ij}(s,x)][D_iv(s,x)]\cdot\varphi(x)
\Big]
\Bigg)\,dx\,ds.
\label{def weak sol}
\end{align}
By the assumptions of Lemma \ref{lemma smooth PIDE}, 
for each $m\in\bN$ Assumptions 2.1, 2.2, and 2.5  
in \cite{DGW2020} are satisfied
(with $p=2$, $m\cal m$, $a\cal a$, $b\cal \mu$, $\eta\cal \eta$,
$\psi\cal g$, and $f\cal h$
in the notation of Assumptions 2.1, 2.2, and 2.5 in \cite{DGW2020}).
Thus, Theorem 2.1 in \cite{DGW2020} ensures that PIDE \eqref{linear PIDE 1}
has a unique generalized solution $v=(v(t))_{t\in[0,T]}$ with $v(0)=g$,
and for every $m\in\bN_0$ the generalised solution 
$v$ 
is strongly continuous as a $W^m_2(\bR^d)$-valued function of $t\in[0,T]$. 
Thus, by Sobolev embedding theorem (see, e.g., 4.12 Part II in \cite{AF}) 
we have for each $\delta\in(0,1)$
the generalised solution $v$ is strongly continuous as a 
$C^2_{\delta}(\bR^d)$-valued function of $t\in[0,T]$.
Hence, for every $(t_0,x_0)\in[0,T]\times\bR^d$ and multi-number $\alpha$ 
with $|\alpha|\leq 2$ we have
\begin{equation}                                          \label{D alpha 1}
\lim_{t\to t_0}
\sup_{x\in\bR^d}\big|
D_\alpha v(t_0,x)-D_\alpha v(t,x)
\big|=0.
\end{equation}
Moreover, the fact that
$v(t,\cdot)\in C^2_{\delta}(\bR^d)$ for all $t\in[0,T]$ and $\delta\in(0,1)$
ensures that for every $(t_0,x_0)\in[0,T]\times\bR^d$ and multi-number $\alpha$ 
with $|\alpha|\leq 2$
$$
\lim_{x\to x_0}\big|D_\alpha v(t_0,x)-D_\alpha v(t_0,x_0)\big|=0.
$$
This together with \eqref{D alpha 1} imply that for every $(t_0,x_0)\in[0,T]\times\bR^d$ 
and multi-number $\alpha$ with $|\alpha|\leq 2$
\begin{align*}
\lim_{(t,x)\to (t_0,x_0)}\big|D_\alpha v(t,x)-D_\alpha v(t_0,x_0)\big|
\leq
& 
\lim_{(t,x)\to (t_0,x_0)}\big|D_\alpha v(t,x)-D_\alpha v(t_0,x)\big|
\\
& \quad
+
\lim_{x\to x_0}\big|D_\alpha v(t_0,x)-D_\alpha v(t_0,x_0)\big|
\\
&
=0.
\end{align*}
Thus, as a real-valued function on $[0,T]\times\bR^d$, 
$v\in C^{1,2}([0,T]\times\bR^d)$.
Furthermore, \eqref{def weak sol}, integrating by parts, and the fact that 
$v\in C^{1,2}([0,T]\times\bR^d)$ ensures that the unique generalised 
solution $v$ is the unique classical solution of \eqref{linear PIDE 1}
with $v(0,x)=g(x)$ for $x\in\bR^d$.
Then by time reversing, the function $u(t,x):=v(T-t,x)$, $(t,x)\in[0,T]\times\bR^d$,
is the unique classical solution of \eqref{linear PIDE} with $u(T,x)=g(x)$ for
$x\in\bR^d$. 
Moreover, by Feynman-Kac
formula (see, e.g., Theorem 17.4.10 in \cite{CE2015}, 
or Theorem 285 in \cite{Situ}) we verify \eqref{FK linear}. 
The proof of this lemma is therefore completed.
\end{proof}

In the rest of this section, 
for every $n\in\bN$ let $g^{(n)}\in C(\bR^d,\bR)$, 
$h^{(n)}\in C([0,T]\times\bR^d,\bR)$, 
$\mu^{(n)}\in C([0,T]\times\bR^d,\bR^d)$,
$\sigma^{(n)}\in C([0,T]\times\bR^d,\bR^{d\times d})$,
and $\eta^{(n)}=(\eta^{(n),1},...,\eta^{(n),d})
\in C([0,T]\times\bR^d\times\bR^d,\bR^d)$. 
Then we make the following assumptions.  

\begin{assumption}                            \label{assumption conv compact}
It holds for all non-empty compact sets 
$\mathcal{K}\subseteq[0,T]\times\bR^d$ that 
\begin{align}
\lim_{n\to\infty}\Bigg[ &
\sup_{(t,x)\in\mathcal{K}}
\Big(\big|g^{(n)}(x)-g(x)\big|+\big|h^{(n)}(t,x)-h(t,x)\big|
+\big\|\mu^{(n)}(t,x)-\mu(t,x)\big\|\nonumber\\
&
+\big\|\sigma^{(n)}(t,x)-\sigma(t,x)\big\|_F\Big)
\Bigg]= 0,                                             \label{conv compact}
\end{align}
and
\begin{equation}
\lim_{n\to\infty}\Bigg[
\sup_{(t,x)\in\mathcal{K}}
\big\|\eta^{(n)}_t(x,z)-\eta_t(x,z)\big\|
\Bigg]= 0  \quad \text{$\nu(dz)$-almost everywhere}.                                             
\label{conv compact z}
\end{equation}  
\end{assumption}

\begin{assumption}                           \label{assumption Lip and growth n}
There exists a constant $C_1>0$ satisfying 
for all $n\in\bN$, $x,y\in\bR^d$, and $t\in[0,T]$ that
\begin{equation}                           \label{assumption Lip n}
|g^{(n)}(x)-g^{(n)}(y)|^2+\|\mu^{(n)}(t,x)-\mu^{(n)}(t,y)\|^2
+\|\sigma^{(n)}(t,x)-\sigma^{(n)}(t,y)\|_F^2
\leq C_1\|x-y\|^2,
\end{equation}
and
\begin{equation}                                     \label{assumption growth n}
|h^{(n)}(t,x)|^2+|g^{(n)}(x)|^2+\|\mu^{(n)}(t,x)\|^2+\|\sigma^{(n)}(t,x)\|_F^2
\leq C_1(1+\|x\|^2).
\end{equation}
\end{assumption}

\begin{assumption}                              \label{assumption eta n}
There exists a constant $C_2>0$ such that
for all $n\in\bN$, $x,y\in \bR^d$, $t\in[0,T]$, and $z\in\bR^{d}$ 
\begin{equation}                                        \label{pointwise eta n}                                     
\|\eta^{(n)}_t(x,z)\|^2\leq C_2(1\wedge \|z\|^2), 
\quad \|\eta^{(n)}_t(x,z)-\eta^{(n)}_t(y,z)\|^2
\leq C_2\|x-y\|^2(1\wedge \|z\|^2).
\end{equation}
\end{assumption}

\begin{assumption}                                  \label{assumption jacobian n}
There exists a constant
$\delta>0$ satisfying for all $n\in\bN$,
$(t,x)\in[0,T]\times\bR^d$, $z\in\bR^d$, and $\gamma\in[0,1]$
that $D_x \eta^{(n)}_t(x,z)$ exists, and
\begin{equation}                                     \label{jacobian cond n}
\delta\leq |\det(\mathbf{I}_d+\gamma D_x \eta^{(n)}_t(x,z))|.
\end{equation}
\end{assumption}
                                    
Then for each $n\in\bN$, let 
$
G_0^{(n)}:(0,T)\times \bR^d\times\bR\times\bR^d\times\bS^d\times C^2(\bR^d)
\to \bR
$
be a functional defined as $G_0$ with $\mu^{(n)}$, 
$\sigma^{(n)}$, $\eta^{(n)}$, and $h^{(n)}$
in place of $\mu$, $\sigma$, $\eta$, and $h$, respectively.

\begin{lemma}                                      \label{lemma conv compact G}
Let Assumptions \ref{assumption Lip and growth}, \ref{assumption pointwise},
\ref{assumption h}, and \ref{assumption conv compact}-\ref{assumption eta n} hold.
Then for every $\varphi\in C^{1,2}((0,T)\times\bR^d)$ and
every compact set 
$\mathcal{K}\subseteq(0,T)\times\bR^d\times\bR\times\bR^d\times\bS^d$ 
it holds that
\begin{equation}                                          \label{conv compact G}
\lim_{n\to\infty}\left(
\sup_{(t,x,r,y,A)\in\mathcal{K}}
\Big|G_0(t,x,r,y,A,\varphi(t,\cdot))-G_0^{(n)}(t,x,r,y,A,\varphi(t,\cdot))\Big|
\right)=0.
\end{equation}
\end{lemma}
\begin{proof}
By Assumption \ref{assumption conv compact} and 
the definitions of $G_0$ and $G_0^{(n)}$, 
it is sufficient to show for every $\varphi\in C^{1,2}((0,T)\times\bR^d)$
and compact set $\mathcal{K}_0\subseteq (0,T)\times\bR^d$ that 
\begin{equation}                                   \label{conv compact J}
\lim_{n\to\infty}\Bigg(\sup_{(t,x)\in\subseteq \mathcal{K}_0}
\big|\mathbb{J}^n\varphi(t,x)\big|
\Bigg)=0,
\end{equation}
where for all $(t,x)\in(0,T)\times\bR^d$, $n\in\bN$, 
and $\varphi\in C^{1,2}((0,T)\times\bR^d)$,   
\begin{align*}
\mathbb{J}^n\varphi(t,x):=
&
\int_{\bR^{d}}\left(\varphi(t,x+\eta^{(n)}_t(x,z))-\varphi(t,x)
-\langle\nabla_x \varphi(t,x),\eta^{(n)}_t(x,z)\rangle\right)\,\nu(dz)\\
&
-\int_{\bR^{d}}\left(\varphi(t,x+\eta_t(x,z))-\varphi(t,x)
-\langle\nabla_x \varphi(t,x),\eta_t(x,z)\rangle\right)\,\nu(dz).
\end{align*}
To see this, note that by Taylor's formula and the triangle inequality 
we have for every 
$\varphi(t,x)\in C^{1,2}((0,T)\times\bR^d)$, $n\in\bN$, 
and $(t,x)\in(0,T)\times\bR^d$ that
\begin{align}
\big|\mathbb{J}^n\varphi(t,x)\big| & = 
\Bigg|
\int_{\bR^d}\int_0^1(1-\alpha)\sum_{i,j=1}^d
\Bigg(
\frac{\partial^2}{\partial x_i\partial x_j}\varphi(t,x+\alpha\eta^{(n)}_t(x,z))
\cdot\eta^{(n),i}_t(x,z)\eta^{(n),j}_t(x,z)\nonumber\\
& \quad
-\frac{\partial^2}{\partial x_i\partial x_j}\varphi(t,x+\alpha\eta_t(x,z))
\cdot\eta^{i}_t(x,z)\eta^{j}_t(x,z)
\Bigg)\,d\alpha\,\nu(dz)
\Bigg|\nonumber\\
&
\leq \int_0^1\int_{\bR^d}(1-\alpha)\sum_{i,j=1}^d\Bigg(\Bigg|
\frac{\partial^2}{\partial x_i\partial x_j}\varphi(t,x+\alpha\eta^{(n)}_t(x,z))
-\frac{\partial^2}{\partial x_i\partial x_j}\varphi(t,x+\alpha\eta_t(x,z))
\Bigg|\nonumber\\
& \quad\quad\quad\quad\quad \cdot
\big|
\eta^{(n),i}_t(x,z)\eta^{(n),j}_t(x,z)
\big|\Bigg)\,\nu(dz)\,d\alpha\nonumber\\
& \quad
+\int_0^1\int_{\bR^d}(1-\alpha)\sum_{i,j=1}^d\Bigg(
\big|
\eta^{i}_t(x,z)\eta^{j}_t(x,z)-\eta^{(n),i}_t(x,z)\eta^{(n),j}_t(x,z)
\big|
\nonumber\\
& \quad\quad\quad\quad\quad \cdot
\Bigg|
\frac{\partial^2}{\partial x_i\partial x_j}\varphi(t,x+\alpha\eta_t(x,z))
\Bigg|
\Bigg)\,\nu(dz)\,d\alpha.
\label{J n est}
\end{align}
Moreover, by Assumptions \ref{assumption pointwise} and \ref{assumption eta n},
the triangle inequality, and Cauchy Schwarz inequality we obtain for every
$\varphi\in C^{1,2}((0,T)\times\bR^d)$,
$n\in\bN$, $\alpha\in[0,1]$, $z\in\bR^d$ ,and $(t,x)\in(0,T)\times\bR^d$ that
\begin{align*}
&
\sum_{i,j=1}^d\Bigg(\Bigg|
\frac{\partial^2}{\partial x_i\partial x_j}\varphi(t,x+\alpha\eta^{(n)}_t(x,z))
-\frac{\partial^2}{\partial x_i\partial x_j}\varphi(t,x+\alpha\eta_t(x,z))
\Bigg|
\cdot
\big|
\eta^{(n),i}_t(x,z)\eta^{(n),j}_t(x,z)
\big|\Bigg)
\\
&
\leq 
2\cdot\sup_{y:\|y-x\|\leq (C_d\vee C_2)^{1/2}}\Big[\big\|
\operatorname{Hess}_x\varphi(t,y)\big\|_F\Big]
\cdot \sum_{i,j=1}^d\left|\eta^{(n),i}_{t}(x,z)\eta^{(n),j}_{t}(x,z)
\right|
\\
&
\leq 
2\cdot\sup_{y:\|y-x\|\leq (C_d\vee C_2)^{1/2}}\Big[\big\|
\operatorname{Hess}_x\varphi(t,y)\big\|_F\Big]
\cdot d\|\eta^{(n)}_{t}(x,z)\|^2
\\
&
\leq
2\cdot\sup_{y:\|y-x\|\leq (C_d\vee C_2)^{1/2}}\Big[\big\|
\operatorname{Hess}_x\varphi(t,y)\big\|_F\Big]
\cdot dC_2(1\wedge\|z\|^2),
\end{align*}
and
\begin{align*}
&
\sum_{i,j=1}^d\Bigg(
\big|
\eta^{i}_t(x,z)\eta^{j}_t(x,z)-\eta^{(n),i}_t(x,z)\eta^{(n),j}_t(x,z)
\big|
\cdot
\Bigg|
\frac{\partial^2}{\partial x_i\partial x_j}\varphi(t,x+\alpha\eta_t(x,z))
\Bigg|
\Bigg)
\\
&
\leq 
\sup_{y:\|y-x\|\leq C_d^{1/2}}\Big[\big\|
\operatorname{Hess}_x\varphi(t,y)\big\|_F\Big]
\cdot \sum_{i,j=1}^d\left[
\left|\eta^{i}_{t}(x,z)\eta^{j}_{t}(x,z)\right|
+
\left|\eta^{(n),i}_{t}(x,z)\eta^{(n),j}_{t}(x,z)\right|
\right]
\\
&
\leq
\sup_{y:\|y-x\|\leq C_d^{1/2}}\Big[\big\|
\operatorname{Hess}_x\varphi(t,y)\big\|_F\Big]
\cdot d(C_2+C_d)(1\wedge\|z\|^2),
\end{align*}
which is integrable on $\bR^d$ with respect to $\nu(dz)$.
Therefore, by \eqref{conv compact z}, \eqref{J n est},
and Lebesgue's
dominated convergence theorem we get \eqref{conv compact J}, 
which completes the proof of this lemma.
\end{proof}

\begin{lemma}                                         \label{lemma subsolution}
Let Assumptions \ref{assumption Lip and growth}, \ref{assumption pointwise},
\ref{assumption h}, and \ref{assumption conv compact}-\ref{assumption eta n} hold.
For every $n\in\bN$, let $u\in C_1([0,T]\times\bR^d)$, 
$u^{(n)}\in C_1([0,T]\times\bR^d)$.
Assume for all non-empty compact sets $\mathcal{K}\subseteq[0,T]\times\bR^d$ that 
\begin{equation}                                        \label{conv compact u}
\lim_{n\to\infty}\Bigg[
\sup_{(t,x)\in\mathcal{K}}
\big|u^{(n)}(t,x)-u(t,x)\big|\Bigg]= 0.                                                                                                                    
\end{equation}
Moreover, assume for all $n\in\bN$ that $u^{(n)}$ 
is a viscosity solution of 
\begin{align}
-\frac{\partial}{\partial t}u^{(n)}(t,x)
+G_0^{(n)}(t,x,u^{(n)}(t,x),\nabla_x u^{(n)}(t,x),\operatorname{Hess}_x
u^{(n)}(t,x),u^{(n)}(t,\cdot))=0
\;\text{on $(0,T)\times \bR^d$}
\label{linear PIDE n}            
\end{align}
with $u^{(n)}(T,x)=g^{(n)}(x)$ for $x\in\bR^d$.
Then $u$ is a viscosity solution 
of the linear PIDE \eqref{linear PIDE} with $u(T,x)=g(x)$ for $x\in\bR^d$.
\end{lemma}

\begin{proof}
This proof is analogous to the proof of Lemma 2.18 in \cite{BHJ2020}.
We only prove the viscosity subsolution case, and the viscosity supersolution
case can be proved analogously. First note that 
\eqref{conv compact}, \eqref{conv compact u}, and 
the condition that $u^{(n)}(T,x)=g^{(n)}(x)$ for all $x\in\bR^d$ and $n\in\bN$
guarantees that $u(T,x)=g(x)$ for $x\in\bR^d$.
Throughout this proof, let $(t_0,x_0)\in(0,T)\times\bR^d$,
and let $\varphi\in C^{1,2}((0,T)\times\bR^d)$ such that 
$\varphi(t_0,x_0)=u(t_0,x_0)$ and $\varphi(t,x)\geq u(t,x)$ for all 
$(t,x)\in[0,T]\times\bR^d$. 
For every $\varepsilon\in(0,\infty)$ define
$\varphi_\varepsilon\in C^{1,2}((0,T)\times\bR^d)$ by
\begin{equation}                              \label{def phi e}
\varphi_\varepsilon(t,x):=\varphi(t,x)+\frac{\varepsilon}{2}
(|t-t_0|^2+\|x-x_0\|^2),\quad (t,x)\in(0,T)\times\bR^d.
\end{equation}
Moreover, let $\xi:=\frac{(T-t_0)\wedge t_0}{2}$, and define the set
$$
\mathcal{K}_0:=\{(t,x)\in[0,T]\times\bR^d:\max\{|t-t_0|,\|x-x_0\|\}\leq \xi\}.
$$
According to \eqref{conv compact u}, we notice for each 
$\varepsilon\in(0,\infty)$  that there exists an integer $N_\varepsilon>0$
such that for all integers $n\geq N_\varepsilon$
$$                                     
\sup_{(t,x)\in\mathcal{K}_0}\Big(|u^{(n)}(t,x)-u(t,x)|\Big)
<\frac{\varepsilon\xi^2}{4}.
$$
Thus, by \eqref{def phi e} we obtain for all $\varepsilon\in(0,\infty)$,
$n\in\bN\cap[N_\varepsilon,\infty)$, and $(t,x)\in\mathcal{K}_0$
with $|t-t_0|^2+\|x-x_0\|^2\geq\xi^2$ that
\begin{align*}
u^{(n)}(t_0,x_0)-\varphi_{\varepsilon}(t_0,x_0)
&
=u^{(n)}(t_0,x_0)-\varphi(t_0,x_0)
=u^{(n)}(t_0,x_0)-u(t_0,x_0)>-\frac{\varepsilon\xi^2}{4}\\
&
>u^{(n)}(t,x)-u(t,x)-\frac{\varepsilon}{2}(|t-t_0|^2+\|x-x_0\|^2)\\
&
\geq u^{(n)}(t,x)-\varphi(t,x)-\frac{\varepsilon}{2}(|t-t_0|^2+\|x-x_0\|^2)\\
&
=u^{(n)}(t,x)-\varphi_\varepsilon(t,x).
\end{align*}
Since for every $\varepsilon\in(0,\infty)$ and $n\in\bN$
the function $u^{(n)}-\varphi_\varepsilon$ is continuous, we obtain
that for each $\varepsilon\in(0,\infty)$ there exist sequences
$\big\{t^{(\varepsilon)}_n\big\}_{n=1}^\infty\subseteq
\{t\in[0,T]:|t-t_0|\leq \xi\}$ 
and
$\big\{x^{(\varepsilon)}_n\big\}_{n=1}^\infty\subseteq
\{x\in\bR^d:\|x-x_0\|\leq \xi\}$ satisfying for all
$n\in\bN\cap[N_\varepsilon,\infty)$ and $(t,x)\in\mathcal{K}_0$ that
\begin{equation}                               \label{ineq n e}
u^{(n)}(t^{(\varepsilon)}_n,x^{(\varepsilon)}_n)
-\varphi_\varepsilon(t^{(\varepsilon)}_n,x^{(\varepsilon)}_n)
\geq u^{(n)}(t,x)-\varphi_\varepsilon(t,x).
\end{equation}
Thus, taking into account that for each $n\in\bN$ and $\varepsilon\in(0,\infty)$,
$\varphi_\varepsilon\in C^{1,2}((0,T)\times\bR^d)$,
and that $u^{(n)}$ is a viscosity subsolution of 
\eqref{linear PIDE n}, 
by Lemma \ref{Lemma equiv def} we have for all $\varepsilon\in(0,\infty)$
and $n\in\bN\cap[N_\varepsilon,\infty)$ that
\begin{align} 
&                                  
-\frac{\partial}{\partial t}\varphi_\varepsilon
(t^{(\varepsilon)}_n,x^{(\varepsilon)}_n)\nonumber\\
&
+G_0^{(n)}(t^{(\varepsilon)}_n,x^{(\varepsilon)}_n,
u^{(n)}(t^{(\varepsilon)}_n,x^{(\varepsilon)}_n),
\nabla_x \varphi_\varepsilon(t^{(\varepsilon)}_n,x^{(\varepsilon)}_n),
\operatorname{Hess}_x
\varphi_\varepsilon(t^{(\varepsilon)}_n,x^{(\varepsilon)}_n)
,\varphi_\varepsilon(t^{(\varepsilon)}_n,\cdot))
\leq 0.                                                           \label{sub n e}
\end{align}
Furthermore, by \eqref{conv compact u}, \eqref{def phi e}, and
\eqref{ineq n e} it holds for all $\varepsilon\in(0,\infty)$ that
\begin{align}
0 & =\limsup_{n\to\infty}\big[u(t_0,x_0)-u^{(n)}(t_0,x_0)\big]
\nonumber\\
&
=\limsup_{n\to\infty}\big[\varphi_\varepsilon(t_0,x_0)-u^{(n)}(t_0,x_0)\big]
\nonumber\\
&
\geq \limsup_{n\to\infty}\big[\varphi_\varepsilon
(t^{(\varepsilon)}_n,x^{(\varepsilon)}_n)
-u^{(n)}(t^{(\varepsilon)}_n,x^{(\varepsilon)}_n)\big]\nonumber\\
&
= \limsup_{n\to\infty}\left[\varphi(t^{(\varepsilon)}_n,x^{(\varepsilon)}_n)
+\frac{\varepsilon}{2}\Big(\big|t^{(\varepsilon)}_n-t_0\big|^2
+\big\|x^{(\varepsilon)}_n-x_0\big\|^2\Big)
-u^{(n)}(t^{(\varepsilon)}_n,x^{(\varepsilon)}_n)\right]\nonumber\\
&
\geq \limsup_{n\to\infty}\left[u(t^{(\varepsilon)}_n,x^{(\varepsilon)}_n)
-u^{(n)}(t^{(\varepsilon)}_n,x^{(\varepsilon)}_n)
+\frac{\varepsilon}{2}\Big(\big|t^{(\varepsilon)}_n-t_0\big|^2
+\big\|x^{(\varepsilon)}_n-x_0\big\|^2\Big)
\right]\nonumber\\
&
= \limsup_{n\to\infty}\left[
\frac{\varepsilon}{2}\Big(\big|t^{(\varepsilon)}_n-t_0\big|^2
+\big\|x^{(\varepsilon)}_n-x_0\big\|^2\Big)
\right].   
\nonumber                                          
\end{align} 
Therefore, we see that for all $\varepsilon\in(0,\infty)$ 
\begin{equation}                                               \label{conv t x}
\limsup_{n\to\infty}\left[
\frac{\varepsilon}{2}\Big(\big|t^{(\varepsilon)}_n-t_0\big|^2
+\big\|x^{(\varepsilon)}_n-x_0\big\|^2\Big)
\right]=0.
\end{equation}
Moreover, the continuity of $u$, \eqref{conv compact u}, and \eqref{conv t x}
ensure for all $\varepsilon\in(0,\infty)$ that
\begin{align}
\lim_{n\to\infty}\big|u^{(n)}(t^{(\varepsilon)}_n,x^{(\varepsilon)}_n)
-u(t_0,x_0)\big|
\leq
&
\lim_{n\to\infty}\big|u^{(n)}(t^{(\varepsilon)}_n,x^{(\varepsilon)}_n)
-u(t^{(\varepsilon)}_n,x^{(\varepsilon)}_n)\big|\nonumber\\
&
+
\lim_{n\to\infty}\big|u(t^{(\varepsilon)}_n,x^{(\varepsilon)}_n)
-u(t_0,x_0)\big|=0.
\label{conv u^n u}
\end{align}
Then by \eqref{def phi e}, \eqref{conv t x}, \eqref{conv u^n u},
Lemma \ref{lemma continuity G 0}, and Lemma \ref{lemma conv compact G}, 
it holds for all $\varepsilon\in(0,\infty)$
that
$$
\lim_{n\to\infty}\left|\frac{\partial}{\partial t}
\varphi_\varepsilon(t^{(\varepsilon)}_n,x^{(\varepsilon)}_n)
-\frac{\partial}{\partial t}\varphi(t_0,x_0)\right|=0,
$$
and
\begin{align*}
&
G_0\Big(t_0,x_0,
\varphi(t_0,x_0),
\nabla_x \varphi(t_0,x_0),
\operatorname{Hess}_x
\varphi(t_0,x_0)+\varepsilon\mathbf{I}_d
,\varphi(t_0,\cdot)\Big)\\
&
=
G_0\Big(t_0,x_0,
u(t_0,x_0),
\nabla_x \varphi_\varepsilon(t_0,x_0),
\operatorname{Hess}_x
\varphi_\varepsilon(t_0,x_0)
,\varphi_\varepsilon(t_0,\cdot)\Big)\\
&
= \lim_{n\to\infty}
G_0\Big(t^{(\varepsilon)}_n,x^{(\varepsilon)}_n,
u^{(n)}(t^{(\varepsilon)}_n,x^{(\varepsilon)}_n),
\nabla_x \varphi_\varepsilon(t^{(\varepsilon)}_n,x^{(\varepsilon)}_n),
\operatorname{Hess}_x
\varphi_\varepsilon(t^{(\varepsilon)}_n,x^{(\varepsilon)}_n)
,\varphi_\varepsilon(t^{(\varepsilon)}_n,\cdot)\Big)\\
&
=\lim_{n\to\infty}
G_0^{(n)}\Big(t^{(\varepsilon)}_n,x^{(\varepsilon)}_n,
u^{(n)}(t^{(\varepsilon)}_n,x^{(\varepsilon)}_n),
\nabla_x \varphi_\varepsilon(t^{(\varepsilon)}_n,x^{(\varepsilon)}_n),
\operatorname{Hess}_x
\varphi_\varepsilon(t^{(\varepsilon)}_n,x^{(\varepsilon)}_n)
,\varphi_\varepsilon(t^{(\varepsilon)}_n,\cdot)\Big).
\end{align*}
Hence, taking the limit as $n\to\infty$ in \eqref{sub n e}, and
using the assumption $u(t_0,x_0)=\varphi(t_0,x_0)$ yields that 
for all $\varepsilon\in(0,\infty)$
\begin{align*}                                  
-\frac{\partial}{\partial t}\varphi
(t_0,x_0)
+G_0\Big(t_0,x_0,
\varphi(t_0,x_0),
\nabla_x \varphi(t_0,x_0),
\operatorname{Hess}_x
\varphi(t_0,x_0)+\varepsilon\mathbf{I}_d
,\varphi(t_0,\cdot)\Big)
\leq 0.                                                           
\end{align*}
Lemma \ref{lemma continuity G 0} 
therefore ensures that 
\begin{align*}                                  
-\frac{\partial}{\partial t}\varphi
(t_0,x_0)
+G_0\Big(t_0,x_0,
\varphi(t_0,x_0),
\nabla_x \varphi(t_0,x_0),
\operatorname{Hess}_x
\varphi(t_0,x_0)
,\varphi(t_0,\cdot)\Big)
\leq 0,                                                           
\end{align*}
which shows that $u$ is a viscosity subsolution 
of \eqref{linear PIDE}. The proof of this lemma is thus completed.
\end{proof}

\begin{lemma}                        \label{lemma v solution linear}
Let Assumptions \ref{assumption Lip and growth}, \ref{assumption pointwise},
\ref{assumption h}, and \ref{assumption conv compact}-\ref{assumption eta n} hold.
Assume that $\mu$ and $\sigma$ are compactly
supported, and that $\eta_{\cdot}(\cdot,z)$ is compactly supported
for all $z\in\bR^d$.
Moreover, assume that there exists a non-empty compact set 
$\mathcal{K}_0\in\bR^d$ such that for all $n\in\bN$, $t\in[0,T]$, and $z\in\bR^d$
\begin{equation}
\operatorname{supp}\{\mu^{(n)}(t,\cdot)\}\cup
\operatorname{supp}\{\sigma^{(n)}(t,\cdot)\}\cup
\operatorname{supp}\{\eta_t^{(n)}(\cdot,z)\}
\subseteq \mathcal{K}_0,
\end{equation}
where $\operatorname{supp}\{\mathfrak{f}\}$ denotes the support of the real-valued
function $\mathfrak{f}$ on $\bR^d$. 
Furthermore, for every $n\in\bN$ and $(t,x)\in[0,T]\times\bR^d$, 
let $\big(X^{t,x,(n)}_s\big)_{s\in[t,T]}:[t,T]\times\Omega\to\bR^d$
be an $\bF$-adapted c\`adl\`ag process satisfying that
$X^{t,x,(n)}_t=x$ and almost surely for all $s\in[0,T]$
\begin{equation}                                               \label{SDE n}                              
dX^{t,x,(n)}_{s}
=\mu^{(n)}(s,X^{t,x,(n)}_{s-})\,ds
+\sigma^{(n)}(s,X^{t,x,(n)}_{s-})\,dW_s
+\int_{\bR^{d}}\eta^{(n)}_s(X^{t,x,(n)}_{s-},z)
\,\tilde{\pi}(dz,ds).
\end{equation}
Then there is a constants $K_0>0$ only depending on $T$ and $L$ 
such that for all $n\in\bN$ and $(t,x)\in[0,T]\times\bR^d$
\begin{align}                             
\bE\Big[\sup_{s\in[t,T]}\big\|X^{t,x,(n)}_s-X^{t,x}_s\big\|^2\Big]
\leq  
&
K_0\cdot
\Bigg[
\sup_{(r,y)\in[0,T]\times\bR^d}
\Bigg\{
\big\|\mu^{(n)}(r,y)-\mu(r,y)\big\|^2\nonumber\\
&
+\big\|\sigma^{(n)}(r,y)-\sigma(r,y)\big\|_F^2
+\int_{\bR^d}\big\|\eta^{(n)}_r(y,z)-\eta_r(y,z)\big\|^2\,\nu(dz)
\Bigg\}
\Bigg]
\label{sup e T n}
\end{align}
\end{lemma}

\begin{proof}
We first notice for each $(t,x)\in[0,T]\times\bR^d$, and $n\in\bN$
that it holds almost surely for all $s\in[t,T]$ that
\begin{align*}
X^{t,x,(n)}_s-X^{t,x}_s=
&
\int_t^s\left(\mu^{(n)}\big(r,X^{t,x,(n)}_{r-}\big)
-\mu\big(r,X^{t,x}_{r-}\big)
\right)dr\\
&
+\int_t^s\left(\sigma^{(n)}\big(r,X^{t,x,(n)}_{r-}\big)
-\sigma\big(r,X^{t,x}_{r-}\big)
\right)dW_r\\
&
+\int_t^s\int_{\bR^d}\left(\eta^{(n)}_r\big(X^{t,x,(n)}_{r-},z\big)
-\eta_r\big(X^{t,x}_{r-},z\big)
\right)\tilde{\pi}(dz,dr).
\end{align*}
Thus, by Jensen's inequality, Burkholder-Davis-Gundy inequality, 
and the fact that
$(\sum_{i=1}^3a_i)^2\leq 3\sum_{i=1}^3a_i^2$ for $a_i\in\bR$, $i=1,2,3$,
it holds for all $(t,x)\in[0,T]\times\bR^d$, $s\in[t,T]$, and $n\in\bN$ that
\begin{equation}                                   \label{e n s t x}
e^{(n)}_s(t,x):=\bE\Bigg[\sup_{r\in[t,s]}
\big\|X^{t,x,(n)}_s-X^{t,x}_s\big\|^2\Bigg]
\leq \sum_{i=1}^3A^{(n),i}_s(t,x),
\end{equation}
where 
\begin{align*}
&
A^{(n),1}_s(t,x):=3T\int_t^s\bE\Big[
\big\|
\mu^{(n)}\big(r,X^{t,x,(n)}_r\big)-\mu\big(r,X^{t,x}_r\big)
\big\|^2
\Big]\,dr,\\
&
A^{(n),2}_s(t,x):=24\int_t^s\bE\Big[
\big\|
\sigma^{(n)}\big(r,X^{t,x,(n)}_r\big)-\sigma\big(r,X^{t,x}_r\big)
\big\|_F^2
\Big]\,dr,\\
&
A^{(n),3}_s(t,x):=24\int_t^s\bE\left[\int_{\bR^d}
\big\|
\eta^{(n)}_r\big(X^{t,x,(n)}_r,z\big)-\eta_r\big(X^{t,x}_r,z\big)
\big\|^2\nu(dz)
\right]dr.
\end{align*}
We notice that Assumptions \ref{assumption Lip and growth}, 
\ref{assumption Lip and growth n}, and \ref{assumption eta n} implies 
for all $(t,x)\in[0,T]\times\bR^d$, $s\in[0,T]$,
and $n\in\bN$ that
\begin{align}
A^{(n),1}_s(t,x)
&
\leq 
6T\int_t^s\bE\left[
\big\|
\mu^{(n)}\big(r,X^{t,x,(n)}_r\big)-\mu\big(r,X^{t,x,(n)}_r\big)
\big\|^2
\right]dr\nonumber\\
& \quad
+6T\int_t^s\bE\left[
\big\|
\mu\big(r,X^{t,x,(n)}_r\big)-\mu\big(r,X^{t,x}_r\big)
\big\|^2
\right]dr\nonumber\\
&
\leq 6T^2\cdot\left[\sup_{(r,y)\in[0,T]\times\bR^d}
\big\|\mu^{(n)}(r,y)-\mu(r,y)\big\|^2
\right]
+6TL
\int_t^se^{(n)}_r(t,x)\,dr,                           \label{n A 1}
\end{align}
\begin{equation}
A^{(n),2}_s(t,x)\leq 48T\cdot\left[\sup_{(r,y)\in[0,T]\times\bR^d}
\big\|\sigma^{(n)}(r,y)-\sigma(r,y)\big\|_F^2
\right]
+48L\int_t^se^{(n)}_r(t,x)\,dr,                           \label{n A 2}
\end{equation}
and
\begin{equation}
A^{(n),3}_s(t,x)\leq 48T\cdot\left[\sup_{(r,y)\in[0,T]\times\bR^d}
\int_{\bR^d}\big\|\eta^{(n)}_r(y,z)-\eta_r(y,z)\big\|^2\,\nu(dz)
\right]
+48L\int_t^se^{(n)}_r(t,x)\,dr.                           \label{n A 3}
\end{equation}
Moreover, by \eqref{SDE moment est}, and Assumptions 
\ref{assumption Lip and growth n} and \ref{assumption eta n} 
it holds for all $(t,x)\in[0,T]\times\bR^d$ and
$n\in\bN$ that 
\begin{equation}                                      \label{SDE moment est n}
\bE\left[\sup_{s\in[t,T]}\big\|X^{t,x,(n)}_s\big\|^2\right]< \infty.
\end{equation}
Then by \eqref{SDE moment est}, \eqref{e n s t x}-\eqref{SDE moment est n},
and Gr\"onwall's lemma we obtain for all $(t,x)\in[0,T]\times\bR^d$ and
$n\in\bN$ that
\begin{align*}
e^{(n)}_T(t,x)\leq & 
6T(T+16)\exp\{6L(T+16)(T-t)\}\cdot
\Bigg[
\sup_{(r,y)\in[0,T]\times\bR^d}
\Bigg\{
\big\|\mu^{(n)}(r,y)-\mu(r,y)\big\|^2\\
&
+\big\|\sigma^{(n)}(r,y)-\sigma(r,y)\big\|_F^2
+\int_{\bR^d}\big\|\eta^{(n)}_r(y,z)-\eta_r(y,z)\big\|^2\,\nu(dz)
\Bigg\}
\Bigg].
\end{align*}
The proof of this lemma is thus completed.
\end{proof}

\begin{proposition}                               \label{proposition smooth PIDE}
Let Assumptions \ref{assumption Lip and growth},
\ref{assumption pointwise}, \ref{assumption jacobian}, 
and \ref{assumption h} hold.
Moreover, we assume that $g$, $\mu$, $\sigma$, and $h$ are compactly
supported, and that $\eta_{\cdot}(\cdot,z)$ is compactly supported
for all $z\in\bR^d$.
Let $u:[0,T]\times\bR^d\to\bR^d$ satisfy for all 
$(t,x)\in[0,T]\times\bR^d$ that
\begin{equation}                                           \label{FK linear 1}
u(t,x)=\bE\Bigg[g(X^{t,x}_T)+\int_t^Th(s,X^{t,x}_s)\,ds\Bigg].
\end{equation}
Then $u$ is a viscosity solution 
of the linear PIDE \eqref{linear PIDE} with $u(T,x)=g(x)$ for $x\in\bR^d$.
\end{proposition}

To prove the above proposition, we recall the notion of mollifications
of functions. 
For $\varepsilon>0$ and 
locally integrable functions $\phi:\bR^d\to\bR$ we use the notation 
$\phi^{(\varepsilon)}$ for the mollification of $\phi$, defined by 
\begin{equation}                                       \label{def mollification}
\phi^{(\varepsilon)}(x)
:=\varepsilon^{-d}\int_{\bR^d}\phi(y)k((x-y)/\varepsilon)\,dy
=\int_{\bR^d}\phi(x-\varepsilon z)k(z)\,dz,\quad x\in\bR^d , 
\end{equation}
where $k:\bR^d\to\bR$ is a fixed nonnegative smooth function on $\bR^d$ 
such that $k(x)=0$ for $|x|\geq1$, $k(-x)=k(x)$ 
for $x\in\bR^d$, and $\int_{\bR^d}k(x)\,dx=1$. 

\begin{proof}[Proof of Proposition \ref{proposition smooth PIDE}]
Let $\{\varepsilon_n\}_{n=1}^\infty\subseteq (0,1]$ 
such that $\lim_{n\to\infty}\varepsilon_n=0$.
Then for each $n\in\bN$ we denote by $g^{(\varepsilon_n)}$, $h^{(\varepsilon_n)}$,
$\mu^{(\varepsilon_n)}$, $\sigma^{(\varepsilon_n)}$, and $\eta^{(\varepsilon_n)}$
the mollifications (in $x\in\bR^d$) of $g$, $h$, $\mu$, $\sigma$, and $\eta$,
respectively.
For every $n\in\bN$, 
let $\big(X^{t,x,(n)}_s\big)_{s\in[t,T]}:[t,T]\times\Omega\to\bR^d$
be an $\bF$-adapted c\`adl\`ag process satisfying that
$X^{t,x,(n)}_t=x$ and almost surely for all $s\in[0,T]$
\begin{equation}                                          \label{SDE epsilon n}                              
dX^{t,x,(n)}_{s}
=\mu^{(\varepsilon_n)}(s,X^{t,x,(n)}_{s-})\,ds
+\sigma^{(\varepsilon_n)}(s,X^{t,x,(n)}_{s-})\,dW_s
+\int_{\bR^{d}}\eta^{(\varepsilon_n)}_s(X^{t,x,(n)}_{s-},z)
\,\tilde{\pi}(dz,ds).
\end{equation}
For each $n\in\bN$, $k\in\bN$, and $(t,x)\in[0,T]\times\bR^d$, we
define 
\begin{equation}                           \label{def u n k epsilon}
u^{n,k}(t,x):=\bE\left[
g^{(\varepsilon_k)}\big(X^{t,x,(n)}_T\big)
+\int_t^Th^{(\varepsilon_k)}\big(s,X^{t,x,(n)}_s\big)\,ds
\right],
\end{equation}
and
\begin{equation}                            \label{def u 0 k epsilon}
u^{0,k}(t,x):=\bE\left[
g^{(\varepsilon_k)}\big(X^{t,x}_T\big)
+\int_t^Th^{(\varepsilon_k)}\big(s,X^{t,x}_s\big)\,ds
\right].
\end{equation}
We first notice that by \eqref{def mollification} 
it holds for all $k\in\bN$ and $(t,x)\in[0,T]\times\bR^d$ that
\begin{align*}
\big|h^{(\varepsilon_k)}(t,x)-h(t,x)\big|
&
=\varepsilon_k^{-d}\left|\int_{\bR^d}\left[h(t,y)-h(t,x)\right]
k((x-y)/\varepsilon_k)\,dy\right|\\
&
\leq \left[\sup_{s\in[0,T]}\sup_{\|y-y'\|\leq\varepsilon_k}
\left(
|h(t,y)-h(t,y')|
\right)
\right]
\cdot
\int_{\bR^d}\varepsilon_k^{-d}k((x-y)/\varepsilon_k)\,dy\\
&
=\sup_{s\in[0,T]}\sup_{\|y-y'\|\leq\varepsilon_k}
\left(
|h(s,y)-h(s,y')|
\right).
\end{align*}
This together with the assumption that $h$ is compactly supported imply 
\begin{equation}                                    \label{conv h n}
\lim_{k\to\infty}\left[\sup_{(t,x)\in[0,T]\times\bR^d}
\big|h^{(\varepsilon_k)}(t,x)-h(t,x)\big|\right]=0.
\end{equation}
Moreover, by \eqref{assumption Lip mu sigma eta}, \eqref{def mollification}, 
and Jensen's inequality we first observe
for all $(t,x)\in[0,T]\times\bR^d$,$n\in\bN$, and $z\in\bR^d$ that
\begin{align}
\big\|\mu^{(\varepsilon_n)}(t,x)-\mu(t,x)\big\|^2
&
=\left\|
\int_{\bR^d}\mu(t,x-\varepsilon_ny)k(y)\,dy
-\int_{\bR^d}\mu(t,x)k(y)\,dy
\right\|^2\nonumber\\
&
\leq \int_{\bR^d}\big\|
\mu(t,x-\varepsilon_ny)-\mu(t,x)
\big\|^2k(y)\,dy\leq\varepsilon_n^2L,                          \label{coca 1}
\end{align}
\begin{equation}                                            \label{coca 2}
\big\|\sigma^{(\varepsilon_n)}(t,x)-\sigma(t,x)\big\|_F^2
\leq\varepsilon_n^2L,  
\quad
\big\|\eta^{(\varepsilon_n)}_t(x,z)-\eta_t(x,z)\big\|^2
\leq\varepsilon_n^2C_d(1\wedge\|z\|^2),                      
\end{equation}
and
$$
\big|g^{(\varepsilon_n)}(x)-g(x)\big|^2\leq \varepsilon_n^2LT^{-1}.
$$
Hence we have
\begin{equation}                                       \label{compact conv mu n}
\lim_{n\to\infty}\Bigg[
\sup_{(t,x)\in[0,T]\times\bR^d}\Big(
\big\|
\mu^{(\varepsilon_n)}(t,x)-\mu(t,x)
\big\|^2
+
\big\|
\sigma^{(\varepsilon_n)}(t,x)-\sigma(t,x)
\big\|^2
+\big|g^{(\varepsilon_n)}(x)-g(x)\big|^2\Big)
\Bigg]=0,
\end{equation}
and
\begin{equation}                                 \label{compact conv eta n}
\lim_{n\to\infty}\left[
\sup_{(t,x)\in[0,T]\times\bR^d}
\big\|
\eta^{(\varepsilon_n)}_t(x,z)-\eta_t(x,z)
\big\|^2
\right]=0 \quad \text{for all } z\in\bR^d.
\end{equation}
Hence by \eqref{conv h n} and analogous argument 
to obtain \eqref{compact conv mu n} and \eqref{compact conv eta n}, 
we have that Assumption \ref{assumption conv compact} holds
with $g^{(\varepsilon_n)}$, $h^{(\varepsilon_n)}$, $\mu^{(\varepsilon_n)}$,
$\sigma^{(\varepsilon_n)}$, and $\eta^{(\varepsilon_n)}$ in place of
$g^{(n)}$, $h^{(n)}$, $\mu^{(n)}$,
$\sigma^{(n)}$, and $\eta^{(n)}$, respectively.
Moreover, by \eqref{assumption Lip f g}, 
\eqref{assumption growth}, \eqref{assumption eta},
and Jensen's inequality we have for all $t\in[0,T]$, $x,x'\in\bR^d$, 
$z\in\bR^d$, and $n\in\bN$ that
\begin{align*}
\big\|\eta^{(\varepsilon_n)}_t(x,z)-\eta^{(\varepsilon_n)}_t(x',z)\big\|^2
&
=\left\|
\int_{\bR^d}\eta_t(x-\varepsilon_ny,z)k(y)\,dy
-\int_{\bR^d}\eta_t(x'-\varepsilon_ny,z)k(y)\,dy
\right\|^2\\
&
\leq \int_{\bR^d}\big\|
\eta_t(x-\varepsilon_ny,z)-\eta_t(x'-\varepsilon_ny,z)
\big\|^2k(y)\,dy\\
&
\leq C_d(1\wedge\|z\|^2)\|x-x'\|^2\int_{\bR^d}k(y)\,dy\\
&
=C_d(1\wedge\|z\|^2)\|x-x'\|^2,
\end{align*}
\begin{equation}                        \label{Lip g epsilon}
\big\|g^{(\varepsilon_n)}(x)-g^{(\varepsilon_n)}(x')\big\|^2
\leq LT^{-1}\|x-x'\|^2,
\end{equation}
\begin{align}
\big\|\eta^{(\varepsilon_n)}_t(x,z)\big\|^2
&
\leq \int_{\bR^d}\|\eta_t(x-\varepsilon_ny,z)\|^2k(y)\,dy
\leq C_d(1\wedge\|z\|^2)\int_{\bR^d}k(y)\,dy\nonumber\\
&
= C_d(1\wedge\|z\|^2),
\label{bbd eta epsilon n}
\end{align}
\begin{align}
\big\|g^{(\varepsilon_n)}(t,x)\big\|^2
&
\leq \int_{\bR^d}\|g(t,x-\varepsilon_ny)\|^2k(y)\,dy
\leq \int_{\bR^d}L(d^p+\|x-\varepsilon_ny\|^2)k(y)\,dy\nonumber\\
&
\leq L(d^p+2\|x\|^2)+\int_{\bR^d}2L\|\varepsilon_ny\|^2k(y)\,dy\nonumber\\
&
\leq 2L(d^p+1)(1+\|x\|^2),                        \label{growth g epsilon}
\end{align}
and
\begin{equation}                                   \label{growth h epsilon}
\big\|h^{(\varepsilon_n)}(t,x)\big\|^2\leq 2(C_0+1)(1+\|x\|^2).
\end{equation}
Thus, by analogous argument also for $\mu^{(\varepsilon_n)}$ 
and $\sigma^{(\varepsilon_n)}$
we get that Assumptions 
\ref{assumption Lip and growth n} and \ref{assumption eta n} hold with
$h^{(\varepsilon_n)}$, $g^{(\varepsilon_n)}$, $\mu^{(\varepsilon_n)}$, 
$\sigma^{(\varepsilon_n)}$, and $\eta^{(\varepsilon_n)}$ in place of
$h^{(n)}$, $g^{(n)}$, $\mu^{(n)}$, $\sigma^{(n)}$, and $\eta^{(n)}$, respectively.
Furthermore, by virtue of Corollary 3.6 in \cite{DGW2020}
there exist constants
$\delta'>0$ and $\varepsilon_0>0$ such that for all 
$\varepsilon\in(0,\varepsilon_0)$,
$(t,x)\in[0,T]\times\bR^d$, $z\in\bR^d$, and $\gamma\in[0,1]$
$$                                     
\delta'\leq \big|\det\big(\mathbf{I}_d
+\gamma D_x \eta^{(\varepsilon)}_t(x,z)\big)\big|.
$$
This verifies that
Assumption \ref{assumption jacobian n} holds for sufficiently large $n$,
with $\eta^{(\varepsilon_n)}$ in place of $\eta^{(n)}$.
Moreover, the application of Lemma 
\ref{lemma v solution linear} yields for all $n\in\bN$ 
and $(t,x)\in[0,T]\times\bR^d$ that 
\begin{align}                             
\bE\Big[\sup_{s\in[t,T]}\big\|X^{t,x,(n)}_s-X^{t,x}_s\big\|^2\Big]
\leq  
&
c_0\cdot
\Bigg[
\sup_{(r,y)\in[0,T]\times\bR^d}
\Bigg\{
\big\|\mu^{(\varepsilon_n)}(r,y)-\mu(r,y)\big\|^2\nonumber\\
&
+\big\|\sigma^{(\varepsilon_n)}(r,y)-\sigma(r,y)\big\|_F^2
+\int_{\bR^d}\big\|\eta^{(\varepsilon_n)}_r(y,z)-\eta_r(y,z)\big\|^2\,\nu(dz)
\Bigg\}
\Bigg],
\label{sup e epsilon T n}
\end{align}
where $c_0$ a positive constant only depending on $T$ and $L$.
Hence, by \eqref{int z & q bound} \eqref{coca 1} and \eqref{coca 2}
we obtain for all $n\in\bN$ that
\begin{align}
\sup_{(t,x)\in[0,T]\times\bR^d}
\bE\Big[\sup_{s\in[t,T]}\big\|X^{t,x,(n)}_s-X^{t,x}_s\big\|^2\Big]
&
\leq  
c_0\cdot\left[
2\varepsilon_n^2L+\varepsilon_n^2C_d\int_{\bR^d}(1\wedge\|z\|^2)\,\nu(dz)
\right]
\nonumber\\
&
\leq
c_0\varepsilon_n^2(2L+C_dKd^p), 
\label{est X n X}
\end{align}
which implies
\begin{equation}                                \label{uni conv X n}
\lim_{n\to \infty}\left(
\sup_{(t,x)\in[0,T]\times\bR^d}
\bE\Big[\sup_{s\in[t,T]}\big\|X^{t,x,(n)}_s-X^{t,x}_s\big\|^2\Big]
\right)=0.
\end{equation}
Furthermore, by Assumptions
\ref{assumption Lip and growth n}-\ref{assumption jacobian n},
applying Lemma \ref{lemma smooth PIDE}
(with $\mu\cal\mu^{(\varepsilon_n)}$, 
$\sigma\cal\sigma^{(\varepsilon_n)}$, 
$\eta\cal\eta^{(\varepsilon_n)}$,
$g\cal g^{(\varepsilon_k)}$, 
$h\cal h^{(\varepsilon_k)}$ for $n,k\in\bN$ in the notation of Proposition
\ref{lemma smooth PIDE}) yields that for all $k\in\bN$ and sufficiently large
$n\in\bN$
the function
$u^{n,k}\in C^{1,2}([0,T]\times\bR^d)$ is a classical solution of
\begin{align}
&
\frac{\partial}{\partial t}u^{n,k}(t,x)
+\langle\nabla_xu^{n,k}(t,x),\mu^{(\varepsilon_n)}(t,x)\rangle
+\frac{1}{2}\operatorname{Trace}
\left(\sigma^{(\varepsilon_n)}(t,x)[\sigma^{(\varepsilon_n)}(t,x)]^T
\operatorname{Hess}_xu^{n,k}(t,x)\right)\nonumber\\
&
+h^{(\varepsilon_k)}(t,x)+\int_{\bR^d}\left(
u^{n,k}(t,x+\eta^{(\varepsilon_n)}_t(x,z))
-u^{n,k}(t,x)-\langle\nabla_xu^{n,k}(t,x),
\eta^{(n)}_t(x,z)\rangle\right)\nu(dz)=0               \label{PIDE n k 00}
\end{align}
on $(0,T)\times \bR^d$ with $u^{n,k}(T,x)=g^{(\varepsilon_k)}(x)$ on $\bR^d$.

Next, by well-known properties of derivatives of mollifications
(see, e.g., Theorem 2.29 in \cite{AF})
and change of variables, we have for every 
$k\in\bN$, $i\in \{1,...,d\}$, and $(t,x)\in[0,T]\times\bR^d$ that
\begin{align}
\frac{\partial}{\partial x_i}h^{(\varepsilon_k)}(t,x)
&
=\varepsilon_k^{-(d+1)}
\int_{\bR^d}h(t,y)\Big(\frac{\partial}{\partial x_i}k\Big)
((x-y)/\varepsilon_k)\,dy
\nonumber\\
&
=\varepsilon_k^{-1}\int_{\bR^d}h(t,x-\varepsilon_kz)
\frac{\partial}{\partial x_i}k(z)\,dz.
\nonumber\\                  
\end{align}
This together with the fact that the functions $h$ and $k$ are compactly
supported and continuous imply for all $k\in\bN$ and $i\in \{1,...,d\}$ that
$\frac{\partial}{\partial x_i}h^{(\varepsilon_k)}\in C([0,T]\times\bR^d)$
and $\frac{\partial}{\partial x_i}h^{(\varepsilon_k)}$ is compactly supported.
Therefore, we obtain that for all $k\in\bN$
\begin{equation}                                     \label{est gradient h k}
\sup_{(t,x)\in[0,T]\times\bR^d}\|\nabla_xh^{(\varepsilon_k)}(t,x)\|<\infty.
\end{equation}
Furthermore, by \eqref{def u n k epsilon}, \eqref{def u 0 k epsilon},  
\eqref{Lip g epsilon}, and the mean value theorem
it holds for all $n,k\in\bN$ and $x\in\bR^d$ that
\begin{align} 
&                                  
\big|u^{n,k}(t,x)-u^{0,k}(t,x)\big|\nonumber\\
&
\leq
\bE\left[\big|g^{(\varepsilon_k)}
\big(X^{t,x,(n)}_T\big)-g^{(\varepsilon_k)}\big(X^{t,x}_T\big)\big|\right]
+\int_t^T\bE\left[\big|
h^{(\varepsilon_k)}(s,X^{t,x,(n)}_s)-h^{(\varepsilon_k)}(s,X^{t,x}_s)
\big|\right]ds
\nonumber\\
&
\leq \left(
\bE\left[LT^{-1}\big\|X^{t,x,(n)}_T-X^{t,x}_T\big\|^2\right]
\right)^{1/2}\nonumber\\
& \quad
+\int_t^T\bE\left[\left(\sup_{(r,y)\in[0,T]\times\bR^d}
\|\nabla_yh^{(\varepsilon_k)}(r,y)\|\right)\cdot
\big\|
X^{t,x,(n)}_s-X^{t,x}_s
\big\|
\right]ds\nonumber\\
&
\leq \left(
\bE\left[LT^{-1}\big\|X^{t,x,(n)}_T-X^{t,x}_T\big\|^2\right]
\right)^{1/2}\nonumber\\
& \quad
+T\left(\sup_{(r,y)\in[0,T]\times\bR^d}
\|\nabla_yh^{(\varepsilon_k)}(r,y)\|\right)\cdot\left(
\bE\Big[\sup_{s\in[t,T]}\big\|X^{t,x,(n)}_s-X^{t,x}_s\big\|^2\Big]
\right)^{1/2}
\label{beer 0}
\end{align}
This together with \eqref{uni conv X n} and \eqref{est gradient h k} show 
for every $k\in\bN$ that
\begin{equation}                              \label{uni conv u n k u 0 k}                                   
\lim_{n\to\infty}\Bigg[
\sup_{(t,x)\in[0,T]\times\bR^d}
\big|u^{n,k}(t,x)-u^{0,k}(t,x)\big|\Bigg]= 0.                                                                                                                    
\end{equation}
Furthermore, by \eqref{SDE moment est}, \eqref{FK linear 1}, 
Assumptions \ref{assumption Lip and growth} and \ref{assumption h}, 
and H\"older's inequality
we have for all $(t,x)\in[0,T]$ that
\begin{align*}
|u(t,x)|
&
\leq \left(\bE\left[\big|g(X^{t,x}_T)\big|^2\right]\right)^{1/2}
+\int_t^T\left(\bE\left[\big|h(s,X^{t,x}_s)\big|^2\right]\right)^{1/2}ds
\nonumber\\
&
\leq\left(\bE\left[L\Big(d^p+\big\|X^{t,x}_T\big\|^2\Big)\right]\right)^{1/2}
+\int_t^T
\left(\bE\left[C_0\Big(1+\big\|X^{t,x}_s\big\|^2\Big)\right]\right)^{1/2}ds
\nonumber\\
&
\leq \left[Ld^p+LC(1+\|x\|^2)\right]^{1/2}
+T\left[C_0+C_0C(1+\|x\|^2)\right]^{1/2}.
\end{align*}
This implies that
\begin{equation}                            \label{linear growth u}
\sup_{(t,x)\in[0,T]\times\bR^d}\frac{|u(t,x)|}{1+\|x\|}<\infty.
\end{equation}
Similarly, using Assumption \ref{assumption Lip and growth n} 
we also obtain for all $k,n\in\bN$ that
\begin{equation}                                \label{linear growth u n k}
\sup_{(t,x)\in[0,T]\times\bR^d}
\frac{|u^{n,k}(t,x)|+|u^{0,k}(t,x)|}{1+\|x\|}<\infty.
\end{equation}
Hence, by \eqref{uni conv u n k u 0 k} and 
the fact that $u^{n,k}\in C_1([0,T]\times\bR^d)$ for
all $k\in\bN$ and sufficiently large $n\in\bN$ shows that $u^{0,k}\in C_1([0,T]\times\bR^d,\bR)$ for all $k\in\bN$.
Therefore, by \eqref{PIDE n k 00}, \eqref{uni conv u n k u 0 k}, 
and Assumptions \ref{assumption conv compact}-\ref{assumption eta n},
applying Lemma
\ref{lemma subsolution} yields for all $k\in\bN$ that 
the function 
$u^{0,k}$ is a viscosity solution of
\begin{align}
&
\frac{\partial}{\partial t}u^{0,k}(t,x)
+\langle\nabla_xu^{0,k}(t,x),\mu(t,x)\rangle
+\frac{1}{2}\operatorname{Trace}
\left(\sigma(t,x)[\sigma(t,x)]^T
\operatorname{Hess}_xu^{0,k}(t,x)\right)\nonumber\\
&
+h^{(\varepsilon_k)}(t,x)+\int_{\bR^d}\left(
u^{0,k}(t,x+\eta_t(x,z)-u^{0,k}(t,x)-\langle\nabla_xu^{0,k}(t,x),
\eta_t(x,z)\rangle\right)\nu(dz)=0               \label{PIDE 0 k 00}
\end{align}
on $(0,T)\times \bR^d$ with $u^{0,k}(T,x)=g^{(\varepsilon_k)}(x)$ on $\bR^d$.
Furthermore, by \eqref{FK linear 1} and \eqref{def u 0 k epsilon} we obtain
for all $k\in\bN$ and $(t,x)\in[0,T]\times\bR^d$ that
\begin{align*}
\big|u(t,x)-u^{0,k}(t,x)\big|
&
\leq \bE\left[\big|g^{(\varepsilon_k)}(X^{t,x}_T)-g(X^{t,x}_T)\big|\right]
+\int_t^T\bE\left[\big|h^{(\varepsilon_k)}(s,X^{t,x}_s)
-h(s,X^{t,x}_s)\big|\right]ds\\
&
\leq \sup_{y\in\bR^d}
\left\{
\big|g^{(\varepsilon_k)}(y)-g(y)\big|
\right\}
+T\cdot\sup_{(r,y)\in[0,T]\times\bR^d}
\left\{
\big|h^{(\varepsilon_k)}(s,y)-h(s,y)\big|
\right\}.
\end{align*}
This together with \eqref{conv h n} and \eqref{compact conv mu n} imply that
\begin{equation}                                          \label{conv u 0 k}
\lim_{n\to\infty}
\sup_{(t,x)\in[0,T]\times\bR^d}\left\{\big|u(t,x)-u^{0,k}(t,x)\big|\right\}=0.
\end{equation}
Hence, \eqref{linear growth u n k} and 
the fact that $u^{0,k}\in C_1([0,T]\times\bR^d)$ for all $k\in\bN$ 
ensures $u\in C_1([0,T]\times\bR^d)$.
Therefore, by \eqref{conv h n}, \eqref{compact conv mu n},
\eqref{PIDE 0 k 00}, 
Assumptions \ref{assumption Lip and growth}-\ref{assumption jacobian},
and Assumption \ref{assumption h},
the application of Lemma \ref{lemma subsolution}
yields that $u$ is a viscosity solution of
\eqref{linear PIDE} with $u(T,x)=g(x)$ for all $x\in\bR^d$.
\end{proof}

\begin{lemma}                                      \label{lemma two SDEs}
Let $B\subseteq \bR^d$ be a closed set.
Let $\bar{\mu}\in C([0,T]\times\bR^d,\bR^d)$, 
$\bar{\sigma}\in C([0,T]\times\bR^d,\bR^{d\times d})$, and
$\bar{\eta}\in C([0,T]\times\bR^d\times\bR^d,\bR^d)$ such that
for all $(t,x)\in[0,T]\times B$ and $z\in\bR^d$
\begin{equation}                                        \label{cond B}
\bar{\mu}(t,x)=\mu(t,x),\quad \bar{\sigma}(t,x)=\sigma(t,x), 
\quad \bar{\eta}_t(x,z)=\eta_t(x,z).
\end{equation}
Assume that $\bar{\mu}$, $\bar{\sigma}$, and $\bar{\eta}$ satisfy
Assumptions \ref{assumption Lip and growth} and \ref{assumption pointwise}.
Moreover, for each $(t,x)\in[0,T]\times\bR^d$ 
let $\big(\bar{X}^{t,x}_s\big)_{s\in[t,T]}:[t,T]\times\Omega\to\bR^d$ be an 
$\bF$-adapted c\`adl\`ag process satisfying that $\bar{X}^{t,x}_t=x$,
and almost surely for all $s\in[t,T]$ 
$$                                             
d\bar{X}^{t,x}_{s}
=\bar{\mu}\big(s,\bar{X}^{t,x}_{s-}\big)\,ds
+\bar{\sigma}\big(s,\bar{X}^{t,x}_{s-}\big)\,dW_s
+\int_{\bR^{d}}\bar{\eta}_s\big(\bar{X}^{t,x}_{s-},z\big)
\,\tilde{\pi}(dz,ds).
$$
For each $(t,x)\in[0,T]\times\bR^d$, let $\tau^{t,x}:\Omega\to[0,T]$ be a stopping
time defined by
\begin{equation}                                     \label{def tau B}
\tau^{t,x}:=\inf
\big\{
s\geq t: X^{t,x}_s\notin B \; \text{or} \; \bar{X}^{t,x}_s\notin B
\big\}
\wedge T.
\end{equation}
Then it holds for all $(t,x)\in[0,T]\times\bR^d$ that
\begin{equation}                                               \label{prob 1 tau}
\bP\left[
\mathbf{1}_{\{s\leq\tau^{t,x}\}}\big\|X^{t,x}_s-\bar{X}^{t,x}_s\big\|=0
\; \text{for} \; \text{all} \; s\in[t,T]
\right]=1.
\end{equation}
\end{lemma}

\begin{proof}
Throughout this proof, we fix $(t,x)\in[0,T]\times\bR^d$, and use the notation
$\tau=\tau^{t,x}$. We first notice that almost surely for all $s\in[t,T]$
\begin{align*}
\bar{X}^{t,x}_{s\wedge\tau}-X^{t,x}_{s\wedge\tau}
&
=\int_t^{s\wedge\tau}\left(
\bar{\mu}\big(r,\bar{X}^{t,x}_{r-}\big)-\mu\big(r,X^{t,x}_{r-}\big)
\right)dr
+\int_t^{s\wedge\tau}\left(
\bar{\sigma}\big(r,\bar{X}^{t,x}_{r-}\big)-\sigma\big(r,X^{t,x}_{r-}\big)
\right)dW_r\\
& \quad
+\int_t^{s\wedge\tau}\int_{\bR^d}\left(
\bar{\eta}_r\big(\bar{X}^{t,x}_{r-},z\big)-\eta_r\big(X^{t,x}_{r-},z\big)
\right)\tilde{\pi}(dz,dr).
\end{align*}
Hence, by Minkowski's inequality, Jensen's inequality, It\^o's isometry,
and \eqref{cond B}
we have for all $s\in[t,T]$ that 
\begin{align*}
\left(\bE\left[\big\|
\bar{X}^{t,x}_{s\wedge \tau}-X^{t,x}_{s\wedge\tau}
\big\|^2\right]\right)^{1/2}
&
\leq 
\left(
\int_t^s(s-t)\bE\left[
\mathbf{1}_{\{r\leq\tau\}}\big\|
\mu\big(r,\bar{X}^{t,x}_r\big)-\mu\big(r,X^{t,x}_r\big)
\big\|^2
\right]dr
\right)^{1/2}\\
& \quad
+ \left(
\int_t^s\bE\left[
\mathbf{1}_{\{r\leq\tau\}}\big\|
\sigma\big(r,\bar{X}^{t,x}_r\big)-\sigma\big(r,X^{t,x}_r\big)
\big\|_F^2
\right]dr
\right)^{1/2}\\
& \quad
+\left(
\int_t^s\bE\left[\int_{\bR^d}
\mathbf{1}_{\{r\leq\tau\}}\big\|
\eta_r\big(\bar{X}^{t,x}_r,z\big)-\eta_r\big(X^{t,x}_r,z\big)
\big\|^2\,\nu(dz)
\right]dr
\right)^{1/2}\\
&
\leq 
\left(
\int_t^s(s-t)\bE\left[
\big\|
\mu\big(r,\bar{X}^{t,x}_{r\wedge\tau}\big)-\mu\big(r,X^{t,x}_{r\wedge\tau}\big)
\big\|^2
\right]dr
\right)^{1/2}\\
& \quad
+ \left(
\int_t^s\bE\left[
\big\|
\sigma\big(r,\bar{X}^{t,x}_{r\wedge\tau}\big)
-\sigma\big(r,X^{t,x}_{r\wedge\tau}\big)
\big\|_F^2
\right]dr
\right)^{1/2}\\
& \quad
+\left(
\int_t^s\bE\left[\int_{\bR^d}
\big\|
\eta_r\big(\bar{X}^{t,x}_{r\wedge\tau},z\big)
-\eta_{r}\big(X^{t,x}_{r\wedge\tau},z\big)
\big\|^2\,\nu(dz)
\right]dr
\right)^{1/2}.
\end{align*}
Thus, by \eqref{assumption Lip mu sigma eta} and the fact that
$(a+b+c)^2\leq 3(a^2+b^2+c^2)$ for $a,b,c\in\bR$, we obtain for all
$s\in[t,T]$ that
\begin{equation}                                 \label{bar difference}
\bE\left[\big\|\bar{X}^{t,x}_{s\wedge\tau}-X^{t,x}_{s\wedge\tau}\big\|^2\right]
\leq
3(T+2)L\int_t^s\bE\left[\big\|\bar{X}^{t,x}_{r\wedge\tau}
-X^{t,x}_{r\wedge\tau}\big\|^2\right]dr.
\end{equation}
Moreover, by \eqref{SDE moment est} it holds that
\begin{equation}                                \label{bar est finite}
\bE\Big[\sup_{s\in[t,T]}\big\|X^{t,x}_s\big\|^2\Big]+
\bE\Big[\sup_{s\in[t,T]}\big\|\bar{X}^{t,x}_s\big\|^2\Big]<\infty.
\end{equation}
Then by \eqref{bar difference}, \eqref{bar est finite},
and Gr\"onwall's lemma we have for every $s\in[t,T]$ that
$$
\bE\left[\big\|\bar{X}^{t,x}_{s\wedge\tau}-X^{t,x}_{s\wedge\tau}\big\|^2\right]
=0.
$$
Hence, considering that $\big(\bar{X}^{t,x}_s\big)_{s\in[t,T]}$ and
$\big(X^{t,x}_s\big)_{s\in[t,T]}$ have c\`adl\`ag sample paths,
we obtain \eqref{prob 1 tau}. Thus, the proof of this lemma is completed.
\end{proof}

\begin{proposition}                            \label{proposition linear PIDE}
Let Assumptions \ref{assumption Lip and growth},
\ref{assumption pointwise}, \ref{assumption jacobian}, 
and \ref{assumption h} hold,
and let $u:[0,T]\times\bR^d\to\bR^d$ satisfy for all 
$(t,x)\in[0,T]\times\bR^d$ that
\begin{equation}                                           \label{FK linear 2}
u(t,x)=\bE\Bigg[g(X^{t,x}_T)+\int_t^Th(s,X^{t,x}_s)\,ds\Bigg].
\end{equation}
Then $u$ is a viscosity solution 
of the linear PIDE \eqref{linear PIDE} with $u(T,x)=g(x)$ for all $x\in\bR^d$.
\end{proposition}

\begin{proof}
Throughout this proof, let $\chi\in C^\infty_c(\bR^d)$ such that 
$\chi(x)=1$ for $\|x\|\leq 1$, 
$0\leq\chi(x)\leq 1$ for $1<\|x\|<2$,
and $\chi(x)=0$ for $\|x\|\geq 2$.
Then for each $n\in\bN$, $(t,x)\in[0,T]\times\bR^d$, and $z\in\bR^d$ 
define $\chi_n(x):=\chi(x/n)$ and
\begin{align*}
&
\mu^{(n)}(t,x):=\mu(t,x)\chi_n(x),\quad \sigma^{(n)}(t,x):=\sigma(t,x)\chi_n(x),
\quad \eta^{(n)}_t(x,z):=\eta_t(x,z)\chi_n(x),\\
&
g^{(n)}(x):=g(x)\chi_n(x), \quad h^{(n)}(t,x):=h(t,x)\chi_n(x).
\end{align*}
For every $n\in\bN$, $(t,x)\in[0,T]\times\bR^d$ 
let $\big(X^{t,x,(n)}_s\big)_{s\in[t,T]}:[t,T]\times\Omega\to\bR^d$ be an 
$\bF$-adapted c\`adl\`ag process satisfying that $X^{t,x}_t=x$,
and almost surely for all $s\in[t,T]$ 
$$                                             
dX^{t,x,(n)}_{s}
=\mu^{(n)}\big(s,X^{t,x,(n)}_{s-}\big)\,ds
+\sigma^{d}\big(s,X^{t,x,(n)}_{s-}\big)\,dW_s
+\int_{\bR^{d}}\eta^{(n)}_s\big(X^{t,x,(n)}_{s-},z\big)
\,\tilde{\pi}(dz,ds).
$$
For each $n\in\bN$, $k\in\bN$, and $(t,x)\in[0,T]\times\bR^d$, we
define 
\begin{equation}                           \label{def u n k}
u^{n,k}(t,x):=\bE\left[
g^{(k)}\big(X^{t,x,(n)}_T\big)+\int_t^Th^{(k)}\big(s,X^{t,x,(n)}_s\big)\,ds
\right],
\end{equation}
and
\begin{equation}                            \label{def u 0 k}
u^{0,k}(t,x):=\bE\left[
g^{(k)}\big(X^{t,x}_T\big)+\int_t^Th^{(k)}\big(s,X^{t,x}_s\big)\,ds
\right].
\end{equation}
By construction, 
for every $n\in\bN$ we have that 
$\mu^{(n)}$, $\sigma^{(n)}$, $\eta^{(n)}$, $g^{(n)}$,
and $h^{(n)}$ satisfy Assumption \ref{assumption conv compact}.
In addition, for all $n\in\bN$, $t\in[0,T]$, and $x,x'\in\bR^d$ with 
$\|x\|\wedge\|x'\|\geq 2n$ we have
$$
\big\|\mu^{(n)}(t,x)-\mu^{(n)}(t,x')\big\|=0.
$$
Moreover, by \eqref{assumption Lip mu sigma eta}, \eqref{assumption growth},
and the mean value theorem
we have for all $n\in\bN$, $t\in[0,T]$, and $x,x'\in\bR^d$ with 
$\|x'\|\leq 2n$
that  
\begin{align*}
&
\big\|\mu^{(n)}(t,x)-\mu^{(n)}(t,x')\big\|^2
=\big\|\mu(t,x)\chi_n(x)-\mu(t,x')\chi_n(x')\big\|^2\\
&
\leq 2\big\|\mu(t,x)\chi_n(x)-\mu(t,x')\chi_n(x)\big\|^2
+2\big\|
\mu(t,x')(\chi_n(x)-\chi_n(x'))
\big\|^2\\
&
\leq
2L\|x-x'\|^2\cdot \sup_{y\in\bR^d}\big(|\chi(y)|^2\big)
+2\cdot\sup_{\|y\|\leq 2n}\big(\|\mu(t,y)\|^2\big)\cdot
\sup_{y\in\bR^d}\big(\big\|\nabla\chi_n(y)\big\|^2\big)\cdot\|x-x'\|^2\\
&
\leq
2L\|x-x'\|^2
+\frac{2Ld^p(1+4n^2)}{n^2}\cdot
\sup_{y\in\bR^d}\big(\big\|\nabla\chi(y)\big\|^2\big)\cdot\|x-x'\|^2\\
&
\leq 2L\|x-x'\|^2\cdot\left[
1+5d^p\cdot
\sup_{y\in\bR^d}\big(\big\|\nabla\chi(y)\big\|^2\big)
\right].
\end{align*}
Furthermore, by \eqref{assumption growth} we have for all $n\in\bN$ and
$(t,x)\in[0,T]\times\bR^d$ that
$$
\big\|\mu^{(n)}(t,x)\big\|^2\leq \big\|\mu(t,x)\big\|^2\leq Ld^p(1+\|x\|^2).
$$
Hence, by analogous arguments we obtain that $\mu^{(n)}$, $\sigma^{(n)}$, $\eta^{(n)}$, $g^{(n)}$,
and $h^{(n)}$ satisfy Assumptions \ref{assumption Lip and growth n}
and \ref{assumption eta n}.
Furthermore, by the product rule it holds for all $n\in\bN$, $\gamma\in[0,1]$,
$(t,x)\in[0,T]\times\bR^d$, and $z\in\bR^d$ that
\begin{equation}                                       \label{jacobian eta n 0}
\gamma D_x\eta^{(n)}_t(x,z)
=\gamma D_x\big(\eta_t(x,z)\chi_n(x)\big)
=\gamma\chi_n(x)D_x\eta_t(x,z)+\gamma\eta_t(x,z)\left[\nabla\chi_n(x)\right]^T,
\end{equation}
where $\eta_t(x,z)\left[\nabla\chi_n(x)\right]^T$
is understood as the matrix product of the column vector
$\eta_t(x,z)$ and the row vector $\left[\nabla\chi_n(x)\right]^T$.
By the definition of $\chi_n(\cdot)$, $n\in\bN$, and Assumption
\ref{assumption pointwise} we also notice that it holds for
all $n\in\bN$, $(t,x)\in[0,T]\times\bR^d$, and $z\in\bR^d$ that
$
0\leq \chi_n(x)\leq 1
$ 
and 
\begin{equation}                                        \label{jacobian eta n 1}
\big\|\eta_t(x,z)\left[\nabla\chi_n(x)\right]^T\big\|_F
\leq \big\|\eta_t(x,z)\big\|\cdot\big\|\nabla \chi_n(x)\big\|
\leq \frac{C_d^{1/2}}{n}\cdot\sup_{y\in\bR^d}\left(\|\nabla \chi(y)\|\right),
\end{equation}
where $C_d$ is the positive constant introduced in \eqref{assumption eta}.
Moreover, by Assumption \ref{assumption pointwise} we have for all 
$(t,x)\in[0,T]\times\bR^d$, $z\in\bR^d$, and $\gamma\in[0,1]$ that
\begin{equation}                                      \label{jacobian eta n 2}
\big\|\gamma D_x\eta_t(x,z)\big\|_F\leq d^{1/2}C_d^{1/2}.
\end{equation}
We also notice that the mapping
$$
\left\{B\in\bR^{d\times d}:\|B\|_F\leq C_d^{1/2}
(d^{1/2}+\sup_{y\in\bR^d}\left(\|\nabla \chi(y)\|\right))\right\}
\ni A\mapsto \big|\det\big(\mathbf{I}_d+A\big)\big|\in\bR
$$
is uniformly continuous.
Hence, \eqref{jacobian cond} 
and \eqref{jacobian eta n 0}-\eqref{jacobian eta n 2} ensure that
there exists an integer $N>0$ such that for all $n\geq N$, 
$(t,x)\in[0,T]\times\bR^d$, $\gamma\in[0,1]$, and $z\in\bR^d$
\begin{equation*}                                    
\frac{\lambda}{2}\leq \big|\det(\mathbf{I}_d+\gamma D_x \eta^{(n)}_t(x,z))\big|.
\end{equation*}
This implies that $\eta^{(n)}$ satisfies Assumption \ref{assumption jacobian n}
with $\delta=\lambda/2$ for sufficiently large $n\in\bN$.
By the above discussion, 
applying Proposition \ref{proposition smooth PIDE}
(with $\mu\cal\mu^{(n)}$, $\sigma\cal\sigma^{(n)}$, $\eta\cal\eta^{(n)}$,
$g\cal g^{(k)}$, $h\cal h^{(k)}$ for $n,k\in\bN$ in the notation of Proposition
\ref{proposition smooth PIDE}) yields that for all $k\in\bN$ and
sufficiently large $n\in\bN$  
the function
$u^{n,k}\in C_1([0,T]\times\bR^d)$ is a viscosity solution of
\begin{align}
&
\frac{\partial}{\partial t}u^{n,k}(t,x)
+\langle\nabla_xu^{n,k}(t,x),\mu^{(n)}(t,x)\rangle
+\frac{1}{2}\operatorname{Trace}
\left(\sigma^{(n)}(t,x)[\sigma^{(n)}(t,x)]^T
\operatorname{Hess}_xu^{n,k}(t,x)\right)\nonumber\\
&
+h^{(k)}(t,x)+\int_{\bR^d}\left(
u^{n,k}(t,x+\eta^{(n)}_t(x,z))-u^{n,k}(t,x)-\langle\nabla_xu^{n,k}(t,x),
\eta^{(n)}_t(x,z)\rangle\right)\nu(dz)=0               \label{PIDE n k}
\end{align}
on $(0,T)\times \bR^d$ with $u^{n,k}(T,x)=g^{(k)}(x)$ on $\bR^d$.

Next, for each $n\in\bN$ and $(t,x)\in[0,T]\times\bR^d$, 
let $\tau^{t,x}_n:\Omega\to[0,T]$ be the stopping
time defined by
\begin{equation}                                     \label{def tau X t x n}
\tau^{t,x}_n:=\inf
\big\{
s\geq t: \max\big(\|X^{t,x,(n)}_s\|,\|X^{t,x}_s\|\big)> n
\big\}
\wedge T.
\end{equation}
Then \eqref{def tau X t x n} and Lemma \ref{lemma two SDEs} ensures that
for all $(t,x)\in[0,T]\times\bR^d$ and $n\in\bN$
\begin{equation}                                       \label{prob equality}
\bP\left(
\mathbf{1}_{\{s\leq \tau^{t,x}_n\}}\big\|X^{t,x,(n)}_s-X^{t,x}_s\big\|=0
\text{ for all $s\in[t,T]$}
\right)=1.
\end{equation}
Hence, considering that for every $n\in\bN$ and $(t,x)\in [0,T]\times\bR^d$
the processes $\big(X^{t,x,(n)}_s\big)_{s\in[t,T]}$ and 
$\big(X^{t,x}_s\big)_{s\in[t,T]}$ have no fixed time of discontinuity,
we have for all $n,k\in\bN$, and $(t,x)\in[0,T]\times\bR^d$ that
\begin{align}
\bE\left[
\big|
g^{(k)}\big(X^{t,x,(n)}_T\big)-g^{(k)}\big(X^{t,x}_T\big)
\big|
\right]
&
=\bE\left[
\mathbf{1}_{\{\tau^{t,x}_n<T\}}\big|
g^{(k)}\big(X^{t,x,(n)}_T\big)-g^{(k)}\big(X^{t,x}_T\big)
\big|
\right]\nonumber\\
&
\leq 2\left[\sup_{y\in\bR^d}\big|g^{(k)}(y)\big|\right]
\bP\big(\tau^{t,x}_n<T\big),
\label{whisky 1}
\end{align}
and
\begin{align}
&
\int_t^T\bE\left[\big|
h^{(k)}(s,X^{t,x,(n)}_s)-h^{(k)}(s,X^{t,x}_s)
\big|\right]ds\nonumber\\
&
=\int_t^T\bE\left[\mathbf{1}_{\{\tau^{t,x}_n<T\}}\big|
h^{(k)}(s,X^{t,x,(n)}_s)-h^{(k)}(s,X^{t,x}_s)
\big|\right]ds
\nonumber\\
&
\leq 2T\left[\sup_{(s,y)\in[0,T]\times\bR^d}\big|h^{(k)}(s,y)\big|\right]
\bP\big(\tau^{t,x}_n<T\big).
\label{whisky 2}
\end{align}
Furthermore, by \eqref{SDE moment est}, the construction of 
$\mu^{(n)}$, $\sigma^{(n)}$, and $\eta^{(n)}$, and Assumptions
\ref{assumption Lip and growth n} and \ref{assumption eta n},
it holds for all $n\in\bN$ and $(t,x)\in[0,T]$ that
\begin{equation}                                             \label{est sup X n}
\bE\Big[\sup_{s\in[t,T]}\big\|X^{t,x,(n)}_s\big\|^2\Big]
\leq C_3(1+\|x\|^2),
\end{equation}
where $C_3$ is a positive constant only depending on $L$, $d$, $p$ and $T$.
Thus, by \eqref{def tau X t x n}, Chebyshev's inequality, and
\eqref{SDE moment est} it holds for all $(t,x)\in[0,T]\times\bR^d$ and $n\in\bN$
that
\begin{align}                                       
\bP\big(\tau^{t,x}_n<T\big)
&
\leq \bP\left(\big\|X^{t,x}_{\tau^{t,x}_n}\big\|
\geq n\right)
+\bP\left(\big\|X^{t,x,(n)}_{\tau^{t,x}_n}\big\|
\geq n\right)
\nonumber\\
&
\leq \frac{\bE\left[\big\|X^{t,x}_{\tau^{t,x}_n}\big\|^2\right]
+\bE\left[\big\|X^{t,x,(n)}_{\tau^{t,x}_n}\big\|^2\right]}{n^2}
\nonumber\\
&
\leq \frac{(C+C_3)(1+\|x\|^2)}{n^2}.                          \label{whisky 3}
\end{align}
Combining \eqref{whisky 1}, \eqref{whisky 2}, and \eqref{whisky 3},
we have for all $(t,x)\in[0,T]\times\bR^d$ and $n,k\in\bN$ that
$$
\big|u^{n,k}(t,x)-u^{0,k}(t,x)\big|\leq\frac{2(C+C_3)(1+\|x\|^2)}{n^2}
\left[
\sup_{y\in\bR^d}\left(\big|g^{(k)}(y)\big|\right)
+T\cdot\sup_{(s,y)\in[0,T]\times\bR^d}\left(\big|h^{(k)}(s,y)\big|\right)
\right].
$$
This implies that for all $k\in\bN$ and all compact sets 
$\mathcal{K}\subseteq [0,T]\times\bR^d$ it holds that
\begin{equation}                                         \label{whisky 4}
\lim_{n\to\infty}\left[
\sup_{(t,x)\in\mathcal{K}}\big|u^{n,k}(t,x)-u^{0,k}(t,x)\big|
\right]=0.
\end{equation} 
Moreover, by \eqref{assumption growth f g}, \eqref{SDE moment est},
\eqref{cond h}, \eqref{assumption growth n}, 
\eqref{def u n k}, \eqref{def u 0 k}, \eqref{est sup X n}, and the fact that
$(a+b)^{1/2}\leq a^{1/2}+b^{1/2}$ for $a,b\geq 0$,
we notice for all $n,k\in\bN$ and $(t,x)\in[0,T]\times\bR^d$ that 
\begin{align}
|u^{n,k}(t,x)|
&
\leq \bE\left[\big|g^{(k)}(X^{t,x,(n)}_T)\big|\right] 
+\int_t^T\bE\left[\big|h^{(k)}(s,X^{t,x,(n)}_s)\big|\right]ds
\nonumber\\
&
\leq \left(\bE\left[L(d^p+\|X^{t,x,(n)}_T\|^2)\right]\right)^{1/2}
+\int_t^T\left(\bE\left[C_0\big(1+\|X^{t,x,(n)}_s\|^2\big)\right]\right)^{1/2}
\,ds
\nonumber\\
&
\leq L^{1/2}d^{p/2}+L^{1/2}C^{1/2}_3(1+\|x\|^2)^{1/2}
+C_0^{1/2}T[1+C_3^{1/2}(1+\|x\|^2)^{1/2}]
\nonumber\\
&
\leq L^{1/2}d^{p/2}+C_0^{1/2}T
+C^{1/2}_3(L^{1/2}+C_0^{1/2}T)(1+\|x\|^2)^{1/2},
\label{diva 1}
\end{align}
and
\begin{equation}                                        \label{diva 2}
|u^{0,k}(t,x)|
\leq L^{1/2}d^{p/2}+C_0^{1/2}T+C^{1/2}(L^{1/2}+C_0^{1/2}T)(1+\|x\|^2)^{1/2}.
\end{equation}
This implies that for all $k,n\in\bN$ 
$$
\sup_{(t,x)\in[0,T]\times\bR^d}
\frac{\big|u^{n,k}(t,x)\big|+\big|u^{0,k}(t,x)\big|}{1+\|x\|}<\infty.
$$
Hence, \eqref{whisky 4} and 
the fact that $u^{n,k}\in C_1([0,T]\times\bR^d)$ for
all $k\in\bN$ and sufficiently large $n\in\bN$ 
ensure that $u^{0,k}\in C_1([0,T]\times\bR^d,\bR)$ for all $k\in\bN$.
Therefore, by \eqref{PIDE n k}, \eqref{whisky 4}, 
Assumptions \ref{assumption Lip and growth}, \ref{assumption pointwise},
and Assumptions \ref{assumption h}-\ref{assumption eta n},
applying Lemma
\ref{lemma subsolution} yields that for all $k\in\bN$ 
the function 
$u^{0,k}$ is a viscosity solution of
\begin{align}
&
\frac{\partial}{\partial t}u^{0,k}(t,x)
+\langle\nabla_xu^{0,k}(t,x),\mu(t,x)\rangle
+\frac{1}{2}\operatorname{Trace}
\left(\sigma(t,x)[\sigma(t,x)]^T
\operatorname{Hess}_xu^{0,k}(t,x)\right)\nonumber\\
&
+h^{(k)}(t,x)+\int_{\bR^d}\left(
u^{0,k}(t,x+\eta_t(x,z))-u^{0,k}(t,x)-\langle\nabla_xu^{0,k}(t,x),
\eta_t(x,z)\rangle\right)\nu(dz)=0               \label{PIDE 0 k}
\end{align}
on $(0,T)\times \bR^d$ with $u^{0,k}(T,x)=g^{(k)}(x)$ on $\bR^d$.
Furthermore, Lemma \ref{lemma est moment} implies for all 
$k\in\bN$ and $(t,x)\in[0,T]\times\bR^d$ that
\begin{align}
&
\frac{\bE\left[\int_t^T\big|h^{(k)}(s,X^{t,x}_s)-h(s,X^{t,x}_s)\big|
\,ds\right]}{d^p+\|x\|^2}\nonumber\\
&
=\int_t^T\bE\left[
\frac{\big|h^{(k)}(s,X^{t,x}_s)-h(s,X^{t,x}_s)\big|}{d^p+\big\|X^{t,x}_s\big\|^2}
\cdot\frac{d^p+\big\|X^{t,x}_s\big\|^2}{d^p+\|x\|^2}
\right]ds\nonumber\\
&
\leq \int_t^T\left[
\sup_{(r,y)\in[0,T]\times\bR^d}\left(\frac{\big|h^{(k)}(r,y)-h(r,y)\big|}
{d^p+\|y\|^2}\right)
\right]
\cdot
\frac{\bE\left[d^p+\big\|X^{t,x}_s\big\|^2\right]}{d^p+\|x\|^2}\,ds
\nonumber\\
&
\leq c_1T\left[
\sup_{(s,y)\in[0,T]\times\bR^d}\left(\frac{\big|h^{(k)}(s,y)-h(s,y)\big|}
{d^p+\|y\|^2}\right)
\right],
\label{rum 1}
\end{align}
where $c_1:=(6LT+1)e^{(1+6L)T}$.
Moreover, by Assumption \ref{assumption h} it holds for all $k\in\bN$ and
$(s,y)\in[0,T]\times\bR^d$ that
\begin{align*}
\frac{\big|h^{(k)}(s,y)-h(s,y)\big|}{d^p+\|y\|^2}
&
=\frac{\mathbf{1}_{\{\|y\|\geq k\}}|h(s,y)(\chi_k(y)-1)|}{d^p+\|y\|^2}\\
&
\leq \frac{2C_0^{1/2}(d^p+\|y\|^2)^{1/2}}{(d^p+k^2)^{1/2}(d^p+\|y\|^2)^{1/2}}
=2C_0^{1/2}(d^p+k^2)^{-1/2}.
\end{align*}
Combining this with \eqref{rum 1} implies that
\begin{equation}                                               \label{rum 2}
\lim_{k\to\infty}\left[
\sup_{(t,x)\in[0,T]\times\bR^d}\left(
\frac{\bE\left[\int_t^T\big|h^{(k)}(s,X^{t,x}_s)-h(s,X^{t,x}_s)\big|
\,ds\right]}{d^p+\|x\|^2}
\right)
\right]=0.
\end{equation}
Analogously, by Lemma \ref{lemma est moment} 
and \eqref{assumption growth f g}
we have
\begin{equation}                                               \label{rum 3}
\lim_{k\to\infty}\left[
\sup_{(t,x)\in[0,T]\times\bR^d}\left(
\frac{\bE\left[\big|g^{(k)}(X^{t,x}_s)-g(X^{t,x}_s)\big|
\right]}{d^p+\|x\|^2}
\right)
\right]=0.
\end{equation}
Then \eqref{def u n k}, \eqref{def u 0 k} together with 
\eqref{rum 2} and \eqref{rum 3} imply that
$$                                              
\lim_{k\to\infty}\left[
\sup_{(t,x)\in[0,T]\times\bR^d}\left(
\frac{\big|u^{0,k}(t,x)-u(t,x)\big|}{d^p+\|x\|^2}
\right)
\right]=0.
$$
Hence, for all non-empty compact set $\mathcal{K}\subseteq [0,T]\times\bR^d$
it holds that
\begin{align}   
&                                   
\lim_{k\to\infty}\left[
\sup_{(t,x)\in\mathcal{K}}\big|u^{0,k}(t,x)-u(t,x)\big|
\right]\nonumber\\
&
\leq 
\lim_{k\to\infty}\left[
\sup_{(t,x)\in\mathcal{K}}\left(
\frac{\big|u^{0,k}(t,x)-u(t,x)\big|}{d^p+\|x\|^2}
\right)
\right]
\cdot
\left[
\sup_{(t,x)\in\mathcal{K}}(d^p+\|x\|^2)
\right]=0.
\label{rum 4}
\end{align}
By \eqref{assumption growth f g}, \eqref{SDE moment est}, \eqref{FK linear 2},
and the analogous 
argument to obtain \eqref{diva 1} and \eqref{diva 2} we have that
$u$ has linear growth. 
Then by \eqref{rum 4} and the fact that 
$u^{0,k}\in C_1([0,T]\times\bR^d)$ for every $k\in\bN$,
we have that $u\in C_1([0,T]\times\bR^d)$. 
Therefore, by \eqref{PIDE 0 k}, \eqref{rum 4}, and Assumptions 
\ref{assumption Lip and growth}, \ref{assumption pointwise},
and \ref{assumption h}-\ref{assumption jacobian n}, we apply
Lemma \ref{lemma subsolution} to get that
$u$ is a viscosity solution of the linear PIDE
\eqref{linear PIDE} with $u(T,x)=g(x)$ for all $x\in\bR^d$. 
The proof of this proposition is thus completed.
\end{proof}

\begin{proposition}                            \label{proposition PIDE existence}
Let Assumptions \ref{assumption Lip and growth}, 
\ref{assumption pointwise}, and \ref{assumption jacobian} hold.
Then the following holds:
\begin{enumerate}[(i)]
\item{
There exists a unique Borel function 
$u:[0,T]\times\bR^d\to\bR$ satisfying
\begin{equation}                                           \label{FK nonlinear}
u(t,x):=\bE\Bigg[g(X^{t,x}_T)+\int_t^Tf(s,X^{t,x}_s,u(s,X^{t,x}_s))\,ds\Bigg],
\quad (t,x)\in[0,T]\times\bR^d.
\end{equation}
}
\item{
Let $u:[0,T]\times\bR^d\to\bR$ satisfy \eqref{FK nonlinear}. Then
$u\in C_1([0,T]\times\bR^d)$.
}
\item{
Let $u:[0,T]\times\bR^d\to\bR$ satisfy \eqref{FK nonlinear}. Then
$u$ is a viscosity solution of
PIDE \eqref{APIDE} with $u(T,x)=g(x)$ for all $x\in\bR^d$.
}
\end{enumerate}
\end{proposition}

\begin{proof}
First recall $\Lambda:=\{(t,s)\in[0,T]^2: t\leq s\}$.
Then we notice that
by Corollary \ref{corollary prob 1}, 
the mapping
$\Lambda\times\bR^d\ni(t,s,x)
\mapsto\big(s,X^{t,x}_{s}\big)\in\cL_0(\Omega,[0,T]\times\bR^d)$
is continuous, where $\cL_0(\Omega,[0,T]\times\bR^d)$ denotes the space
of all measurable functions from $\Omega$ to $[0,T]\times\bR^d$ equipped with
the metric deduced by convergence in probability.
Moreover, notice that for every nonnegative Borel function
$\varphi:[0,T]\times\bR^d\to[0,\infty)$, the mapping
$
\cL_0(\Omega,[0,T]\times\bR^d)\ni Z \mapsto \bE\big[\varphi(Z)\big]\in [0,\infty]
$
is measurable.
Hence for all $\theta\in\Theta$,
$d\in\bN$, and all nonnegative Borel functions
$\varphi:[0,T]\times\bR^d\to[0,\infty)$ it holds that the mapping
\begin{equation}                                     \label{measurable 1}                              
\Lambda\times\bR^d\ni(t,s,x)\mapsto
\bE\Big[\varphi\big(s,X^{t,x}_{s}\big)\Big]
\in[0,\infty]
\end{equation}
is measurable. 
Furthermore, by Lemma \ref{lemma est moment},
we observe that for all 
$(t,x)\in[0,T]\times\bR^d$ and $s\in[t,T]$ 
\begin{equation}                                    \label{estimate 1}                                         
\bE\Big[d^p+\big\|X^{t,x}_{s}\big\|^2\Big]\leq 
c_1(d^p+\|x\|^2)e^{\rho_1(s-t)},
\end{equation}
where $c_1:=1+6LT$ and $\rho_1:=1+6L$.
Then, combining \eqref{assumption Lip f g}, \eqref{assumption growth f g}, 
\eqref{measurable 1}, \eqref{estimate 1}, 
and applying Proposition \ref{Prop FP} (with
$Z^{t,x}_s\cal X^{t,x}_{s}$, $a\cal \max\{L^{1/2},L^{1/2}T^{-1/2}\}$,
$b \cal c_1^{1/2}e^{\rho_1T/2}$, $c \cal L^{1/2}$,
$f \cal f$, $g \cal g$, $u \cal u$ in the notation of Proposition
\ref{Prop FP}),
we obtain that there is a unique Borel function 
$u:[0,T]\times\bR^d\to\bR$ satisfying \eqref{FK nonlinear} as well as for
all $t\in[0,T]$ that
\begin{equation}                                               \label{u est}
\sup_{x\in\bR^d}\frac{|u(t,x)|}{\sqrt{d^p+\|x\|^2}}\leq 2L^{1/2}
\exp\big\{L^{1/2}(T-t)\big\}.
\end{equation}
Hence, we finish the proof of (i).
Furthermore, for every $r\in(0,\infty)$ let $B_r\subseteq\bR^d$ 
be the set defined by
$
B_r:=\{x\in\bR^d:\|x\|\leq r\}
$.
Then by \eqref{assumption growth f g} and \eqref{u est} we have for all 
$(s,y)\in[0,T]\times\bR^d$ that
\begin{align*}
\frac{|f(s,y,0)|+|u(s,y)|+g(y)}{d^p+\|y\|^2}
&
\leq \frac{T^{-1}L^{1/2}(d^p+\|y\|^2)^{1/2}
+2L^{1/2}\exp\{L^{1/2}T\}(d^p+\|y\|^2)^{1/2}}
{d^p+\|y\|^2}\\
& \quad
+\frac{L^{1/2}(d^p+\|y\|^2)^{1/2}}{d^p+\|y\|^2}\\
&
=\frac{L^{1/2}(T^{-1}+2\exp\{L^{1/2T}\}+1)}{(d^p+\|y\|^2)^{1/2}}.
\end{align*}
This implies that
\begin{align}
&
\inf_{r\in(0,\infty)}
\left[
\sup_{s\in[0,T]}\sup_{y\in\bR^d/B_r}
\frac{|f(s,y,0)|+|u(s,y)|+|g(y)|}{d^p+\|y\|^2}
\right]\nonumber\\
&
\leq
\inf_{r\in(0,\infty)}
\left[
\sup_{s\in[0,T]}\sup_{y\in\bR^d/B_r}
\frac{L^{1/2}(T^{-1}+2\exp\{L^{1/2T}\}+1)}{(d^p+\|y\|^2)^{1/2}}
\right]\nonumber\\
&
=\inf_{r\in(0,\infty)}\left[
\frac{L^{1/2}(T^{-1}+2\exp\{L^{1/2T}\}+1)}{(d^p+\|r\|)^{1/2}}
\right]
=0.
\label{inf 0}
\end{align}
Then, by \eqref{assumption Lip f g}, \eqref{assumption growth f g},
\eqref{estimate 1}, \eqref{inf 0}, and Corollary \ref{corollary prob 1}, 
we can apply Corollary 2.7 in \cite{BGHJ2019}
(with $X^{t,x}\cal X^{t,x}$, $f\cal f$, $g\cal g$, and
$V(t,x)\cal (d^p+\|x\|^2)$ in the notation of Corollary 2.7 in \cite{BGHJ2019}) 
to get that $u:[0,T]\times\bR^d\to\bR$ is continuous
(Noting that the condition $\bE[V(s,X^{t,x}_s)]\leq V(t,x)$ in 
Corollary 2.7 in \cite{BGHJ2019} can be relaxed to
$\bE[V(s,X^{t,x}_s)]\leq cV(t,x)$ for any $c\geq 1$).
Combining this with \eqref{u est}, we obtain that $u\in C_1([0,T]\times\bR^d)$.
Thus, we obtain (ii).
Next, we define $\mathbf{h}:[0,T]\times\bR^d\to\bR$ by
$$
\mathbf{h}(t,x):=f(t,x,u(t,x)), \quad (t,x)\in[0,T]\times\bR^d.
$$
By \eqref{assumption Lip f g}, \eqref{assumption growth f g},
and \eqref{u est}
it holds for all $(t,x)\in[0,T]\times\bR^d$ that
\begin{align}
|\mathbf{h}(t,x)|
&
=|f(t,x,u(t,x))|\leq |f(t,x,u(t,x))-f(t,x,0)|+|f(t,x,0)|\nonumber\\
&
\leq L^{1/2}|u(t,x)|+L^{1/2}T^{-1/2}(d^p+\|x\|^2)^{1/2}\nonumber\\
&
\leq L^{1/2}d^{p/2}(2L^{1/2}\exp\{L^{1/2}T\}+T^{-1/2})(1+\|x\|^2)^{1/2}.
\label{LG bf h}
\end{align}
This implies that $\mathbf{h}$ satisfies Assumption \ref{assumption h}.
Thus, Assumptions \ref{assumption Lip and growth}-\ref{assumption jacobian} 
and the fact that $u\in C_1([0,T]\times\bR^d)$ satisfies \eqref{FK nonlinear}
allow us to
apply Proposition \ref{proposition linear PIDE} to obtain that
$u$ is a viscosity solution of PIDE \eqref{APIDE} 
with $u(T,x)=g(x)$ for all $x\in\bR^d$, which proves (iii).
The proof of this proposition is therefore completed.
\end{proof}

\section{\textbf{Multilevel Picard (MLP) approximations 
not involving Euler schemes}}
\label{section no Euler}
\subsection{Setting}                                 \label{MLP setting}
We assume the settings in Section \ref{section setting}.
Moreover, for each $d\in\bN$ and $(t,x)\in[0,T]\times\bR^d$
let $(Z^{d,\theta,t,x}_{s})_{s\in[t,T]}$:
$[t,T]\times \Omega\to \bR^d$, $\theta\in\Theta$, be 
$\cB([0,T])\otimes\cF/\cB(\bR^d)$-measurable functions. 
Assume for all $(t,x)\in[0,T]\times\bR^d$ and $s\in[t,T]$ that 
$Z^{d,\theta,t,x}_{s}$,
$\theta\in\Theta$, are independent and identically distributed, and assume 
that $(Z^{d,\theta,t,x}_{s})_{(\theta,s)\in\Theta\times[t,T]}$
and $(\xi^\theta)_{\theta\in\Theta}$ are independent.
For each $d\in \bN$, $(t,x)\in[0,T]\times\bR^d$, $s\in[t,T]$, 
and $\theta\in\Theta$, let $(Z^{(d,\theta,t,x,l,i)}_{s})_{(l,i)\in\bN\times\bZ}$
be independent copies of $Z^{d,\theta,t,x}_{s}$.
For each $d\in\bN$, $n\in\bN_0$, $M\in\bN$, and $\theta\in\Theta$, 
let $\mathcal{U}^{d,\theta}_{n,M}:[0,T]\times \bR^d \times \Omega \to \bR$ 
be a measurable function
satisfying for all $(t,x)\in[0,T]\times\bR^d$ that
$$
\mathcal{U}^{d,\theta}_{n,M}(t,x)
=\frac{\mathbf{1}_{\bN}(n)}{M^n}\sum_{i=1}^{M^n}
g^d\Big(Z^{(d,\theta,t,x,0,-i)}_{T}\Big)
$$
$$
+\sum_{l=0}^{n-1}\frac{(T-t)}{M^{n-l}}
\left[\sum^{M^{n-l}}_{i=1}\left(F^d\Big(\mathcal{U}^{(d,\theta,l,i)}_{l,M}\Big)
-\mathbf{1}_{\bN}(l)F^d\Big(\mathcal{U}^{(d,\theta,-l,i)}_{l-1,M}\Big)\right)
\Big(\cR^{(\theta,l,i)}_t,Z^{(d,\theta,t,x,l,i)}_{\cR^{(\theta,l,i)}_t}\Big)
\right],
$$
where $\Big(\mathcal{U}^{(d,\theta,l,i)}_{n,M}\Big)_{(l,i)
\in (\bZ/\{0\}\times\bN_0)}$
are independent copies of $\mathcal{U}^{d,\theta}_{n,M}$,
and where $\big(\cR^{(\theta,l,i)}_t\big)_{(l,i)\in\bN\times\bN_0}$ are independent
copies of $\cR^{\theta}_t$ for each $t\in[0,T]$.
Furthermore, let $b,c,\rho\in(0,\infty)$, and let
$u^d:[0,T]\times\bR^d\to\bR$, $d\in\bN$, be measurable functions
such that for all $d\in\bN$, $(t,x)\in[0,T]\times\bR^d$, and $s\in[t,T]$  
$$
\bE\Big[\big|g^d(Z^{d,0,t,x}_{T})\big|\Big]+\int_t^T
\bE\Big[\big|(F^d(u^d))(r,Z^{d,0,t,x}_{r})\big|\Big]\,dr<\infty,
$$
$$
\bE\Big[d^p+\big\|Z^{d,0,t,x}_{s}\big\|^2\Big]
\leq be^{\rho(s-t)}(d^p+\|x\|^2),
$$
$$
u^d(t,x)=\bE\left[g(Z^{d,0,t,x}_{T})
+\int_t^T(F^d(u^d))(r,Z^{d,0,t,x}_{r})\,dr\right],
$$
and
\begin{equation}                                    \label{LG F g u}
|(F^d(0))(t,x)|^2+|g^d(x)|^2+|u^d(t,x)|^2\leq c(d^p+\|x\|^2).
\end{equation}

\subsection{Error bounds for MLP approximations}
\label{section Error Bounds 1}
The following lemma is taken from \cite{HJKN2020}.
\begin{lemma}                                          \label{Lemma MLP}
Assume Setting \ref{MLP setting}. 
Then for all $d,n,M\in\bN$
and $t\in[0,T]$ it holds
that 
\begin{align}
&
\sup_{x\in\bR^d}\left( 
\frac{\bE\left[\big|\mathcal{U}^{d,0}_{n,M}(t,x)-u^d(t,x)\big|^2\right]}
{d^p+\|x\|^2} \right)^{1/2}
\leq 2e^{M/2}M^{-n/2}(1+2TL^{1/2})^{n-1}b^{1/2}e^{\rho(T-t)/2}\nonumber\\
&                                         
\cdot \left(\sup_{t\in[0,T]}\sup_{x\in \bR^d}
\left(
\frac{\max\left\{|T\cdot (F^d(0))(t,x)|,|g^d(x)|\right\}}
{\sqrt{d^p+\|x\|^2}}\right)
+TL^{1/2}\sup_{t\in[0,T]}\sup_{x\in\bR^d}
\left(\frac{|u^d(t,x)|}{\sqrt{d^p+\|x\|^2}}\right)                \label{est MLP}
\right)
\end{align}
\end{lemma}
\begin{proof}
For $b=1$, the above statement is proved in Corollary 3.12 in \cite{HJKN2020} 
(with $c\cal c$, $\rho\cal \rho$,
$L\cal L^{1/2}$,
$\tau\cal t$, $U^\theta_{n,M}\cal \mathcal{U}^{d,\theta}_{n,M}$,
$Y^\theta_{t,s}(x)\cal Z^{d,\theta,t,x}_s$,
$\varphi\cal (\bR^d\ni x\mapsto(d^p+\|x\|^2)\in\bR)$,
$F\cal F^d$, $g\cal g^d$, $u\cal u^d$
in the notation of Corollary 3.12 in \cite{HJKN2020}).
For the general case $b\in(0,\infty)$, one can precisely follow the proof
of Corollary 3.12 in \cite{HJKN2020} to obtain \eqref{est MLP}.
\end{proof}

\section{\textbf{Proof of the main results}}  
\label{section MLP}

\begin{proof}[Proof of Theorem \ref{MLP conv}]
(i) is well-known
(see, e.g., Theorem 3.1 in \cite{Kunita} and Theorem 117 in \cite{Situ}).
To prove (ii) we notice that
by Corollary \ref{corollary prob 1}, for all $\theta\in\Theta$ and
$d\in \bN$ the mapping
$\Lambda\times\bR^d\ni(t,s,x)
\mapsto\big(s,X^{d,\theta,t,x}_{s}\big)\in\cL_0(\Omega,[0,T]\times\bR^d)$
is continuous, where $\cL_0(\Omega,[0,T]\times\bR^d)$ denotes the metric space
of all measurable functions from $\Omega$ to $[0,T]\times\bR^d$ equipped with
the metric deduced by convergence in probability.
Moreover, notice that for all $d\in\bN$ and nonnegative Borel functions 
$\varphi:[0,T]\times\bR^d\to[0,\infty)$, the mapping
$
\cL_0(\Omega,[0,T]\times\bR^d)\ni Z \mapsto \bE\big[\varphi(Z)\big]\in [0,\infty]
$
is measurable.
Hence for all $\theta\in\Theta$,
$d\in\bN$, and all nonnegative Borel functions
$\varphi:[0,T]\times\bR^d\to[0,\infty)$ it holds that the mapping
\begin{equation}                                  \label{measurability 1}
\Lambda\times\bR^d\ni(t,s,x)\mapsto
\bE\Big[\varphi\big(s,X^{d,\theta,t,x}_{s}\big)\Big]
\in[0,\infty]
\end{equation}
is measurable. Analogously, considering Corollary \ref{lemma prob 2}, we obtain 
for all $\theta\in\Theta$,
$d\in\bN$, $N\in\bN$ and all nonnegative Borel functions
$\varphi:[0,T]\times\bR^d\to[0,\infty)$ that the mapping
\begin{equation}                                  \label{measurability 2}
\Lambda\times\bR^d\ni(t,s,x)\mapsto
\bE\Big[\varphi\big(s,Y^{d,\theta,t,x,N,\delta,\cM}_{s}\big)\Big]
\in[0,\infty]
\end{equation}
is measurable.
Furthermore, by Lemma \ref{lemma est moment}
and Lemma \ref{Lemma stability},
we observe that for all $d\in \bN$, $\theta\in\Theta$,
$t\in[0,T]$, $s\in[t,T]$, and $x,y\in\bR^d$
\begin{equation}                                          \label{est theta}
\bE\Big[d^p+\big\|X^{d,\theta,t,x}_{s}\big\|^2\Big]\leq 
c_1(d^p+\|x\|^2)e^{\rho_1(s-t)},
\end{equation}
and
\begin{equation}                                      \label{stability initial}
\bE\Big[\big\|X^{d,\theta,t,x}_{s}-X^{d,\theta,t,y}_{s}\big\|^2\Big]
\leq 4\|x-y\|^2e^{\rho_2(s-t)},
\end{equation}
where $c_1:=1+6LT$, $\rho_1:=1+6L$, and $\rho_2:=4L(T+2)$.
Moreover, \eqref{estimate M} in Lemma \ref{lemma delta M}
ensures for all $\theta\in\Theta$, $d\in\bN$,
$\delta\in(0,1)$, $(t,x)\in[0,T]\times\bR^d$, $s\in[t,T]$, and
$N,\cM\in\bN$ with $\cM\geq \delta^{-2}Kd^p$ 
that 
\begin{equation}                                         \label{est theta Euler}
\bE\Big[d^p+\big\|Y^{d,\theta,t,x,N,\delta,\cM}_{s}\big\|^2\Big]
\leq c_3(d^p+\|x\|^2)e^{\rho_3(s-t)},
\end{equation}
where $c_3:=5\cdot\max\{1,4LT(T+8)\}$, and $\rho_3:=20L(T+8)$. 
Then, combining \eqref{assumption Lip f g}, \eqref{assumption growth f g}, 
\eqref{measurability 1}, \eqref{est theta}, 
and applying Proposition \ref{Prop FP} (with
$Z^{t,x}_s\cal X^{d,0,t,x}_{s}$, 
$a\cal \max\{L^{1/2},L^{1/2}T^{-1/2}\}$,
$b \cal c_1^{1/2}e^{\rho_1T/2}$, 
$c \cal L^{1/2}$,
$f \cal f^d$, $g \cal g^d$, $u \cal u^d$ in the notation of Proposition
\ref{Prop FP}),
we obtain (ii) and
\begin{equation}                                               \label{u^d est}
\frac{|u^d(t,x)|}{\sqrt{d^p+\|x\|^2}}\leq 2L^{1/2}
\exp\big\{L^{1/2}(T-t)\big\}.
\end{equation}
Analogously, combining \eqref{assumption Lip f g}, \eqref{assumption growth f g}, 
\eqref{measurability 2}, and \eqref{est theta Euler},
and applying Proposition \ref{Prop FP} (with
$Z^{t,x}_s\cal Y^{d,0,t,x,N,\delta,\cM}_{s}$, $b \cal c_3^{1/2}e^{\rho_3T/2}$, 
$c \cal L^{1/2}$,
$f \cal f^d$, $g \cal g^d$, $u \cal u^d_{N,\delta,\cM}$ 
in the notation of Proposition
\ref{Prop FP}),
we obtain for each $d\in\bN$, $\delta\in(0,1)$, 
and $N,\cM\in\bN$ with $\cM\geq \delta^{-2}Kd^p$ that
there exists an unique
Borel function $u^d_{N,\delta,\cM}:[0,T]\times\bR^d\to\bR$ satisfying for all
$(t,x)\in[0,T]\times\bR^d$ that

\begin{align}
&
\bE\Big[\big|g^d(Y^{d,0,t,x,N,\delta,\cM}_{T})\big|\Big]
+\int_t^T\bE\Big[\big|f^d(s,Y^{d,0,t,x,N,\delta,\cM}_{s},
u^d_{N,\delta,\cM}(s,Y^{d,0,t,x,N,\delta,\cM}_{s}))\big|\Big]\,ds<\infty,
\label{M 1}
\\
&
\sup_{y\in\bR^d}\sup_{s\in[0,T]}
\left(\frac{|u^d_{N,\delta,\cM}(s,y)|}{\sqrt{d^p+\|y\|^2}}
\right)<\infty,
\label{M 2}
\\
&
u^d_{N,\delta,\cM}(t,x)=\bE\Big[g^d(Y^{d,0,t,x,N,\delta,\cM}_{T})\Big]
+\int_t^T\bE\Big[f^d(s,Y^{d,0,t,x,N,\delta,\cM}_{s}
,u^d_{N,\delta,\cM}(s,Y^{d,0,t,x,N,\delta,\cM}_{s}))\Big]\,ds,
\label{M 3}
\\
&
\frac{|u^d_{N,\delta,\cM}(t,x)|}{\sqrt{d^p+\|x\|^2}}\leq 2L^{1/2}
\exp\big\{L^{1/2}(T-t)\big\}.
\label{M 4}
\end{align}
Furthermore, we notice that Proposition \ref{proposition uniqueness PIDE} ensures that
for each $d\in\bN$ there exists at most one viscosity solution 
$v^d\in C_1([0,T]\times\bR^d)$.
Hence, combining this with Proposition \ref{proposition PIDE existence} we
obtain (iii) and (iv). 

Next, we start to prove (v). 
We first observe that by \eqref{error Euler}, \eqref{est error Y delta} 
and \eqref{error M}, there exist a positive constant $c_4=c_4(T,L,L_1,L_2,K)$ 
satisfying for all $\theta\in\Theta$, $d\in\bN$,
$\delta\in(0,1)$, $(t,x)\in[0,T]\times\bR^d$, and
$N,\cM\in\bN$ with $\cM\geq \delta^{-2}Kd^p$ 
that 
\begin{equation}                                       \label{Euler error theta}
\bE\Bigg[\sup_{s\in[t,T]}
\left\|X^{d,\theta,t,x}_{s}-Y^{d,\theta,t,x,N,\delta,\cM}_{s}\right\|^2\Bigg]
\leq c_4(d^p+\|x\|^2)\left(N^{-1}T+\delta^qd^p+\delta^{-2}Kd^p\cM^{-1}\right).
\end{equation}
Furthermore, we notice that it holds 
for all $d\in\bN$, $\delta\in(0,1)$, $n,M,N,\cM\in\bN$,
and $(t,x)\in[0,T]\times\bR^d$ that
\begin{align}                                   
\left(\bE\Big[\big|U^{d,0,\delta,\cM}_{n,M,N}(t,x)-u^d(t,x)\big|^2
\Big]\right)^{1/2}
\leq
& 
\left(\bE\Big[\big|U^{d,0,\delta,\cM}_{n,M,N}(t,x)-u^d_{N,\delta,\cM}(t,x)\big|^2
\Big]\right)^{1/2}\nonumber\\
&
+\big|u^d_{N,\delta,\cM}(t,x))-u^d(t,x)\big|.
\label{error norm 1}                        
\end{align}
Considering \eqref{assumption growth f g}, \eqref{est theta Euler},
\eqref{M 1}-\eqref{M 4}, 
and the fact that 
$1+a\leq e^a$ for all $a\in\bR$, and applying Lemma \ref{Lemma MLP} (with
$\mathcal{U}^{d,0}_{n,M}\cal U^{d,0,\delta,\cM}_{n,M,N}$,
$Z^{d,0,t,x}_s\cal Y^{d,0,t,x,N,\delta,\cM}_s$,
$b\cal c_3$, $\rho\cal \rho_3$, $u^d\cal u^d_{N,\delta,\cM}$,
$c\cal 3\cdot\big(4L\exp\{2L^{1/2}T\}\vee T^{-2}L\vee L\big)$
in the notation of Lemma \ref{Lemma MLP}), 
we have for all $d\in\bN$, $\delta\in(0,1)$, $n\in\bN_0$, $t\in[0,T]$, and
$M,N,\cM\in\bN$ with $\cM\geq \delta^{-2}Kd^p$ that
\begin{align}
\sup_{x\in\bR^d} &
\left(
\frac{\bE\left[\big|U^{d,0,\delta,\cM}_{n,M,N}(t,x)
-u^d_{N,\delta,\cM}(t,x)\big|^2\right]}{d^p+\|x\|^2}
\right)^{1/2}\leq
2e^{M/2}M^{-n/2}(1+2TL^{1/2})^{n-1}c_3^{1/2}e^{\rho_3(T-t)/2}\nonumber\\
&
\cdot \left[\sup_{t\in[0,T]}\sup_{x\in\bR^d}
\left(\frac{\max\{T(F^d(0)(t,x)),|g^d(x)|\}}{\sqrt{d^p+\|x\|^2}}\right)
+TL^{1/2}\sup_{t\in[0,T]}\sup_{x\in\bR^d}
\left(\frac{u^d_{N,\delta,\cM}(t,x)}{\sqrt{d^p+\|x\|^2}}
\right)\right]\nonumber\\
&
\leq 2e^{M/2}M^{-n/2}(1+2TL^{1/2})^{n-1}c_3^{1/2}e^{\rho_3(T-t)/2}
\left(L^{1/2}+2TL\exp\{L^{1/2}(T-t)\}\right)
\nonumber\\
&
\leq 2e^{M/2}M^{-n/2}(1+2TL^{1/2})^{n-1}e^{\rho_3(T-t)/2}
(c_3L)^{1/2}\left(1+2TL^{1/2}
\exp\{L^{1/2}(T-t)\}\right)\nonumber\\
&
\leq 
C_1e^{M/2}M^{-n/2}e^{2nTL^{1/2}},
\label{error 1}
\end{align}
where
$C_1:=2e^{\rho_3(T-t)/2}
(c_3L)^{1/2}\left(1+2TL^{1/2}
\exp\{L^{1/2}T\}\right)$.
To deal with the second term in \eqref{error norm 1}, we notice that
by \eqref{stability initial} and \eqref{Euler error theta},
for all $d\in\bN$, $\delta\in(0,1)$, $(t,x)\in[0,T]\times\bR^d$, 
$s\in[t,T]$, $r\in[s,T]$ and $N,\cM\in\bN$ with $\cM\geq \delta^{-2}Kd^p$ 
it holds that
\begin{align*}
&
\bE\left[\bE\left[\left\|X^{d,0,s,x'}_{r}-X^{d,0,s,y'}_{r}\right\|^2\right]
\Big|_{(x',y')=(X^{d,0,t,x}_{s},Y^{d,0,t,x,N,\delta,\cM}_{s})}\right]\\
&
\leq 4e^{\rho_2(r-s)}
\bE\left[\left\|X^{d,0,t,x}_{s}-Y^{d,0,t,x,N,\delta,\cM}_{s}\right\|^2\right]\\
&
\leq 4c_4(d^p+\|x\|^2)e^{\rho_2(r-s)}
(N^{-1}T+\delta^qd^p+\delta^{-2}Kd^p\cM^{-1}).
\end{align*}
Moreover, by \eqref{stability initial} 
we notice that for each $d\in\bN$, $s\in[0,T]$, and $r\in[s,T]$ the mapping 
$
\bR^d\times\bR^d\ni(x,y)\mapsto \big(X^{d,0,s,x}_r,X^{d,0,s,y}_r\big)
\in \cL_0(\Omega,\bR^d\times\bR^d)
$
is continuous and hence measurable, 
and we have for all nonnegative Borel functions 
$h:\bR^d\times\bR^d\to[0,\infty)$ that the mapping
$
\cL_0(\Omega,\bR^d\times\bR^d)\ni Z \mapsto \bE\big[h(Z)\big]\in [0,\infty]
$
is measurable. 
Hence, it holds for all $d\in\bN$, $s\in[0,T]$,
$r\in[s,T]$ and all nonnegative Borel functions 
$h:\bR^d\times\bR^d\to[0,\infty)$ that the mapping
\begin{equation}                                           \label{measurability 3}
\bR^d\times\bR^d\ni(x,y)\mapsto
\bE\Big[h\big(X^{d,0,s,x}_{r},X^{d,0,s,y}_{r}\big)\Big]\in[0,\infty]
\end{equation}
is measurable.
Furthermore, Lemma 2.2 in \cite{HJKNW2018} ensures that for all 
$d\in\bN$, $t\in[0,T]$, $s\in[t,T]$, $r\in[s,T]$, $x,y\in\bR^d$ and
all nonnegative Borel functions $h:\bR^d\times\bR^d\to[0,\infty)$
it holds that
\begin{equation}                                     \label{expectation equality}
\bE\left[\bE \left[h\Big(X^{d,0,s,x'}_{r},X^{d,0,s,y'}_{r}\Big)\right]
\Big|_{(x',y')=(X^{d,0,t,x}_{s},X^{d,0,t,y}_{s})}\right]
=\bE\left[h\Big(X^{d,0,t,x}_{r},X^{d,0,t,y}_{r}\Big)\right].
\end{equation}
Then using \eqref{assumption Lip f g}, \eqref{est theta},
\eqref{est theta Euler}, \eqref{Euler error theta}, \eqref{u^d est},
\eqref{M 1}-\eqref{M 4},
\eqref{measurability 3}, \eqref{expectation equality},
and (iv), 
and applying Lemma \ref{lemma perturbation} with
$T\cal T$, $L\cal L^{1/2}$, 
$X^{t,x,1} \cal X^{d,0,t,x}$,
$X^{t,x,2}\cal 
Y^{d,0,t,x,N,\delta,\cM}$,
$\eta\cal (c_1\vee c_3)e^{(\rho_1\vee \rho_3)T}$,
$\delta\cal \big[4c_4e^{\rho_2 T}(N^{-1}T
+\delta^qd^p+\delta^{-2}Kd^p\cM^{-1})\big]^{1/2}$,
$c\cal c_0:=c_1\vee c_3\vee \big(2L^{1/2}\exp\{L^{1/2}T\}\big)$,
$\rho\cal \rho_1\vee\rho_3$
in the notation of Lemma \ref{lemma perturbation},
we obtain for all $d,N\in\bN$, $\delta\in(0,1)$, 
$(t,x)\in[0,T]\times\bR^d$, 
and $\cM\in\bN$ with $\cM\geq \delta^{-2}Kd^p$ that
\begin{align}
&
\big|u^d_{N,\delta,\cM}(t,x)-u^d(t,x)\big|
\nonumber\\
&
\leq 
4c_0\big[4c_4e^{\rho_2 T}(N^{-1}T
+\delta^qd^p+\delta^{-2}Kd^p\cM^{-1})\big]^{1/2}
(TL)^{1/2}(1+L^{1/2}T)(d^p+\|x\|^2)
\nonumber\\
&
\quad \cdot\exp\left\{\left(L^{1/2}+(\rho_1\vee \rho_3)/2
+(Lc_0)^{1/2}(c_1\vee c_3)^{1/2}e^{(\rho_1 \vee \rho_3)T/2}\right)(T-t)
\right\}\nonumber\\
&
\leq C_2(d^p+\|x\|^2)(N^{-1}T+\delta^qd^p+\delta^{-2}Kd^p\cM^{-1})^{-1/2},
\label{error 2}
\end{align}
where 
\begin{align*}
C_2:=
&
\exp\left\{\left(L^{1/2}+(\rho_1\vee \rho_3)/2
+(Lc_0)^{1/2}(c_1\vee c_3)^{1/2}e^{(\rho_1 \vee \rho_3)T/2}\right)T
\right\}\\
&
\cdot 8c_0\big(c_4TLe^{\rho_2T}\big)^{1/2}\big(1+L^{1/2}T\big).
\end{align*}
Hence, combining \eqref{error norm 1}, \eqref{error 1}, 
and \eqref{error 2} yields
for for all $d\in\bN$, $\delta\in(0,1)$, $n\in\bN_0$, $(t,x)\in[0,T]\times\bR^d$, 
and $M,N,\cM\in\bN$ with $\cM\geq \delta^{-2}Kd^p$ that
\begin{align*}
&
\Big(\bE\left[\big|U^{d,0,\delta,\cM}_{n,M,N}(t,x)
-u^d(t,x)\big|^2\right]\Big)^{1/2}\\
&
\leq
(C_1\vee C_2)(d^p+\|x\|^2)
\left[e^{M/2}M^{-n/2}e^{2nTL^{1/2}}
+(N^{-1}T+\delta^qd^p+\delta^{-2}Kd^p\cM^{-1})^{1/2}\right],
\end{align*}
which proves (v) with $c=C_1\vee C_2$. 
Thus, the proof of this theorem is complete.
\end{proof}

\begin{proof}[Proof of Theorem \ref{MLP complexity}]
Applying (v) in Theorem 4.2 in \cite{HJKN2020} with 
$n\cal n$, $M\cal M$, $\mathfrak{C}_{n,M}\cal \mathfrak{C}_{n,M}^{(d)}$,
$\mathfrak{m}\cal \mathfrak{e}^{(d)}, \mathfrak{g}\cal\mathfrak{g}^{(d)}$,
$\mathfrak{f}\cal\mathfrak{f}^{(d)}$
in the notation of Theorem 4.2 in \cite{HJKN2020},
we obtain \eqref{cc 2}.
Next, for each $d\in\bN$, $\varepsilon\in(0,1]$, and $x\in\bR^d$ define
$\mathbf{n}^d(x,\varepsilon)$ by
\begin{equation}                                                \label{def n}
\mathbf{n}^d(x,\varepsilon):=\inf\{n\in\bN\cap [2,\infty):
\sup_{k\in[n,\infty)\cap\bN}\sup_{t\in[0,T]}
\bE\left[\big|U^{d,0}_{(k)}(t,x)-u^d(t,x)\big|^2\right]<\varepsilon^2\},
\end{equation}
where we use the convention $\inf(\emptyset)=\infty$
($\emptyset$ denotes the empty set), 
and use the shorter notation
$$
U^{d,0}_{(k)}(t,x):=U^{d,0,k^{-k/q},k^{k+2k/q}Kd^p}_{k,k,k^k}(t,x),\quad
\text{$\forall d\in\bN$, $k\in\bN$, $(t,x)\in[0,T]\times\bR^d$}.
$$
Applying (v) in Theorem \ref{MLP conv} (with
$n \cal n$, $M\cal n$, $N\cal n^n$, $\delta\cal n^{-n/q}$, 
$\cM\cal n^{n+2n/q}Kd^p$
in the notation of Theorem \ref{MLP conv}),
we have for all $d\in\bN$, 
$n\in\bN$, and $(t,x)\in[0,T]\times\bR^d$ that
\begin{align}
&
\left(\bE\left[\big|U^{d,0}_{(n)}(t,x)-u^d(t,x)\big|^2\right]\right)^{1/2}
\nonumber\\
&
\leq
c(d^p+\|x\|^2)\left[e^{n/2}n^{-n/2}e^{2nTL^{1/2}}
+(n^{-n}\cdot T+d^pn^{-n}+n^{-n})^{1/2}\right],               \label{MLP n}
\end{align}
where $c=c(T,L,L_1,L_2,K)$
is the constant introduced in (v) in Theorem \ref{MLP conv}.
Moreover, we observe for all integers $n\geq 2\exp\{4TL^{1/2}+1\}$ that
$$
e^{n/2}e^{2nTL^{1/2}}n^{-n/2}
=\exp\left\{(4TL^{1/2}+1)\cdot n/2\right\}\cdot n^{-n/2}
\leq 2^{-n/2},
$$
which implies that
$$
\lim_{n\to\infty}e^{n/2}e^{2nTL^{1/2}}n^{-n/2}=0.
$$
Therefore, by \eqref{MLP n} we have for all $d\in\bN$, $n\in\bN$, $\varepsilon\in(0,1]$, and $(t,x)\in[0,T]\times\bR^d$ that
$$
\mathbf{n}^d(x,\varepsilon)<\infty \quad \text{and} \quad
\sup_{n\in[\mathbf{n}^d(x,\varepsilon),\infty)\cap \bN}
\left(\bE\left[\big|U^{d,0}_{(n)}(t,x)
-u^d(t,x)\big|^2\right]\right)^{1/2}<\varepsilon,
$$
which establishes \eqref{error epsilon}.
Furthermore, by \eqref{cc 2} and \eqref{MLP n}
we obtain for all $d,n\in\bN$, $\gamma\in(0,1]$, 
and $(t,x)\in[0,T]\times\bR^d$ that
\begin{align}
&
\left(\sum_{k=1}^{n+1}\mathfrak{C}^{(d)}_{k,k}\right)
\left(\bE\left[\big|U^{d,0}_{(n)}(t,x)-u^d(t,x)\big|^2
\right]\right)^{\frac{\gamma+4}{2}}
\nonumber\\
&
\leq 12(3\mathfrak{e}^{(d)}+\mathfrak{g}^{(d)}+2\mathfrak{f}^{(d)})36^nn^{2n} 
\big[c(d^p+\|x\|^2)\big]^{\gamma+4}
(e^{n/2}n^{-n/2}e^{2nTL^{1/2}}+(1+T+d^p)^{1/2}n^{-n/2})^{\gamma+4}
\nonumber\\
& 
= 12(3\mathfrak{e}^{(d)}+\mathfrak{g}^{(d)}+2\mathfrak{f}^{(d)})36^n
n^{-\gamma n/2}\big[c(d^p+\|x\|^2)\big]^{\gamma+4}
(e^{n/2}e^{2nTL^{1/2}}+(1+T+d^p)^{1/2})^{\gamma+4}.                     
\label{cc 4}
\end{align} 
Then \eqref{def n} and \eqref{cc 4} show for all $d\in\bN$,
$\varepsilon,\gamma\in(0,1]$ and $(t,x)\in[0,T]\times\bR^d$ that
\begin{align*}
\left(\sum_{k=1}^{\mathbf{n}^d(x,\varepsilon)}\mathfrak{C}^{(d)}_{k,k}\right)
\varepsilon^{\gamma+4}
&
\leq \left(\sum_{k=1}^{\mathbf{n}^d(x,\varepsilon)}\mathfrak{C}^{(d)}_{k,k}\right)
\left(\bE\left[
\big|U^{d,0}_{(\mathbf{n}^d(x,\varepsilon)-1)}(t,x)-u^d(t,x)\big|^2
\right]\right)^{\frac{\gamma+4}{2}}\\
&
\leq 12(3\mathfrak{e}^{(d)}+\mathfrak{g}^{(d)}+2\mathfrak{f}^{(d)})
\big[c(d^p+\|x\|^2)\big]^{\gamma+4}
\nonumber\\
&
\quad
\cdot \sup_{n\in\bN}
\left(36^n
n^{-\gamma n/2}(e^{n/2}e^{2nTL^{1/2}}+(1+T+d^p)^{1/2})^{\gamma+4}\right).
\end{align*}
Moreover, note that we have for all $\gamma\in(0,1]$ and $n\in\bN$ satisfying 
$$
n\geq \left(2\cdot6^4\cdot\exp\{(1+4TL^{1/2})(\gamma+4)\}\right)^{1/\gamma}
$$ that
\begin{align*}
36^nn^{-\gamma n/2}(e^{n/2}e^{2nTL^{1/2}})^{\gamma+4}
&
=6^{4n/2}\cdot\exp\{(1+4TL^{1/2})(\gamma+4)n/2\}\cdot n^{-\gamma n/2}\\
&
=\left[6^4\cdot\exp\{(1+4TL^{1/2})(\gamma+4)\}\right]^{n/2}\cdot n^{-\gamma n/2}\\
&
\leq 2^{-n/2},
\end{align*}
which implies for all $\gamma\in(0,1]$ that
$$
\sup_{n\in\bN}
\left(36^n
n^{-\gamma n/2}(e^{n/2}e^{2nTL^{1/2}}+(1+T+d^p)^{1/2})^{\gamma+4}\right)<\infty.
$$
Hence, the proof of this theorem is complete.
\end{proof}

\clearpage

\appendix

\section{\textbf{Preliminaries: dimension-depending bounds for SDEs with jumps}}
\label{section pre}

Throughout this chapter, we assume the settings in Section \ref{section setting},
and fix $d$ and $\theta$. To ease notations, we will omit the superscripts $d$ and
$\theta$ for the notations introduced in Section \ref{section setting}
(e.g., for each $(t,x)\in[0,T]\times\bR^d$, and $s\in[t,T]$, 
$X^{d,\theta,t,x}_{s}$ 
will be denoted by $X^{t,x}_{s}$).
We follow \cite{GS2021} which provides dimension-depending
bounds for (time discretization of) SDEs with jumps but without time dependence
for the coefficients.

\subsection{Bounds for non-discretized SDEs}

\begin{lemma}                                             \label{lemma est moment}
For all $(t,x)\in\bR^d$ and $s\in[t,T]$ it holds that
\begin{equation}                                            \label{est 1+}
\bE\big[\|X^{t,x}_s\|^2\big]\leq (6Ld^p(s-t)+\|x\|^2)e^{(1+6L)(s-t)}.
\end{equation}
\end{lemma}
\begin{proof}
We fix $(t,x)\in[0,T]\times\bR^d$ throughout this proof. 
For $X^{t,x}=(X^{t,x,1},...,X^{t,x,d})$,
applying It\^o's formula 
(see. e.g. Theorem 3.1 in \cite{GW2019}) to $\|X^{t,x}\|^2$ yields
that almost surely for all $s\in[t,T]$ 
\begin{align}
\|X^{t,x}_{s}\|^2=
&
\|x\|^2+\int_t^s\left(2\langle X^{t,x}_{r-},\mu(r,X^{t,x}_{r-})\rangle
+\|\sigma(r,X^{t,x}_{r-})\|_F^2\right)\,dr\nonumber\\
&                                
+2\int_t^s\sum_{i=1}^d
\sum_{j=1}^dX^{t,x,i}_{r-}\sigma^{ij}(r,X^{t,x}_{r-})\,dW^j_r
+2\int_t^s\int_{\bR^{d}}\langle X^{t,x}_{r-},\eta_r(X^{t,x}_{r-},z)\rangle
\,\tilde{\pi}(dz,dr)\nonumber\\
&
+\int_t^s\int_{\bR^{d}}\|\eta_r(X^{t,x}_{r-},z)\|^2\,\pi(dz,dr).     \label{Ito1}
\end{align}
Next, for each $n\in \bN$
define the following stopping time 
\begin{align}
\tau^{t,x}_n:=\inf
\Big\{
&
s\geq t:\int_t^s\sum_{i=1}^d\sum_{j=1}^d
\big|2X^{t,x,i}_r\sigma^{ij}(r,X^{t,x}_r)\big|^2\,dr\nonumber\\
&
+\int_t^s\int_{\bR^{d}}\sum_{i=1}^d
\big|2\eta^i_r(X^{t,x}_r,z)X^{t,x,i}_{r}\big|^2
\,\nu(dz)\,dr
\geq n
\Big\}\wedge T.                                                        \label{def tau}
\end{align}
Then, with \eqref{def tau} we observe that the 
stochastic integrals with respect to the Brownian motion and the compensated
Poisson random measure, respectively, on the right hand side of \eqref{Ito1} 
are true martingales on $\llbracket 0,\tau^{t,x}_n \rrbracket$ for each $n\in\bN$.
Thus, by \eqref{Ito1} we obtain for all $s\in[t,T]$ and $n\in\bN$ that
\begin{align}
\bE\left[\left\|X^{t,x}_{s\wedge \tau^{t,x}_n}\right\|^2\right]=
&
\|x\|^2+2\bE\left[\int_t^{s\wedge \tau^{t,x}_n}
\langle X^{t,x}_{r},\mu(r,X^{t,x}_r) \rangle\,dr\right]
+\bE\left[\int_t^{s\wedge \tau^{t,x}_n}\|\sigma(r,X^{t,x}_{r})\|_F^2\,dr\right]
\nonumber\\
&                                                  
+\bE\left[\int_t^{s\wedge \tau^{t,x}_n}\int_{\bR^{d}}\|\eta_r(X^{t,x}_{r},z)\|^2
\,\nu(dz)\,dr\right].
\label{Ito2}
\end{align}
Considering \eqref{assumption Lip mu sigma eta}, \eqref{assumption growth}, 
and \eqref{Ito2},
by Young's inequality and the fact that $(a+b)^2\leq 2a^2+2b^2$ 
for all $a,b\in\bR$, we have for all $s\in[t,T]$ and $n\in\bN$ that
\begin{align*}
\bE\Big[\left\|X^{t,x}_{s\wedge \tau^{t,x}_n}\right\|^2\Big]
&
\leq \|x\|^2
+\int^s_t\bE\Big[\left\|X^{t,x}_{r\wedge\tau^{t,x}_n}\right\|^2\Big]\,dr\\
&
\quad+2\bE\left[\int_t^{s\wedge \tau^{t,x}_n}
\left\|\mu(r,X^{t,x}_r)-\mu(r,0)\right\|^2\,dr\right]
+2\int_t^{s}\left\|\mu(r,0)\right\|^2\,dr\\
&
\quad+2\bE\left[\int_t^{s\wedge \tau^{t,x}_n}
\left\|\sigma(r,X^{t,x}_r)-\sigma(r,0)\right\|_F^2\,dr\right]
+2\int_t^{s}\left\|\sigma(r,0)\right\|_F^2\,dr\\
&
\quad+2\bE\left[\int_t^{s\wedge \tau^{t,x}_n}
\int_{\bR^{d}}\left\|\eta_r(X^{t,x}_r,z)-\eta_r(0,z)\right\|^2
\,\nu(dz)\,dr\right]\\
&
\quad+2\int_t^{s}
\int_{\bR^d}\left\|\eta_r(0,z)\right\|^2\,\nu(dz)\,dr\\
&
\leq \|x\|^2+6Ld^p(s-t)
+(1+6L)\int^s_t\bE\left[\left\|X^{t,x}_{r\wedge\tau^{t,x}_n}\right\|^2\right]dr.
\end{align*}
Therefore, by \eqref{SDE moment est} and Gr\"onwall's lemma we have
for all $s\in[t,T]$ and $n\in\bN$ that
\begin{align*}
\bE\left[\left\|X^{t,x}_{s\wedge \tau^{t,x}_n}\right\|^2\right]
&
\leq (6Ld^p(s-t)+\|x\|^2)e^{(1+6L)(s-t)}.
\end{align*}
Finally, the application of Fatou's lemma yields for all $s\in[t,T]$ that
$$
\bE\left[\left\|X^{t,x}_{s}\right\|^2\right]
\leq 
\liminf_{n\to\infty}\bE\left[\left\|X^{t,x}_{s\wedge\tau^{t,x}_n}\right\|^2\right]
\leq (6Ld^p(s-t)+\|x\|^2)e^{(1+6L)(s-t)},
$$
which completes the proof of this lemma.
\end{proof}

\begin{lemma}                                            \label{Lemma stability}
There exist a constant $\rho$ only
depending on $L$ and $T$, satisfying for all $t\in[0,T]$, $s\in[t,T]$, 
and $x,y\in\bR^d$ that
$$
\bE\Big[\big\|X^{t,x}_s-X^{t,y}_s\big\|^2\Big]
\leq 4\|x-y\|^2e^{\rho(s-t)}.
$$
\end{lemma}

\begin{proof} 
We first fix $t\in[0,T]$ and $x,y\in\bR^d$.
Notice for every $s\in[t,T]$ that
\begin{align*}
X^{t,x}_s-X^{t,y}_s=  
&
x-y+\int_t^s\left(\mu(r,X^{t,x}_{r-})-\mu(r,X^{t,y}_{r-})\right)dr
+\int_t^s\left(\sigma(r,X^{t,x}_{r-})-\sigma(r,X^{t,y}_{r-})\right)dW_r\\
&
+\int_t^s\int_{\bR^d}\left(
\eta_r(X^{t,x}_{r-},z)-\eta_r(X^{t,y}_{r-},z)
\right)\tilde{\pi}(dz,dr).
\end{align*}
Then by Jensen's inequality, It\^o's isometry, and
the fact that $(\sum_{i=1}^4a_i)^2\leq 4\sum_{i=1}^4a_i^2$ for
$a_i\in\bR$, $i=1,2,3,4$, it holds for all $s\in[t,T]$ that
\begin{align*}
\bE\Big[\big\|X^{t,x}_s-X^{t,y}_s\big\|^2\Big]\leq 
&
4\|x-y\|^2
+4(s-t)\int_t^s\bE\left[
\big\|\mu(r,X^{t,x}_r)-\mu(r,X^{t,y}_r)\big\|^2
\right]dr\\
&
+4\int_t^s\bE\left[
\big\|\sigma(r,X^{t,x}_r)-\sigma(r,X^{t,y}_r)\big\|_F^2
\right]dr\\
&
+4\int_t^s\int_{\bR^d}\bE\left[
\big\|\eta_r(X^{t,x}_r,z)-\eta_r(X^{t,y}_r,z)\big\|^2
\right]\nu(dz)\,dr.
\end{align*} 
Thus, by \eqref{assumption Lip mu sigma eta} we obtain for all $s\in[t,T]$
that
\begin{align*}
\bE\Big[\big\|X^{t,x}_s-X^{t,y}_s\big\|^2\Big]\leq
& 
4\|x-y\|^2+4(s-t)L\int_t^s\bE\Big[\big\|X^{t,x}_r-X^{t,y}_r\big\|^2\Big]\,dr\\
&
+8L\int_t^s\bE\Big[\big\|X^{t,x}_r-X^{t,y}_r\big\|^2\Big]\,dr.
\end{align*}
Therefore, by \eqref{SDE moment est} and Gr\"onwall's lemma, 
we have for all $s\in[t,T]$ that
$$
\bE\Big[\big\|X^{t,x}_s-X^{t,y}_s\big\|^2\Big]\leq
4\|x-y\|^2e^{4L(T+2)(s-t)},
$$
which completes the proof of this lemma.
\end{proof}

\begin{lemma}                                           \label{lemma s s'}
There exists a positive constant $c$ only depending on $T$ and $L$ satisfying
for all $(t,x)\in[0,T]\times\bR^d$, $s\in[t,T]$, and $s'\in[s,T]$ that
\begin{equation}                                       \label{est s s'}
\bE\left[\big\|X^{t,x}_{s'}-X^{t,x}_s\big\|^2\right]
\leq c\left[(s'-s)+(s'-s)^2\right](d^p+\|x\|^2).
\end{equation}
\end{lemma}

\begin{proof}
We notice that almost surely for all $(t,x)\in[0,T]\times\bR^d$, $s\in[t,T]$ 
and $s'\in[s,T]$ that
$$
X^{t,x}_{s'}=
X^{t,x}_s+\int_s^{s'}\mu(r,X^{t,x}_{r-})\,dr
+\int_s^{s'}\sigma(r,X^{t,x}_{r-})\,dW_r
+\int_s^{s'}\int_{\bR^{d}}\eta_r(X^{t,x}_{r-},z)\,\tilde{\pi}(dz,dr).
$$
Hence, by It\^o's isometry, H\"older's inequality, 
\eqref{assumption Lip mu sigma eta}, \eqref{assumption growth} 
and Lemma \ref{lemma est moment} it holds for all $(t,x)\in[0,T]\times\bR^d$, $s\in[t,T]$ 
and $s'\in[s,T]$ that
\begin{align}
\bE\left[\big\|X^{t,x}_{s'}-X^{t,x}_s\big\|^2\right]
&
\leq 3(s'-s)\int_s^{s'}\bE\left[\left\|\mu(r,X^{t,x}_r)\right\|^2\right]\,dr
+3\int_s^{s'}\bE\left[\left\|\sigma(r,X^{t,x}_r)\right\|_F^2\right]\,dr
\nonumber\\
&
\quad +3\int_s^{s'}\bE\left[\int_{\bR^{d}}
\left\|\eta_r(X^{t,x}_r,z)\right\|^2\,\nu(dz)\right]dr
\nonumber\\
&
\leq 3(s'-s)\int_s^{s'}\bE\left[2\left\|\mu(r,X^{t,x}_r)-\mu(r,0)\right\|^2
+2\left\|\mu(r,0)\right\|^2\right]\,dr
\nonumber\\
&
\quad +3\int_s^{s'}\bE\left[2\left\|\sigma(r,X^{t,x}_r)-\sigma(r,0)\right\|_F^2
+2\left\|\sigma(r,0)\right\|^2\right]\,dr
\nonumber\\
&
\quad +3\int_s^{s'}\bE\left[\int_{\bR^{d}}
\left(2\left\|\eta_r(X^{t,x}_r,z)-\eta_r(0,z)\right\|^2
+2\left\|\eta_r(0,z)\right\|^2\right)\nu(dz)\right]dr\nonumber\\
&
\leq 3(s'-s)\int_s^{s'}\left(2L\bE\left[\|X^{t,x}_r\|^2\right]+2Ld^p\right)\,dr
\nonumber\\
&
\quad+6\int_s^{s'}\left(2L\bE\left[\|X^{t,x}_r\|^2\right]+2Ld^p\right)\,dr
\nonumber\\
&
\leq c(s'-s)^2(d^p+\|x\|^2)+c(s'-s)(d^p+\|x\|^2) \nonumber\\
\end{align}
where $c=12L(1+6LT)e^{(1+6L)T}$. 
\end{proof}

\begin{lemma}                                   \label{Lemma L2 continuity}
For all $t\in[0,T]$, $t'\in[t,T]$,
$s\in[t',T]$, and $x,x'\in\bR^d$ it holds that
\begin{align}   
\bE\left[\left\|X^{t,x}_s-X^{t',x'}_s\right\|^2\right]                                              
\leq & 9(\|x-x'\|+|t-t'|+|t-t'|^{1/2})^2
\cdot \exp \{6LT(T+4)\}\nonumber\\
&                                  
\cdot\left[1+3L^{1/2}d^{p/2}(1+\|x'\|^2)^{1/2}\right]^2.                   \label{est t t'}
\end{align}
\end{lemma}

\begin{proof}
We follow the idea of the proof of Lemma 3.6 in \cite{BGHJ2019},
and fix $t\in[0,T]$, $t'\in[t,T]$, $s\in[t',T]$, and $x,x'\in\bR^d$
throughout this proof.
One first notice that almost surely
$$
X^{t,x}_s-X^{t',x'}_s
=X^{t,x}_{t'}-x'
+\int_{t'}^s\left(\mu(r,X^{t,x}_{r-})-\mu(r,X^{t',x'}_{r-})\right)\,dr
$$
$$
+\int_{t'}^s\left(\sigma(r,X^{t,x}_{r-})-\sigma(r,X^{t',x'}_{r-})\right)\,dW_r
+\int_{t'}^s\int_{\bR^{d}}
\left(\eta_r(X^{t,x}_{r-},z)-\eta_r(X^{t',x'}_{r-},z)\right)
\,\tilde{\pi}(dz,dr).
$$
Then by Minkowski's inequality and H\"older's inequality we have
\begin{align*}
\left(\bE\left[\left\|X^{t,x}_s-X^{t',x'}_s\right\|^2\right]\right)^{1/2}
\leq
& 
\left(\bE\left[\left\|X^{t,x}_{t'}-x'\right\|^2\right]\right)^{1/2}\\
&
+\left(\int_{t'}^s(s-t')\bE\left[
\left\|\mu(r,X^{t,x}_r)-\mu(r,X^{t',x'}_r)\right\|^2
\right]dr\right)^{1/2}\\
&
+\left(\bE\left[\left\|\int_{t'}^s
\left(\sigma(r,X^{t,x}_{r-})-\sigma(r,X^{t',x'}_{r-})\right)
\,dW_r\right\|^2\right]\right)^{1/2}\\
&
+\left(\bE\left[\left\|\int_{t'}^s\int_{\bR^{d}}
\left(\eta_r(X^{t,x}_{r-},z)-\eta_r(X^{t',x'}_{r-},z)\right)
\,\tilde{\pi}(dz,dr)\right\|^2\right]\right)^{1/2}
\end{align*}
Therefore, by It\^o's isometry and \eqref{assumption Lip mu sigma eta} 
we obtain
\begin{align*}
\left(\bE\left[\left\|X^{t,x}_s-X^{t',x'}_s\right\|^2\right]\right)^{1/2}
& 
\leq \left(\bE\left[\left\|X^{t,x}_{t'}-x'\right\|^2\right]\right)^{1/2}\\
& \quad
+L^{1/2}\left((s-t')\int_{t'}^s
\bE\left[\left\|X^{t,x}_r-X^{t',x'}_r\right\|^2\right]dr\right)^{1/2}\\
& \quad
+\left(\int_{t'}^s
\bE\left[\left\|\sigma(r,X^{t,x}_r)-\sigma(r,X^{t',x'}_r)\right\|_F^2\right]\,dr 
\right)^{1/2}\\
& \quad
+\left(\int_{t'}^s\bE\left[\int_{\bR^{d}}
\left\|\eta_r(X^{t,x}_r,z)-\eta_r(X^{t',x'}_r,z)\right\|^2
\nu(dz)\right]dr\right)^{1/2}\\
&
\leq \left(\bE\left[\left\|X^{t,x}_{t'}-x\right\|^2\right]\right)^{1/2}
+L^{1/2}\left((s-t')\int_{t'}^s
\bE\left[\left\|X^{t,x}_r-X^{t',x'}_r\right\|^2\right]dr\right)^{1/2}\\
& \quad
+2L^{1/2}\left(\int_{t'}^s
\bE\left[\left\|X^{t,x}_r-X^{t',x'}_r\right\|^2\right]\,dr\right)^{1/2}.
\end{align*}
Then by the fact that for all $a,b,c\in \bR$ it holds that
$(a+b+c)^2\leq 3(a^2+b^2+c^2)$ and H\"older's inequality we have
\begin{align}
\bE\left[\left\|X^{t,x}_s-X^{t',x'}_s\right\|^2\right]
&
\leq 3\bE\left[\left\|X^{t,x}_{t'}-x'\right\|^2\right]
+3LT\int_{t'}^s\bE\left[\left\|X^{t,x}_r-X^{t',x'}_r\right\|^2\right]
\,dr\nonumber\\
& \quad
+12L\int_{t'}^s\bE\left[\left\|X^{t,x}_r-X^{t',x'}_r\right\|^2\right]
\,dr\nonumber\\
&
= 3\bE\left[\left\|X^{t,x}_{t'}-x'\right\|^2\right]
+3L(T+4)\int_{t'}^s
\bE\left[\left\|X^{t,x}_r-X^{t',x'}_r\right\|^2\right]\,dr  \label{est t' 1}
\end{align}
Moreover, by \eqref{est t' 1}, \eqref{SDE moment est}, and Gr\"onwall's lemma 
we get 
\begin{equation}                                               \label{est t' 2}
\bE\left[\left\|X^{t,x}_s-X^{t',x'}_s\right\|^2\right]
\leq 3\bE\left[\left\|X^{t,x}_{t'}-x'\right\|^2\right]
\cdot\exp\{3LT(T+4)\}.
\end{equation}
Furthermore, we notice that almost surely
\begin{align*}
X^{t,x}_{t'}-x'=x-x'+\int_t^{t'}\mu(r,X^{t,x}_{r-})\,dr
+\int_{t}^{t'}\sigma(r,X^{t,x}_{r-})\,dW_r
+\int_t^{t'}\int_{\bR^{d}}\eta_r(X^{t,x}_{r-},z)\,\,\tilde{\pi}(dz,dr).
\end{align*}
Hence, by Minkowski's inequality and H\"older's inequality we get
$$
\left(\bE\left[\left\|X^{t,x}_{t'}-x'\right\|^2\right]\right)^{1/2}
\leq
\|x-x'\|
+\left((t'-t)\int_t^{t'}\bE\left[\left\|\mu(r,X^{t,x}_r)\right\|^2\right]
dr\right)^{1/2}
$$
\begin{equation}                           \label{t x t'}
+\left(\bE\left[\left\|\int_t^{t'}\sigma(r,X^{t,x}_{r-})\,dW_r
\right\|^2\right]\right)^{1/2}
+\left(\bE\left[\left\|\int_t^{t'}\int_{\bR^{d}}
\eta_r(X^{t,x}_{r-},z)\,\tilde{\pi}(dz,dr)\right\|^2\right]\right)^{1/2}.
\end{equation}
Moreover, by It\^o's isometry, \eqref{t x t'} implies that
$$
\left(\bE\left[\left\|X^{t,x}_{t'}-x'\right\|^2\right]\right)^{1/2}
\leq
\|x-x'\|+\left((t'-t)\int_t^{t'}
\bE\left[\left\|\mu(r,X^{t,x}_r)\right\|^2\right]\,dr\right)^{1/2}
$$
$$
+\left(\int_t^{t'}
\bE\left[\left\|\sigma(r,X^{t,x}_{r})\right\|_F^2\right]\,dr\right)^{1/2}
+\left(\int_t^{t'}\bE\left[\int_{\bR^{d}}
\left\|\eta_r(X^{t,x}_r,z)\right\|^2\nu(dz)\right]dr\right)^{1/2}.
$$
Thus, applying Minkowski's inequality, \eqref{assumption Lip mu sigma eta},
and \eqref{assumption growth} yields
\begin{align*}
&
\left(\bE\left[\left\|X^{t,x}_{t'}-x'\right\|^2\right]\right)^{1/2}\\
&\leq
\|x-x'\|
+\left((t'-t)\int_{t}^{t'}\|\mu(r,x')\|^2\,dr\right)^{1/2}
+\left((t'-t)\int_t^{t'}\bE\left[\left\|\mu(r,X^{t,x}_r)-\mu(r,x')\right\|^2
\right]dr\right)^{1/2}\\
&
\quad+\left(\int_t^{t'}\|\sigma(r,x')\|_F^2\,dr\right)^{1/2}
+\left(\int_t^{t'}
\bE\left[\left
\|\sigma(r,X^{t,x}_r)-\sigma(r,x')\right\|_F^2\right]\,dr\right)^{1/2}\\
&
\quad+\left(\int_t^{t'}\bE\left[\int_{\bR^{d}}
\left\|\eta_r(X^{t,x}_r,z)-\eta_r(x',z)\right\|^2\nu(dz)\right]dr\right)^{1/2}
+\left(\int_t^{t'}\int_{\bR^{d}}\|\eta_r(x',z)\|^2\,\nu(dz)\,dr\right)^{1/2}
\\
&
\leq \|x-x'\|
+\left((t'-t)\int_{t}^{t'}\|\mu(r,x')\|^2\,dr\right)^{1/2}
+L^{1/2}\left((t'-t)\int_t^{t'}
\bE\left[\left\|X^{t,x}_r-x'\right\|^2\right]dr\right)^{1/2}\\
&
\quad
+\left(\int_t^{t'}\|\sigma(r,x')\|_F^2\,dr\right)^{1/2}
+2L^{1/2}\left(\int_t^{t'}
\bE\left[\left\|X^{t,x}_r-x'\right\|^2\right]\,dr\right)^{1/2}
\\
& \quad
+\left(\int_t^{t'}\int_{\bR^{d}}\|\eta_r(x',z)\|^2\,\nu(dz)\,dr\right)^{1/2}.
\end{align*}
Then using H\"older's
inequality, and the fact that for all
$a,b,c\in\bR$ it holds that $(a+b+c)^2\leq 3a^2+3b^2+3c^2$, we have
\begin{align}
\bE\left[\left\|X^{t,x}_{t'}-x'\right\|^2\right] \leq
&
3L(T+4)\int_t^{t'}\bE\left[\left\|X^{t,x}_r-x'\right\|^2\right]\,dr\nonumber\\
&
+3\Bigg[\|x-x'\|+\left(|t-t'|\cdot\int_t^{t'}\|\mu(r,x')\|^2\,dr\right)^{1/2}
+\left(\int_t^{t'}\|\sigma(r,x')\|_F^2\,dr\right)^{1/2}\nonumber\\
&
+\left(\int_t^{t'}\int_{\bR^{d}}\|\eta_r(x',z)\|^2\,\nu(dz)\,dr\right)^{1/2}
\Bigg]^2.                                                       \label{X t'}
\end{align}
This together with \eqref{assumption growth} imply that
\begin{align*}
\bE\left[\left\|X^{t,x}_{t'}-x'\right\|^2\right]\leq
&
3L(T+4)\int_t^{t'}\bE\left[\left\|X^{t,x}_r-x'\right\|^2\right]\,dr\\
&
+3\Big[\|x-x'\|+|t-t'|\cdot L^{1/2}d^{p/2}(1+\|x'\|^2)^{1/2}\\
&
+|t-t'|^{1/2}\cdot L^{1/2}d^{p/2}(1+\|x'\|^2)^{1/2}
+|t-t'|^{1/2}L^{1/2}d^{p/2}(1+\|x'\|^2)^{1/2}\Big]^2.
\end{align*}
Hence, we have
\begin{align*}
\bE\left[\left\|X^{t,x}_{t'}-x'\right\|^2\right]\leq 
&
3L(T+4)\int_t^{t'}\bE\left[\left\|X^{t,x}_r-x'\right\|^2\right]\,dr
+3(\|x-x'\|+|t-t'|+|t-t'|^{1/2})^2\\
&
\cdot\left[1+3L^{1/2}d^{p/2}(1+\|x'\|^2)^{1/2}\right]^2
\end{align*}
Therefore, by \eqref{SDE moment est} and Gr\"onwall's lemma we have
\begin{align}
\bE \left[\left\|X^{t,x}_{t'}-x'\right\|^2\right]
&
\leq 3(\|x-x'\|+|t-t'|+|t-t'|^{1/2})^2
\cdot\exp\{3LT(T+4)\}
\nonumber\\
&
\quad
\cdot
\left[1+3L^{1/2}d^{p/2}(1+\|x'\|^2)^{1/2}\right]^2.            \label{est t' 3}
\end{align}
Hence, combining \eqref{est t' 3} with \eqref{est t' 2} yields \eqref{est t t'},
which completes the proof of this lemma.
\end{proof}
From now on, for ease of notation, 
we define $\Lambda=\{(t,s)\in[0,T]^2:t\leq s\}$.
\begin{corollary}                                         \label{corollary prob 1}
For each $(t,x)\in[0,T]\times\bR^d$ and $s\in[t,T]$,
let $(t_k,s_k,x_k)_{k=1}^\infty$ be a sequence such that 
$(t_k,s_k,x_k)\in\Lambda\times\bR^d$ for all $k\in\bN$ 
and $\lim_{k\to\infty}(t_k,s_k,x_k)=(t,s,x)$.
Then for any $\varepsilon>0$ it holds that
\begin{equation}                                         \label{conv in prob 1}
\lim_{k\to\infty}
\mathbb{P}\left(\left\|X^{t,x}_s-X^{t_k,x_k}_{s_k}\right\|\geq \varepsilon\right)=0.
\end{equation}
\end{corollary}

\begin{proof}
Fix $\varepsilon\in(0,\infty)$, $(t,x)\in[0,T]\times\bR^d$, $s\in[t,T]$ 
and a sequence 
$(t_k,s_k,x_k)_{k=1}^\infty$ satisfying that 
$(t_k,s_k,x_k)\in\Lambda\times\bR^d$ for all $k\in\bN$, 
and $\lim_{k\to\infty}(t_k,s_k,x_k)=(t,s,x)$. By Chebyshev inequality and
the fact that $(a+b)^2\leq2a^2+2b^2$ for all $a,b\in\bR$, we have
\begin{align*}
\bP\left(\left\|X^{t,x}_s-X^{t_k,x_k}_{s_k}\right\|\geq \varepsilon\right)
&
\leq \varepsilon^{-2}\bE\left[\left\|X^{t,x}_s-X^{t_k,x_k}_{s_k}\right\|^2\right]
\\
&
\leq 2\varepsilon^{-2}\bE\left[\left\|X^{t,x}_s-X^{t,x}_{s_k}\right\|^2\right]
+2\varepsilon^{-2}
\bE\left[\left\|X^{t,x}_{s_k}-X^{t_k,x_k}_{s_k}\right\|^2\right].
\end{align*}
Therefore, the application of Lemma \ref{lemma s s'} 
and Lemma \ref{Lemma L2 continuity}
yields
\begin{align}
\bP\left(\left\|X^{t,x}_s-X^{t_k,x_k}_{s_k}\right\|\geq \varepsilon\right) \leq 
&
2\varepsilon^{-2}\cdot c\big[|s_k-s|+|s_k-s|^2\big](d^p+\|x\|^2)\nonumber\\
& 
+2\varepsilon^{-2}\cdot 9\left(\|x_k-x\|+|t_k-t|+|t_k-t|^{1/2}\right)^2
\exp\{6LT(T+4)\}\nonumber\\
&
\cdot \left[1+3L^{1/2}d^{p/2}(1+\|x_k\|^2+\|x\|^2)^{1/2}\right]^2,
\label{conv in prob 0}
\end{align}
where $c$ is the positive constant introduced in \eqref{est s s'}
in Lemma \ref{lemma s s'}. Finally, passing limit as $k\to\infty$ 
in \eqref{conv in prob 0} gives \eqref{conv in prob 1}.
\end{proof}

\subsection{Bounds for Euler approximations}

\begin{lemma} [Euler Approximation]                            \label{Lemma Euler}                                              
There exist positive constants
$C_1=C_1(T,L)$, $C_2=C_2(T,L,L_1,L_2)$, $\rho_1=\rho_1(L)$, 
and $\rho_2=\rho_2(L)$ satisfying 
for all $N\in\bN$ and $(t,x)\in[0,T]\times\bR^d$ that
\begin{equation}                                     \label{sup Euler}
\bE \Big[\sup_{s\in[t,T]}\left\|Y^{t,x,N}_s\right\|^2\Big]
\leq C_1(d^p+\|x\|^2)e^{\rho_1(T-t)}
\end{equation}
and
\begin{equation}                                     \label{error Euler}
\bE\Big[\sup_{s\in[t,T]}\left\|X^{t,x}_s-Y^{t,x,N}_s\right\|^2\Big]
\leq C_2(d^p+\|x\|^2)N^{-1}e^{\rho_2(T-t)}.
\end{equation} 
\end{lemma}

\begin{proof}
We fix $(t,x)\in[0,T]\times\bR^d$ and $N\in\bN$ throughout this proof.
For $Y^{t,x,N}=(Y^{t,x,N,1},...,Y^{t,x,N,d})$,
applying It\^o's formula to $\big\|Y^{t,x,N}_s\big\|^2$ shows almost surely
for all $s\in[t,T]$ that
\begin{align*}
\left\|Y^{t,x,N}_s\right\|^2=
&
\|x\|^2+2\int_t^s\left(
\big\<Y^{t,x,N}_r,\mu(\kappa_N(r),Y^{t,x,N}_{\kappa_N(r)})\big\>
+\frac{1}{2}
\left\|\sigma(\kappa_N(r),Y^{t,x,N}_{\kappa_N(r)})\right\|_F^2\right)\,dr\\
&
+2\sum_{i=1}^d\sum_{j=1}^d
\int_t^sY^{t,x,N,i}_{r-}\sigma^{ij}
(\kappa_N(r-),Y^{t,x,N}_{\kappa_N(r-)})\,dW_r^j\\
&
+2\int_t^s\int_{\bR^{d}}
\big\<Y^{t,x,N}_{r-},\eta_{\kappa_N(r-)}(Y^{t,x,N}_{\kappa_N(r-)},z)\big\>
\,\tilde{\pi}(dz,dr)\\
&
+\int_t^s\int_{\bR^{d}}
\left\|\eta_{\kappa_N(r-)}(Y^{t,x,N}_{\kappa_N(r-)},z)\right\|^2\,\pi(dz,dr).
\end{align*}
Then, by Davis' inequality
(see, e.g., Theorem VII.90 and VII (90.1) in \cite{DM1982})
we obtain for every $s\in[t,T]$ that
\begin{equation}                                                 \label{A_s}
a_N(s):=\bE\Big[\sup_{r\in[t,s]}\left\|Y^{t,x,N}_r\right\|^2\Big]\leq 
\|x\|^2+\sum_{i=1}^4A_i(s),
\end{equation}
where
\begin{align*}
&
A_1(s):=2\bE\left[\int_t^s\left
(\left|\big\<Y^{t,x,N}_r,\mu(\kappa_N(r),Y^{t,x,N}_{\kappa_N(r)})\big\>\right|
+\frac{1}{2}\left\|\sigma(\kappa_N(r),Y^{t,x,N}_{\kappa_N(r)})\right\|_F^2\right)
dr\right],\\
&
A_2(s):=8\bE\left[\left(
\int_t^s\left\|\left[\sigma(\kappa_N(r),Y^{t,x,N}_{\kappa_N(r)})\right]^T
Y^{t,x,N}_{\kappa_N(r)}
\right\|^2\,dr
\right)^{1/2}\right],\\
&
A_3(s):=8\bE\left[\left(
\int_t^s\int_{\bR^{d}}\left|\<Y^{t,x,N}_{r-},
\eta_{\kappa_N(r-)}(Y^{t,x,N}_{\kappa_N(r-)},z)\>
\right|^2
\,\pi(dz,dr)
\right)^{1/2}\right],\\
&
A_4(s):=\bE\left[
\int_t^s\int_{\bR^{d}}\left\|
\eta_{\kappa_N(r)}(Y^{t,x,N}_{\kappa_N(r)},z)\right\|^2
\,\nu(dz)\,dr\right].
\end{align*}
Furthermore, by Cauchy-Schwarz inequality, Young's inequality,
and \eqref{assumption growth} we have for all $s\in[t,T]$ that
\begin{align}
A_1(s)
&
\leq \bE\left[\int_t^s\left(\left\|Y^{t,x,N}_r\right\|^2
+\left\|\mu(\kappa_N(r),Y^{t,x,N}_{\kappa_N(r)})\right|^2
+\left\|\sigma(\kappa_N(r),Y^{t,x,N}_{\kappa_N(r)})\right\|_F^2\right)\,dr\right]\nonumber\\
&
\leq \bE\Bigg[\int_t^s\Big(\left\|Y^{t,x,N}_r\right\|^2
+2\left\|\mu(\kappa_N(r), Y^{t,x,N}_{\kappa_N(r)})-\mu(\kappa_N(r),0)\right\|^2
+2\left\|\mu(\kappa_N(r),0)\right\|^2\nonumber\\
&
\quad +2\left\|\sigma(\kappa_N(r), Y^{t,x,N}_{\kappa_N(r)})
-\sigma(\kappa_N(r),0)\right\|_F^2
+2\left\|\sigma(\kappa_N(r),0)\right\|_F^2
\Big)\,dr\Bigg]\nonumber\\
&
\leq \bE\left[\int_t^s\left(\left\|Y^{t,x,N}_r\right\|^2
+4L\left\|Y^{t,x,N}_{\kappa_N(r)}\right\|^2+4Ld^p\right)\,dr\right]\nonumber\\
&
\leq 4TLd^p+(1+4L)\int_t^sa_N(r)\,dr.                         \label{A1}
\end{align}
Similarly, by Cauchy-Schwarz inequality, Young's inequality, 
\eqref{assumption Lip mu sigma eta}, and \eqref{assumption growth} we get
for all $s\in[t,T]$ and $\varepsilon>0$ that
\begin{align}
A_2(s)
&
\leq 8\bE\left[\left(\int_t^s\left\|Y^{t,x,N}_{\kappa_N(r)}\right\|^2
\left\|\sigma(\kappa_N(r),Y^{t,x,N}_{\kappa_N(r)})\right\|_F^2\,dr\right)^{1/2}
\right]\nonumber\\
&
\leq 8\bE\left[\sup_{r\in[t,s]}(\left\|Y^{t,x,N}_r\right\|)\Big(
\int_t^s\Big(2\left\|\sigma(\kappa_N(r),Y^{t,x,N}_{\kappa_N(r)})
-\sigma(\kappa_N(r),0)\right\|_F^2
+2\left\|\sigma(\kappa_N(r),0)\right\|_F^2
\Big)\,dr
\Big)^{1/2}\right]\nonumber\\
&
\leq 4\varepsilon a_N(s)+\frac{4}{\varepsilon}\bE\left[
\int_t^s\left(2L\left\|Y^{t,x,N}_{\kappa_N(r)}\right\|^2+2Ld^p\right)\,dr\right]
\nonumber\\
&
\leq 4\varepsilon a_N(s)+\frac{8TLd^p}{\varepsilon}
+\frac{8L}{\varepsilon}\int_t^sa_N(r)\,dr,                             \label{A2}
\end{align}
and
\begin{align}                                                    
A_3(s) &
\leq 8\bE\left[\sup_{r\in[t,s]}(\left\|Y^{t,x,N}_r\right\|)\Big(
\int_t^s\int_{\bR^{d}}
\big\|\eta_{\kappa_N(r-)}(Y^{t,x,N}_{\kappa_N(r-)},z)\big\|^2
\,\pi(dz,dr)
\Big)^{1/2}\right]\nonumber\\
&
\leq 4\varepsilon a_N(s)+\frac{4}{\varepsilon}A_4(s).
\label{A3}
\end{align}
Moreover, by \eqref{assumption Lip mu sigma eta} and \eqref{assumption growth} 
we have for all $s\in[t,T]$ that
\begin{align}                                                  
A_4(s) & \leq \bE\left[\int_t^s\int_{\bR^{d}}\left(
2\left\|\eta_{\kappa_N(r)}(Y^{t,x,N}_{\kappa_N(r)},z)
-\eta_{\kappa_N(r)}(0,z)\right\|^2
+2\left\|\eta_{\kappa_N(r)}(0,z)\right\|^2
\right)\,\nu(dz)\,dr \right]\nonumber\\
&
\leq
\bE\left[\int_t^s\left(
2L\left\|Y^{t,x,N}_{\kappa_N(r)}\right\|^2+2Ld^p
\right)\,dr \right]\nonumber\\
&
\leq 2TLd^p+2L\int_t^sa_N(r)\,dr.
\label{A4}
\end{align}
Then, combining \eqref{A_s} and \eqref{A1}-\eqref{A4} yields for all
$s\in[t,T]$ that
\begin{align}                                                    
a_N(s)\leq & \|x\|^2+8\varepsilon a_N(s)
+2TLd^p\left(3+8/\varepsilon\right)
+\left(1+6L+16L/\varepsilon\right)\int_t^sa_N(r)\,dr.
\label{a_s}
\end{align}
Define a sequence of stopping times
$\{\tau^N_k\}_{k=1}^\infty$ by
$$
\tau^N_k:=\inf\{s\geq t:\left\|Y^{t,x,N}_s\right\|\geq k\}\wedge T,
\quad k\in\bN
$$
Then by Assumption \ref{assumption pointwise}, we have that
\begin{align*}
\bE\left[\sup_{s\in[t,T]}\left\|Y^{t,x,N}_{s\wedge\tau^N_k}\right\|^2\right]
&
\leq 
k^2+\bE\left[\left\|Y^{t,x,N}_{\tau^N_k}
-Y^{t,x,N}_{\tau^N_k-}\right\|^2\right]
\\
&
= k^2+\bE\left[
\Big\|
\int_{\bR^d}\eta_{\kappa_N(\tau^N_k-)}\big(Y^{t,x,N}_{\kappa_N(\tau^N_k-)},z\big)
\,\pi\big(dz,\{\tau^N_k\}\big)
\Big\|^2
\right]
\\
&
\leq k^2+\bE\left[
\int_{\bR^d}C_d(1\wedge\|z\|^2)\,\pi\big(dz,\{\tau^N_k\}\big)
\right]
\\
&
\leq k^2+C_d<\infty.
\end{align*}
Hence, by taking $\varepsilon=1/9$ and replacing $s$ by $s\wedge \tau^N_k$
in \eqref{a_s}, as well as using Gr\"onwall's lemma we obtain
$$
a_N(s\wedge\tau^N_k)\leq C_1(d^p+\|x\|^2)e^{\rho_1(s-t)}\quad
\text{for all $s\in[t,T]$ and $k\in\bN$,}
$$
where
$C_1:=9\cdot\max\{150LT,1\}$ and $\rho_1:=9(1+150L)$. 
Then the application of Fatou's lemma gives
$$
a_N(T)\leq \liminf_{k\to\infty}a_N(T\wedge \tau^N_k)\leq C_1(d^p+\|x\|^2)
e^{\rho_1(T-t)},
$$
which proves \eqref{sup Euler}.

Next, we show the proof of \eqref{error Euler},
which is analogous to the proof of \eqref{sup Euler}.
For every $s\in[t,T]$, and $i,j\in\{1,...,d\}$, we define
$$
\tilde{\mu}^N_s:=\mu(s,X^{t,x}_{s-})-\mu(\kappa_N(s-),Y^{t,x,N}_{\kappa_N(s-)}),
\quad 
\tilde{\eta}^N_s(z):=\eta_s(X^{t,x}_{s-},z)
-\eta_{\kappa_N(s-)}(Y^{t,x,N}_{\kappa_N(s-)},z),
$$
and
$$
\tilde{\sigma}^N_s:=\sigma(s,X^{t,x}_{s-})
-\sigma(\kappa_N(s-),Y^{t,x,N}_{\kappa_N(s-)}),
\quad 
\tilde{\sigma}^{N,ij}_s:=\sigma^{ij}(s,X^{t,x}_{s-})
-\sigma^{ij}(\kappa_N(s-),Y^{t,x,N}_{\kappa_N(s-)}).
$$
For
$X^{t,x}-Y^{t,x,N}=(X^{t,x,1}-Y^{t,x,N,1},...,X^{t,x,d}-Y^{t,x,N,d})$,
the application of It\^o's formula shows almost surely for all $s\in[t,T]$ that
\begin{align*}
\left\|X^{t,x}_s-Y^{t,x,N}_s\right\|^2=
&
2\int_t^s\left(\big\<(X^{t,x}_{r-}-Y^{t,x,N}_{r-}),\tilde{\mu}^N_r\big\>
+\frac{1}{2}\left\|\tilde{\sigma}^N_r\right\|_F^2\right)\,dr\\
&
+2\sum_{i=1}^d\sum_{j=1}^d
\int_t^s(X^{t,x,i}_{r-}-Y^{t,x,N,i}_{r-})\tilde{\sigma}^{N,ij}_r\,dW_r^j\\
&
+2\int_t^s\int_{\bR^{d}}\big\<(X^{t,x}_{r-}-Y^{t,x,N}_{r-}),
\tilde{\eta}^N_r(z)\big\>\,\tilde{\pi}(dz,dr)\\
&
+\int_t^s\int_{\bR^{d}}\left\|\tilde{\eta}^N_r(z)\right\|^2\,\pi(dz,dr).
\end{align*}
Then by Davis' inequality 
we obtain for all $s\in[t,T]$ that
\begin{equation}                                             \label{E_s}
e_N(s):=\bE\Big[\sup_{r\in[t,s]}\left\|X^{t,x}_r-Y^{t,x,N}_r\right\|^2\Big]
\leq \sum_{i=1}^4E_i(s),
\end{equation}
where
\begin{align*}
&
E_1(s):=
2\bE\left[\int_t^s\left(\left|
\<(X^{t,x}_{r-}-Y^{t,x,N}_{r-}),\tilde{\mu}^N_r\>\right|
+\frac{1}{2}\left\|\tilde{\sigma}^N_r\right\|_F^2
\right)\,dr\right],\\
&
E_2(s):=
8\bE\left[\left(\int_t^s\left\|[\tilde{\sigma}^N_r]^T
(X^{t,x}_{r-}-Y^{t,x,N}_{r-})\right\|^2
\,dr\right)^{1/2}\right],\\
&
E_3(s):=
8\bE\left[\left(\int_t^s\int_{\bR^{d}}
\left|\big\<(X^{t,x}_{r-}-Y^{t,x,N}_{r-}),\tilde{\eta}^N_r(z)\big\>\right|^2
\,\pi(dz,dr)\right)^{1/2}\right],\\
&
E_4(s):=\bE\left[\int_t^s\int_{\bR^{d}}
\left\|\tilde{\eta}^N_r(z)\right\|^2\,\nu(dz)\,dr\right].
\end{align*}
Moreover, by Lemma \ref{lemma s s'} we notice for all $s\in[t,T]$ that
\begin{equation}                                              \label{eq kappa}             
\bE\left[\left\|X^{t,x}_{s}-X^{t,x}_{\kappa_N(s)}\right\|^2\right]
\leq 2c(d^p+\|x\|^2)N^{-1}T(T+1),
\end{equation}
where $c=c(T,L)$ is the positive constant introduced in \eqref{est s s'}
in Lemma \ref{lemma s s'}.
Then, by \eqref{assumption Lip mu sigma eta}, \eqref{assumption growth},  
\eqref{Holder mu sigma d}, and \eqref{Holder eta d} we obtain
for all $s\in[t,T]$ that
\begin{align}                                           
\left\|\tilde{\mu}^N_s\right\| &
\leq  \left\|\mu(s,X^{t,x}_{s-})-\mu(s,X^{t,x}_{\kappa_N(s-)})\right\|
+\left\|\mu(s,X^{t,x}_{\kappa_N(s-)})
-\mu(\kappa_N(s-),X^{t,x}_{\kappa_N(s-)})\right\|
\nonumber\\
&
\quad +\left\|\mu(\kappa_N(s-),
X^{t,x}_{\kappa_N(s-)})-\mu(\kappa_N(s-),Y^{t,x,N}_{\kappa_N(s-)})\right\|
\nonumber\\
&
\leq 
L^{1/2}\left\|X^{t,x}_{s-}-X^{t,x}_{\kappa_N(s-)}\right\|+(L_1N^{-1}T)^{1/2}
+L^{1/2}\left\|X^{t,x}_{\kappa_N(s-)}-Y^{t,x,N}_{\kappa_N(s-)}\right\|,   
\label{mu~}
\end{align}
\begin{equation}                                             \label{sigma~}
\left\|\tilde{\sigma}^N_s\right\|_F\leq 
L^{1/2}\left\|X^{t,x}_{s-}-X^{t,x}_{\kappa_N(s-)}\right\|+(L_1N^{-1}T)^{1/2}
+L^{1/2}\left\|X^{t,x}_{\kappa_N(s-)}-Y^{t,x,N}_{\kappa_N(s-)}\right\|,
\end{equation}
and
\begin{equation}                                           \label{eta~}
\int_{\bR^{d}}\left\|\tilde{\eta}^N_s(z)\right\|^2\,\nu(dz)
\leq 3L\left\|X^{t,x}_{s-}-X^{t,x}_{\kappa_N(s-)}\right\|^2+3L_2N^{-1}T
+3L\left\|X^{t,x}_{\kappa_N(s-)}-Y^{t,x,N}_{\kappa_N(s-)}\right\|^2.
\end{equation}      
Thus, by \eqref{eq kappa}, \eqref{mu~}, \eqref{sigma~}, Cauchy-Schwarz inequality, 
and Young's inequality we obtain for every $s\in[t,T]$ that
\begin{align}
E_1(s)
&
\leq 2\bE\Bigg[\int_t^s\big(\left\|X^{t,x}_r-Y^{t,x,N}_r\right\|
\left((L_1N^{-1}T)^{1/2}
+L^{1/2}\left\|X^{t,x}_r-X^{t,x}_{\kappa_N(r)}\right\|
+L^{1/2}\left\|X^{t,x}_{\kappa_N(r)}-Y^{t,x,N}_{\kappa_N(r)}\right\|\right)\nonumber\\
&
\quad+\frac{1}{2}\left((L_1N^{-1}T)^{1/2}
+L^{1/2}\left\|X^{t,x}_r-X^{t,x}_{\kappa_N(r)}\right\|
+L^{1/2}\left\|X^{t,x}_{\kappa_N(r)}-Y^{t,x,N}_{\kappa_N(r)}\right\|\right)^2\,dr
\Bigg]\nonumber\\
&
\leq \bE\left[\int_t^s\left(
\left\|X^{t,x}_r-Y^{t,x,N}_r\right\|^2+6L_1N^{-1}T
+6L\left\|X^{t,x}_r-X^{t,x}_{\kappa_N(r)}\right\|^2
+6L\left\|X^{t,x}_{\kappa_N(r)}-Y^{t,x,N}_{\kappa_N(r)}\right\|^2
\right)dr\right]
\nonumber\\
&
\leq 6T^2(L_1+2c(T+1)L)(d^p+\|x\|^2)N^{-1}+(1+6L)\int_t^se_N(r)\,dr,
\label{E1}
\end{align}
where we recall that 
$e_N(s):=\bE\Big[\sup_{r\in[t,s]}\left\|X^{t,x}_r-Y^{t,x,N}_r\right\|^2\Big]$ 
for each $s\in[t,T]$.
Analogously, by \eqref{eq kappa}, \eqref{sigma~}, Cauchy-Schwarz inequality, 
and Young's inequality
we get for all $s\in[t,T]$ and $\varepsilon>0$ that
\begin{align}
E_2(s)
&
\leq 8\bE\left[\left( \int_t^s\left\|X^{t,x}_r-Y^{t,x,N}_r\right\|^2
\left\|\tilde{\sigma}^N_r\right\|_F^2 
\,dr\right)^{1/2}\right]\nonumber\\
&
\leq 8\bE\left[\left(\sup_{r\in[t,s]}\left\|X^{t,x}_r-Y^{t,x,N}_r\right\|\right)
\left(\int_t^s\|\tilde{\sigma}^N_r\|_F^2\,dr\right)^{1/2}\right]\nonumber\\
&
\leq
\frac{12}{\varepsilon}\bE\left[\int_t^s\left(L_1N^{-1}T
+L\left\|X^{t,x}_{r}-X^{t,x}_{\kappa_N(r)}\right\|^2
+L\left\|X^{t,x}_{\kappa_N(r)}-Y^{t,x,N}_{\kappa_N(r)}\right\|^2\right)
\,dr\right]\nonumber\\
&
\quad+4\varepsilon \bE\Bigg[\sup_{r\in[t,s]}\left\|X^{t,x}_r-Y^{t,x,N}_r\right\|^2\Bigg]
\nonumber\\
&
\leq 4\varepsilon e_N(s)+\frac{12T^2(L_1+2c(T+1)L)(d^p+\|x\|^2)}{\varepsilon N}
+\frac{12L}{\varepsilon}\int_t^se_N(r)\,dr,          \label{E2}
\end{align}
and
\begin{align}                                       
E_3(s) &
\leq 8\bE\left[\left(\sup_{r\in[t,s]}\left\|X^{t,x}_r-Y^{t,x,N}_r\right\|\right)
\left(\int_t^s\int_{\bR^{d}}\|\tilde{\eta}^N_r(z)\|^2
\,\pi(dz,dr)\right)^{1/2}\right]\nonumber\\
&
\leq 4\varepsilon e_N(s)+\frac{4E_4(s)}{\varepsilon}.          
\label{E3}
\end{align}
Moreover, by \eqref{eq kappa} and \eqref{eta~} we see for all $s\in[t,T]$ that
\begin{align}
E_4(s)
&
\leq \bE\Big[\int_t^s\Big(
3L_2N^{-1}T+3L\left\|X^{t,x}_r-X^{t,x}_{\kappa_N(r)}\right\|^2 
+3L\left\|X^{t,x}_{\kappa_N(r)}-Y^{t,x,N}_{\kappa_N(r)}\right\|^2 
\Big)\,dr\Big]
\nonumber\\
&
\leq \frac{3T^2(L_2+2c(T+1)L)(d^p+\|x\|^2)}{N}+3L\int_t^se_N(r)\,dr.
\label{E4}
\end{align}
Then combining \eqref{E_s} and \eqref{E1}-\eqref{E4} yields
for each $s\in[t,T]$ that
\begin{align*}
e_N(s)
&
\leq
8\varepsilon e_N(s)+\left(
1+9L+\frac{24L}{\varepsilon}\right)
\int_t^se_N(r)\,dr
\\
&
\quad+
\left[
6T^2(L_1+2c(T+1)L)\left(1+\frac{2}{\varepsilon}\right)
+3T^2(L_2+2c(T+1)L)\left(1+\frac{4}{\varepsilon}\right)
\right]\frac{(d^p+\|x\|^2)}{N}.
\end{align*} 
Hence, by \eqref{SDE moment est} and \eqref{sup Euler},  
taking $\varepsilon=1/9$ and applying Gr\"onwall's lemma gives
\begin{equation}                                   \label{eq C3}
e_N(T)=\bE\Bigg[\sup_{s\in[t,T]}\left\|X^{t,x}_s-Y^{t,x,N}_s\right\|^2\Bigg]
\leq C_2
(d^p+\|x\|^2)N^{-1}e^{\rho_2(T-t)},
\end{equation}
where $C_2:=27T^2(38L_1+37L_2+150c(T+1)L)$ and $\rho_2:=9(1+225L)$.
Therefore, \eqref{error Euler} is proved,
and we complete the proof of this lemma.
\end{proof}

\subsection{Bounds for Euler approximations with truncated jumps}

\begin{lemma}
There exist positive constants $C_1=C_1(T,L)$, $C_2=C_2(T,L,K)$, 
$\rho_1=\rho_1(T,L)$, and $\rho_2=\rho_2(T,L)$ satisfying
for each $(t,x)\in[0,T]\times\bR^d$, $N\in\bN$, and $\delta\in(0,1)$ that
\begin{equation}                                         \label{est Y delta}
\bE\Bigg[\sup_{s\in[t,T]}\left\|Y^{t,x,N,\delta}_s\right\|^2\Bigg]
\leq C_1(d^p+\|x\|^2)e^{\rho_1(T-t)},
\end{equation}
and
\begin{equation}                                      \label{est error Y delta}
\bE\Bigg[\sup_{s\in[t,T]}\left\|Y^{t,x,N,\delta}_s-Y^{t,x,N}_s\right\|^2\Bigg]
\leq C_2\delta^qd^p(d^p+\|x\|^2)e^{\rho_2(T-t)}.
\end{equation}
\end{lemma}

\begin{proof}
As before, we first fix $(t,x)\in[0,T]\times\bR^d$, $N\in\bN$, 
and $\delta\in(0,1)$.
One can follow the proof of Lemma \ref{Lemma Euler} for \eqref{sup Euler} 
to obtain that
$$                                       
\bE\Bigg[\sup_{s\in[t,T]}\left\|Y^{t,x,N,\delta}_s\right\|^2\Bigg]
\leq C_1(d^p+\|x\|^2)e^{\rho_1(T-t)}
$$
with positive constants $C_1=C_1(T,L)$ and $\rho_1=\rho_1(T,L)$ defined 
in Lemma \ref{Lemma Euler}, which proves \eqref{est Y delta}.
Next, by Davis' inequality, H\"older's inequality, and 
\eqref{assumption Lip mu sigma eta} 
we observe that
\begin{align}
\bE & \Bigg[\sup_{u\in[t,s]}\left\|Y^{t,x,N,\delta}_u-Y^{t,x,N}_u\right\|^2\Bigg]
\leq  3(s-t)\int_t^s\bE
\left[\left\|\mu(\kappa_N(r),Y^{t,x,N,\delta}_{\kappa_N(r)})
-\mu(\kappa_N(r),Y^{t,x,N}_{\kappa_N(r)})\right\|^2\right]dr\nonumber\\
&
\quad+12\int_t^s\bE \left[\left
\|\sigma(\kappa_N(r),Y^{t,x,N,\delta}_{\kappa_N(r)})
-\sigma(\kappa_N(r),Y^{t,x,N}_{\kappa_N(r)})\right\|_F^2\right]dr\nonumber\\
&
\quad+12\int_t^s\bE\left[\int_{\bR^{d}}\left\|
\eta_{\kappa_N(r)}(Y^{t,x,N,\delta}_{\kappa_N(r)},z)\mathbf{1}_{A_\delta}(z)
-\eta_{\kappa_N(r)}(Y^{t,x,N}_{\kappa_N(r)},z)
\right\|^2\nu(dz)\right]dr\nonumber\\
&
\leq  3(s-t)\int_t^s\bE
\left[\left\|\mu(\kappa_N(r),Y^{t,x,N,\delta}_{\kappa_N(r)})
-\mu(\kappa_N(r),Y^{t,x,N}_{\kappa_N(r)})\right\|^2\right]dr\nonumber\\
&
\quad+12\int_t^s\bE \left[\left
\|\sigma(\kappa_N(r),Y^{t,x,N,\delta}_{\kappa_N(r)})
-\sigma(\kappa_N(r),Y^{t,x,N}_{\kappa_N(r)})\right\|_F^2\right]dr\nonumber\\
&
\quad+24\int_t^s\bE\left[\int_{\bR^{d}}\left\|
\eta_{\kappa_N(r)}(Y^{t,x,N,\delta}_{\kappa_N(r)},z)\mathbf{1}_{A_\delta}(z)
-\eta_{\kappa_N(r)}(Y^{t,x,N}_{\kappa_N(r)},z)\mathbf{1}_{A_\delta}(z)
\right\|^2\nu(dz)\right]dr\nonumber\\
&
\quad+24\int_t^s\bE\left[\int_{\bR^{d}}\left\|
\eta_{\kappa_N(r)}(Y^{t,x,N}_{\kappa_N(r)},z)\mathbf{1}_{\bR^d/ A_\delta}(z)
\right\|^2\nu(dz)\right]dr\nonumber\\
&
\leq 3(T+12)L\int_t^s\bE
\Bigg[\sup_{u\in[t,r]}\left\|Y^{t,x,N,\delta}_u-Y^{t,x,N}_u\right\|^2\Bigg]dr
\nonumber\\
& \quad
+24\int_t^T\bE\Bigg[\int_{\bR^{d}/A_\delta}\Big\|\eta_{\kappa_N(r)}(Y^{t,x,N}_{\kappa_N(r)},z)\Big\|^2\,\nu(dz)\Bigg]dr.               \label{delta 1}
\end{align}
Furthermore, by \eqref{int z & q bound} and \eqref{sup Euler} it holds that
\begin{align}                                           
\int_t^T\bE\Bigg[\int_{\bR^{d}/A_\delta}\Big\|\eta_{\kappa_N(r)}
(Y^{t,x,N}_{\kappa_N(r)},z)\Big\|^2\,\nu(dz)\Bigg]dr 
&
\leq K\delta^qd^p\int_t^T\bE\Bigg[1+\Big\|Y^{t,x,N}_{\kappa_N(r)}\Big\|^2\Bigg]\,dr\nonumber\\
&
\leq (C_1+1)KT\delta^qd^p(d^p+\|x\|^2)e^{\rho_1 T},              \label{delta 2}
\end{align}
where $C_1=C_1(T,L)$ and $\rho_1=\rho_1(L)$ are positive constants defined in
\eqref{sup Euler} in Lemma \ref{Lemma Euler}.
Combining \eqref{delta 1} and \eqref{delta 2}, 
as well as taking into account \eqref{sup Euler} and \eqref{est Y delta},
by Gr\"onwall's lemma we obtain
$$
\bE\Bigg[\sup_{s\in[t,T]}\left\|Y^{t,x,N,\delta}_s-Y^{t,x,N}_s\right\|^2\Bigg]
\leq
C_2\delta^qd^p(d^p+\|x\|^2)e^{\rho_2(T-t)},
$$
where $C_2:=24(C_1+1)KTe^{\rho_1T}$
and $\rho_2:=3(T+12)L$.
Thus, the proof of this lemma is completed.
\end{proof}

\subsection{Bounds for Euler approximations with truncated jumps and
Monte Carlo approximations of compensator integrals}

\begin{lemma}                                       \label{lemma delta M}
There exist positive constants $C_1$, $C_2$, 
$\rho_1$, and $\rho_2$ only depending on $T$ and $L$ satisfying
for all $(t,x)\in[0,T]\times\bR^d$, $\delta\in(0,1)$, and $N,\cM\in\bN$
with $\cM\geq \delta^{-2}Kd^p$ that
\begin{equation}                                       \label{estimate M}
\bE\Bigg[\sup_{s\in[t,T]}\big\|Y^{t,x,N,\delta,\cM}_s\big\|^2\Bigg]\leq
C_1(d^p+\|x\|^2)e^{\rho_1(T-t)}
\end{equation}
and
\begin{equation}                                      \label{error M}
\bE\Bigg[
\sup_{s\in[t,T]}\big\|Y^{t,x,N,\delta,\cM}_s-Y^{t,x,N,\delta}_s\big\|^2\Bigg]
\leq C_2\delta^{-2}Kd^p\cM^{-1}(d^p+\|x\|^2)e^{\rho_2(T-t)}.
\end{equation} 
\end{lemma}

\begin{proof}
We fix $(t,x)\in[0,T]$, $\delta\in(0,1)$, and $N,\cM\in\bN$ 
with $\cM\geq \delta^{-2}Kd^p$ throughout this proof.
To prove \eqref{estimate M}, we first notice that by \eqref{int z & q bound} 
it holds that
\begin{equation}                                              \label{nu finite}
\nu(A_\delta)=\int_{A_\delta}\frac{1\wedge\|z\|^2}{1\wedge\|z\|^2}\,\nu(dz)
\leq \delta^{-2}\int_{\bR^{d}}(1\wedge\|z\|^2)\,\nu(dz)\leq \delta^{-2}Kd^p,
\end{equation}
which shows that $\nu$ is a finite measure on
$(A_\delta,\cB(A_\delta))$. Thus, by \eqref{def Y N delta} we have 
almost surely for all $s\in[t,T]$ that
\begin{align}
Y^{t,x,N,\delta,\cM}_s=
&
x+\int_t^s\mu(\kappa_N(r-),Y^{t,x,N,\delta,\cM}_{\kappa_N(r-)})\,dr
+\int_t^s\sigma(\kappa_N(r-),Y^{t,x,N,\delta,\cM}_{\kappa_N(r-)})\,dW_r\nonumber\\
&
+\int_t^s\int_{A_\delta}
\eta_{\kappa_N(r-)}(Y^{t,x,N,\delta,\cM}_{\kappa_N(r-)},z)\,\tilde{\pi}(dz,dr)
+\int_t^s\int_{A_\delta}
\eta_{\kappa_N(r-)}(Y^{t,x,N,\delta,\cM}_{\kappa_N(r-)},z)\,\nu(dz)\,dr
\nonumber\\
&
-\int_t^s\Big(\frac{\nu(A_\delta)}{\cM}\sum_{j=1}^{\cM}
\eta_{\kappa_N(r-)}(Y^{t,x,N,\delta,\cM}_{\kappa_N(r-)},
V^{N,\delta,\cM}_{\frac{(\kappa_N(r-)-t)N}{T-t},j})
\Big)dr.                                                \label{Y N delta M 1}            
\end{align}
By H\"older's inequality, Burkholder-Davis-Gundy inequality
(see, e.g., Theorem VII.92 in \cite{DM1982}), 
Minkowski integral inequality, and the fact that
$(\sum_{i=1}^5a_i)^2\leq 5\sum_{i=1}^5a_i^2$ for $a_i\in\bR$, $i=1,2,...,5$,
we obtain for each $s\in[t,T]$ that
\begin{equation}                                       \label{a N delta M}
a^\delta_{N,\cM}(s):=\bE\Bigg[\sup_{r\in[t,s]}\big\|Y^{t,x,N,\delta,\cM}_r\big\|^2\Bigg]
\leq 5\left(\|x\|^2+\sum_{i=1}^4A_i(s)\right),
\end{equation} 
where
\begin{align}
&                                               
A_1(s):=(s-t)\bE\left[\int_t^s\big\|
\mu(\kappa_N(r),Y^{t,x,N,\delta,\cM}_{\kappa_N(r)})
\big\|^2\,dr\right],                                    
\label{def A1}\\
&
A_2(s):=8\bE\left[\int_t^s\big\|
\sigma(\kappa_N(r),Y^{t,x,N,\delta,\cM}_{\kappa_N(r)})
\big\|_F^2\,dr\right],
\label{def A2}\\
&
A_3(s):=8\bE\left[\int_t^s\int_{A_\delta}\big\|
\eta_{\kappa_N(r)}(Y^{t,x,N,\delta,\cM}_{\kappa_N(r)},z)
\big\|^2\,\nu(dz)\,dr\right],
\label{def A3}\\
&
A_4(s):=(s-t)\bE\Bigg[
\int_t^s\Big\|\int_{A_\delta}\eta_{\kappa_N(r)}(Y^{t,x,N,\delta,\cM}_{\kappa_N(r)},z)
\nu(dz)\nonumber\\
&
\qquad\qquad
-\Big(\frac{\nu(A_\delta)}{\cM}\sum_{j=1}^{\cM}
\eta_{\kappa_N(r)}(Y^{t,x,N,\delta,\cM}_{\kappa_N(r)},
V^{N,\delta,\cM}_{\frac{(\kappa_N(r)-t)N}{T-t},j})
\Big)\Big\|^2\,dr
\Bigg].
\label{def A4}
\end{align}
Next,by the definition of $V^{N,\delta,\cM}_{i,j}$, we observe for all
$\delta\in(0,1)$,
$y\in\bR^d$, $s\in[0,T]$, $i=0,1,...,N$, and
$j=1,2,...,\cM$ that
\begin{equation}                                             \label{V int}
\nu(A_\delta)\bE\Big[\eta_s(y,V^{N,\delta,\cM}_{i,j})\Big]
=\int_{A_\delta}\eta_s(y,z)\,\nu(dz).
\end{equation}
Then by the independence properties of $V^{N,\delta,\cM}_{i,j}$, 
\eqref{assumption Lip mu sigma eta}, 
\eqref{assumption growth}, and \eqref{V int} it holds for all $r\in[t,T]$ that 
\begin{align}
&
\bE\Big[
\Big\|\int_{A_\delta}\eta_{\kappa_N(r)}(Y^{t,x,N,\delta,\cM}_{\kappa_N(r)},z)
\nu(dz)
-\Big(\frac{\nu(A_\delta)}{\cM}\sum_{j=1}^{\cM}
\eta_{\kappa_N(r)}(Y^{t,x,N,\delta,\cM}_{\kappa_N(r)},
V^{N,\delta,\cM}_{\frac{(\kappa_N(r)-t)N}{T-t},j})
\Big)\Big\|^2
\Big]\nonumber\\
&
=\bE\left[\bE\Big[
\Big\|\int_{A_\delta}\eta_{\kappa_N(r)}(y,z)
\nu(dz)
-\Big(\frac{\nu(A_\delta)}{\cM}\sum_{j=1}^{\cM}
\eta_{\kappa_N(r)}(y,V^{N,\delta,\cM}_{\frac{(\kappa_N(r)-t)N}{T-t},j})
\Big)\Big\|^2\Big]\Bigg|_{y=Y^{t,x,N,\delta,\cM}_{\kappa_N(r)}}
\right]\nonumber\\
&
=[\nu(A_\delta)]^2\cM^{-1}\bE\left[
\sum_{k=1}^d\bE\Bigg[
\Big|\bE\big[\eta_{\kappa_N(r)}^k(y,V^{N,\delta,\cM}_{1,1})\big]
-\eta_{\kappa_N(r)}^k(y,V^{N,\delta,\cM}_{1,1})\Big|^2
\Bigg]\Bigg|_{y=Y^{t,x,N,\delta,\cM}_{\kappa_N(r)}}
\right]\nonumber\\
&
\leq[\nu(A_\delta)]^2\cM^{-1}\bE\left[
\sum_{k=1}^d\bE\big[\big|\eta^k_{\kappa_N(r)}(y,V^{N,\delta,\cM}_{1,1})\big|^2\big]
\Bigg|_{y=Y^{t,x,N,\delta,\cM}_{\kappa_N(r)}}\right]\nonumber\\
&
=\nu(A_\delta)\cM^{-1}\bE\left[
\int_{A_\delta}\big\|\eta_{\kappa_N(r)}(Y^{t,x,N,\delta,\cM}_{\kappa_N(r)},z)
\big\|^2
\,\nu(dz)\right]
\label{est delta Y}
\\
&\leq 2\nu(A_\delta)\cM^{-1}\bE\left[
\int_{A_\delta}\big\|\eta_{\kappa_N(r)}(Y^{t,x,N,\delta,\cM}_{\kappa_N(r)},z)
-\eta_{\kappa_N(r)}(0,z)\big\|^2\,\nu(dz)
\right]\nonumber\\
&
\quad
+2\nu(A_\delta)\cM^{-1}\bE\left[
\int_{A_\delta}\big\|\eta_{\kappa_N(r)}(0,z)\big\|^2\,\nu(dz)
\right]\nonumber\\
&
\leq 2\nu(A_\delta)\cM^{-1}L\left(\bE\Big[
\big\|Y^{t,x,N,\delta,\cM}_{\kappa_N(r)}\big\|^2\Big]+d^p\right).
\label{error com int}
\end{align}
Thus, by \eqref{nu finite} and the assumption that $\cM\geq \delta^{-2}Kd^p$,
we have for all $s\in[t,T]$ that
\begin{align}
A_4(s)
&
\leq 2\nu(A_\delta)\cM^{-1}L(s-t)\left(
\int_t^s\bE\Big[\big\|Y^{t,x,N,\delta,\cM}_{\kappa_N(r)}\big\|^2\Big]\,dr
+(s-t)d^p
\right)\nonumber\\
&
\leq 2\delta^{-2}Kd^p\cM^{-1}L(s-t)\left(
\int_t^s\bE\Big[\big\|Y^{t,x,N,\delta,\cM}_{\kappa_N(r)}\big\|^2\Big]\,dr
+(s-t)d^p
\right)
\label{A4 delta M 0}
\\
&
\leq 2L(s-t)\left(
\int_t^s\bE\Big[\big\|Y^{t,x,N,\delta,\cM}_{\kappa_N(r)}\big\|^2\Big]\,dr
+(s-t)d^p
\right).
\label{A4 delta M}
\end{align}

Furthermore, using \eqref{assumption Lip mu sigma eta} 
and \eqref{assumption growth} we obtain for all
$s\in[t,T]$ that
\begin{align}
 & A_1(s)
\leq 2(s-t)\bE\Big[\int_t^s
\big\|\mu(\kappa_N(r),Y^{t,x,N,\delta,\cM}_{\kappa_N(r)})
-\mu(\kappa_N(r),0)\big\|^2
\,dr\Big]\nonumber\\
& \quad\quad\quad\;\;
+2(s-t)\bE\Big[\int_t^s
\big\|\mu(\kappa_N(r),0)\big\|^2
\,dr\Big]\nonumber\\
&
\quad\quad\;\;
\leq 2(s-t)L\int_t^s\bE
\Big[\big\|Y^{t,x,N,\delta,\cM}_{\kappa_N(r)}\big\|^2\Big]\,dr
+2(s-t)^2Ld^p,
\label{A1 delta M}\\
&
A_2(s)\leq 
16L\int_t^s\bE\Big[\big\|Y^{t,x,N,\delta,\cM}_{\kappa_N(r)}\big\|^2\Big]\,dr
+2(s-t)8Ld^p,
\label{A2 delta M}\\
&
A_3(s)\leq 
16L\int_t^s\bE\Big[\big\|Y^{t,x,N,\delta,\cM}_{\kappa_N(r)}\big\|^2\Big]\,dr
+2(s-t)8Ld^p.
\label{A3 delta M}
\end{align}
Then combining \eqref{a N delta M} and \eqref{A4 delta M}-\eqref{A3 delta M}
yields for all $s\in[t,T]$ that
\begin{equation}                                            \label{est a}                            
a^\delta_{N,\cM}(s)\leq 5\|x\|^2+20LT(T+8)d^p
+20L(T+8)\int_t^sa^\delta_{N,\cM}(r)\,dr.
\end{equation}
Therefore, by the same stopping time
argument used to prove \eqref{sup Euler} in Lemma \ref{Lemma Euler} and
Gr\"onwall's lemma we get that
$$
a^\delta_{N,\cM}(T)\leq C_1(d^p+\|x\|^2)e^{\rho_1(T-t)}
$$
with $C_1:=5\cdot\max\{1,4LT(T+8)\}$ and $\rho_1:=20L(T+8)$, which proves
\eqref{estimate M}. 

Next,  we start to prove \eqref{error M}. By \eqref{def Y N delta} and
\eqref{Y N delta M 1} we notice for all $s\in[t,T]$ that 
\begin{align*}
Y^{t,x,N,\delta,\cM}_s-Y^{t,x,N,\delta}_s=
&
\int_t^s\left(
\mu(\kappa_N(r-),Y^{t,x,N,\delta,\cM}_{\kappa_N(r-)})
-\mu(\kappa_N(r-),Y^{t,x,N,\delta}_{\kappa_N(r-)})
\right)dr\\
&
+\int_t^s\left(
\sigma(\kappa_N(r-),Y^{t,x,N,\delta,\cM}_{\kappa_N(r-)})
-\sigma(\kappa_N(r-),Y^{t,x,N,\delta}_{\kappa_N(r-)})
\right)dW_r\\
&
+\int_t^s\int_{A_\delta}\left(
\eta_{\kappa_N(r-)}(Y^{t,x,N,\delta,\cM}_{\kappa_N(r-)},z)
-\eta_{\kappa_N(r-)}(Y^{t,x,N,\delta}_{\kappa_N(r-)},z)
\right)\tilde{\pi}(dz,dr)\\
&
+\int_t^s\int_{A_\delta}
\eta_{\kappa_N(r-)}(Y^{t,x,N,\delta,\cM}_{\kappa_N(r-)},z)\,\nu(dz)\,dr
\\
&
-\int_t^s\Big(\frac{\nu(A_\delta)}{\cM}\sum_{j=1}^{\cM}
\eta_{\kappa_N(r-)}(Y^{t,x,N,\delta,\cM}_{\kappa_N(r-)},
V^{N,\delta,\cM}_{\frac{(\kappa_N(r-)-t)N}{T-t},j})
\Big)dr.
\end{align*}
Analogous to the argument to obtain \eqref{a N delta M}, 
by Burkholder-Davis-Gundy inequality and H\"older's inequality, for each 
$s\in[t,T]$ it holds that
\begin{equation}
e^\delta_{N,\cM}(s):=\bE\Bigg[
\sup_{r\in[t,s]}\big\|Y^{t,x,N,\delta,\cM}_r-Y^{t,x,N,\delta}_r\big\|^2\Bigg]
\leq 4\sum_{i=1}^3E_i(s)+4A_4(s),
\label{e N delta M}
\end{equation}
where $A_4(s)$ is defined by \eqref{def A4}, and
\begin{align*}
&
E_1(s):=(s-t)\bE\left[\int_t^s
\big\|\mu(\kappa_N(r),Y^{t,x,N,\delta,\cM}_{\kappa_N(r)})
-\mu(\kappa_N(r),Y^{t,x,N,\delta}_{\kappa_N(r)})\big\|^2\,dr\right],\\
&
E_2(s):=8\bE\left[\int_t^s
\big\|\sigma(\kappa_N(r),Y^{t,x,N,\delta,\cM}_{\kappa_N(r)})
-\sigma(\kappa_N(r),Y^{t,x,N,\delta}_{\kappa_N(r)})\big\|_F^2\,dr\right],\\
&
E_3(s):=8\bE\left[
\int_t^s\int_{A_\delta}\big\|
\eta_{\kappa_N(r)}(Y^{t,x,N,\delta,\cM}_{\kappa_N(r)},z)
-\eta_{\kappa_N(r)}(Y^{t,x,N,\delta}_{\kappa_N(r)},z)
\big\|^2\nu(dz)\,dr
\right]
\end{align*}
Moreover, by \eqref{assumption Lip mu sigma eta} we have for all $s\in[t,T]$ that
\begin{align}
&
E_1(s)\leq 2(s-t)L\int_t^s\bE\Big[\big\|
Y^{t,x,N,\delta,\cM}_{\kappa_N(r)}-Y^{t,x,N,\delta}_{\kappa_N(r)}
\big\|^2\Big]\,dr,
\label{E1 delta}\\
&
E_2(s)\leq 16L\int_t^s\bE\Big[\big\|
Y^{t,x,N,\delta,\cM}_{\kappa_N(r)}-Y^{t,x,N,\delta}_{\kappa_N(r)}
\big\|^2\Big]\,dr,
\label{E2 delta}\\
&
E_3(s)\leq 16L\int_t^s\bE\Big[\big\|
Y^{t,x,N,\delta,\cM}_{\kappa_N(r)}-Y^{t,x,N,\delta}_{\kappa_N(r)}
\big\|^2\Big]\,dr.
\label{E3 delta}
\end{align}
Hence, combining \eqref{e N delta M}, \eqref{E1 delta}-\eqref{E3 delta}, 
and \eqref{A4 delta M 0} we obtain 
for all $s\in[t,T]$ that
\begin{align*}
e^\delta_{N,\cM}(s)\leq 8L(16+T)\int_t^se^\delta_{N,\cM}(r)\,dr
+8L\delta^{-2}Kd^p\cM^{-1}T\left(
\int_t^s\bE\Big[\big\|Y^{t,x,N,\delta,\cM}_{\kappa_N(r)}\big\|^2\Big]\,dr
+(s-t)d^p\right).
\end{align*}
Thus, by \eqref{estimate M} we have for all $s\in[t,T]$ that
\begin{equation}                                        \label{estimate e}
e^\delta_{N,\cM}(s)\leq 8L(16+T)\int_t^s e^\delta_{N,\cM}(r)\,dr
+16\delta^{-2}Kd^p\cM^{-1}LT^2C_1(d^p+\|x\|^2)e^{\rho_1(T-t)}.
\end{equation}
Furthermore, considering \eqref{est Y delta} and \eqref{estimate M},
we apply Gr\"onwall's lemma to \eqref{estimate e} to get
$$
e^{\delta}_{N,\cM}(T)\leq C_2\delta^{-2}d^pK\cM^{-1}(d^p+\|x\|^2)e^{\rho_2(T-t)},
$$
where $C_2:=16LT^2C_1$ and $\rho_2:=8L(16+T)$ are positive constants. 
Then, the proof of Lemma \eqref{lemma delta M} is completed.
\end{proof}

\begin{lemma}                                      \label{Lemma Y s s'}
There exist positive constants $C_1$ and $\rho_1$ only depending on $T$ and $L$
satisfying for all $(t,x)\in[0,T]\times\bR^d$, $s\in[t,T]$, $s'\in[s,T]$, 
$\delta\in(0,1)$, and $N,\cM\in\bN$ with $\cM\geq \delta^{-2}Kd^p$ that
\begin{equation}                                            \label{M delta s s'}
\bE\left[\big\|Y^{t,x,N,\delta,\cM}_{s'}-Y^{t,x,N,\delta,\cM}_{s}\big\|^2\right]
\leq 8L(s'-s)(T+2+\nu(A_\delta)\cM^{-1})
\left[d^p+C_1(d^p+\|x\|^2)e^{\rho_1T}\right]
\end{equation}
\end{lemma}

\begin{proof}
We fix $(t,x)\in[0,T]\times\bR^d$, $s\in[t,T]$, $s'\in[s,T]$, 
$\delta\in(0,1)$, and $N,\cM\in\bN$ with $\cM\geq \delta^{-2}Kd^p$ throughout
this proof. 
By \eqref{Y N delta M 1}, It\^o's isometry,  H\"older's inequality,
and the fact that $(\sum_{i=1}^4a_i)^2\leq 4\sum_{i=1}^4a_i^2$ for $a_i\in\bR$,
$i=1,...,4$, we have 
\begin{align}
\bE\left[\big\|Y^{t,x,N,\delta,\cM}_{s'}-Y^{t,x,N,\delta,\cM}_s\big\|^2\right]
&
\leq 4(s'-s)\int_s^{s'}\bE\left[\left\|\mu(\kappa_N(r),
Y^{t,x,N,\delta,\cM}_{\kappa_N(r)})
\right\|^2\right]\,dr\nonumber\\
&
\quad
+4\int_s^{s'}\bE\left[\left\|\sigma(\kappa_N(r),
Y^{t,x,N,\delta,\cM}_{\kappa_N(r)})
\right\|_F^2\right]\,dr
\nonumber\\
&
\quad +4\int_s^{s'}\bE\left[\int_{\bR^{d}}
\left\|\eta_{\kappa_N(r)}
(Y^{t,x,N,\delta,\cM}_{\kappa_N(r)},z)\right\|^2\,\nu(dz)\right]dr
\nonumber\\
&
\quad +4\bE\Bigg[\Big\|
\int_s^{s'}\Big(
\int_{A_\delta}\eta_{\kappa_N(r)}(Y^{t,x,N,\delta,\cM}_{\kappa_N(r)},z)
\,\nu(dz)\nonumber\\
&
\quad\quad
-\frac{\nu(A_\delta)}{\cM}\sum_{j=1}^\cM\eta_{\kappa_N(r)}
(Y^{t,x,N,\delta,\cM}_{\kappa_N(r)},V^{N,\delta,\cM}_{\frac{(\kappa_N(r)-t)N}{T-t},j})
\Big)\,dr
\Big\|^2\Bigg].                                             
\label{error s s' 1}
\end{align}
Moreover, by H\"older's inequality, \eqref{error com int}, and \eqref{estimate M}
it holds that
\begin{align}
&
\bE\Bigg[\Big\|
\int_s^{s'}\Big(
\int_{A_\delta}\eta_{\kappa_N(r)}(Y^{t,x,N,\delta,\cM}_{\kappa_N(r)},z)
\,\nu(dz)-
\frac{\nu(A_\delta)}{\cM}\sum_{j=1}^\cM\eta_{\kappa_N(r)}
(Y^{t,x,N,\delta,\cM}_{\kappa_N(r)},V^{N,\delta,\cM}_{\frac{(\kappa_N(r)-t)N}{T-t},j})
\Big)\,dr
\Big\|^2\Bigg]\nonumber\\
&
\leq 2(s'-s)\nu(A_\delta)\cM^{-1}L\int_s^{s'}\left(
d^p+\bE\Big[\big\|Y^{t,x,N,\delta,\cM}_{\kappa_N(r)}\big\|^2\Big]
\right)dr\nonumber\\
&
\leq 2(s'-s)^2\nu(A_\delta)\cM^{-1}L
\left[d^p+C_1(d^p+\|x\|^2)e^{\rho_1T}\right],             \label{delta s s'}
\end{align}
where $C_1=C_1(T,L)$ and $\rho_1=\rho_1(T,L)$ are positive constants 
introduced in \eqref{estimate M} in Lemma~\ref{lemma delta M}.
Thus, by \eqref{error s s' 1}, \eqref{delta s s'}, \eqref{assumption Lip mu sigma eta}, and \eqref{assumption growth} we obtain
\begin{align}
&
\bE\left[\big\|Y^{t,x,N,\delta,\cM}_{s'}-Y^{t,x,N,\delta,\cM}_{s}\big\|^2\right]
\nonumber\\
&
\leq 
4(s'-s)\int_s^{s'}\bE
\left[2\left\|\mu(\kappa_N(r),Y^{t,x,N,\delta,\cM}_{\kappa_N(r)})
-\mu(\kappa_N(r),0)\right\|^2
+2\left\|\mu(\kappa_N(r),0)\right\|^2\right]\,dr
\nonumber\\
&
\quad +4\int_s^{s'}\bE\left[2\left\|
\sigma(\kappa_N(r),Y^{t,x,N,\delta,\cM}_{\kappa_N(r)})
-\sigma(\kappa_N(r),0)\right\|_F^2
+2\left\|\sigma(\kappa_N(r),0)\right\|^2\right]\,dr
\nonumber\\
&
\quad +4\int_s^{s'}\bE\left[\int_{\bR^{d}}
\left(2\left\|\eta_{\kappa_N(r)}
(Y^{t,x,N,\delta,\cM}_{\kappa_N(r)},z)-\eta_{\kappa_N(r)}(0,z)\right\|^2
+2\left\|\eta_{\kappa_N(r)}(0,z)\right\|^2\right)\nu(dz)\right]dr\nonumber\\
&
\quad+8(s'-s)^2\nu(A_\delta)\cM^{-1}L
\left[d^p+C_1(d^p+\|x\|^2)e^{\rho_1T}\right]\nonumber\\
&
\leq 8L(s'-s)\int_s^{s'}\left(\bE\left[\big\|
Y^{t,x,N,\delta,\cM}_{\kappa_N(r)}\big\|^2+d^p\right]\right)dr
+16L\int_s^{s'}\left(\bE\left[\big\|
Y^{t,x,N,\delta,\cM}_{\kappa_N(r)}\big\|^2+d^p\right]\right)dr\nonumber\\
&
\quad+8(s'-s)^2\nu(A_\delta)\cM^{-1}L
\left[d^p+C_1(d^p+\|x\|^2)e^{\rho_1T}\right].\nonumber
\end{align}
Hence, by \eqref{estimate M} we obtain
for all $s\in[t,T]$ and $s'\in[s,T]$ that
\begin{align}
\bE\left[\big\|Y^{t,x,N,\delta,\cM}_{s'}-Y^{t,x,N,\delta,\cM}_{s}\big\|^2\right]
&
\leq 8L(s'-s)^2\left[d^p+C_1(d^p+\|x\|^2)e^{\rho_1T}\right]\nonumber\\
& \quad
+16L(s'-s)\left[d^p+C_1(d^p+\|x\|^2)e^{\rho_1T}\right]\nonumber\\
& \quad
+8(s'-s)\nu(A_\delta)\cM^{-1}L\left[d^p+C_1(d^p+\|x\|^2)e^{\rho_1T}\right]
\nonumber\\
&
\leq 8L(s'-s)(T+2+\nu(A_\delta)\cM^{-1})
\left[d^p+C_1(d^p+\|x\|^2)e^{\rho_1T}\right],\nonumber            
\end{align}
which completes the proof of this lemma.
\end{proof}

\begin{lemma}                                                \label{Lemma e_s}
There exist constants $C_1$ and $\rho_1$ only depending on $T$ and $L$ satisfying
for all $t\in[0,T]$, $t'\in[t,T]$,
$s\in[t',T]$, $x,x'\in\bR^d$, $\delta\in(0,1)$, 
and $N,\cM\in\bN$ with $\cM\geq \delta^{-2}Kd^p$ that
\begin{align}
e_s & :=\sup_{u\in[t',s]}\bE\Big[\big\|Y^{t,x,N,\delta,\cM}_u
-Y^{t',x',N,\delta,\cM}_u\big\|^2\Big]\nonumber\\
&
\leq 16\exp\left\{4LT\left(2T+8+\nu(A_\delta)\cM^{-1}T\right)\right\}\nonumber\\
&
\quad
\cdot
\Big\{
(\|x-x'\|+|t'-t|+|t'-t|^{1/2})\big[1+3L^{1/2}d^{p/2}(1+\|x\|^2)^{1/2}\big]^2\nonumber\\
& 
\quad
+2L(t'-t)\nu(A_\delta)\cM^{-1}
\left[d^p+C_1(d^p+\|x\|^2)e^{\rho_1T}\right]
\Big\}.                                                       \label{sup t t'}
\end{align}
\end{lemma}

\begin{proof}
We fix $x,x'\in\bR^d$, $t\in[0,T]$, $t'\in[t,T]$, 
$\delta\in(0,1)$, and $N,\cM\in\bN$ with $\cM\geq \delta^{-2}Kd^p$ throughout
this proof. One notices that almost surely for all $s\in[t',T]$
\begin{align*}
& Y^{t,x,N,\delta,\cM}_s-Y^{t',x',N,\delta,\cM}_s\\
&
=
Y^{t,x,N,\delta,\cM}_{t'}-x'
+\int_{t'}^s\left(\mu(\kappa_N(r-),Y^{t,x,N,\delta,\cM}_{\kappa_N(r-)})
-\mu(\kappa_N(r-),Y^{t',x',N,\delta,\cM}_{\kappa_N(r-)})\right)\,dr\\
& \quad
+\int_{t'}^s\left(\sigma(\kappa_N(r-),
Y^{t,x,N,\delta,\cM}_{\kappa_N(r-)})
-\sigma(\kappa_N(r-),Y^{t',x',N,\delta,\cM}_{\kappa_N(r-)})\right)\,dW_r\\
& \quad
+\int_{t'}^s\int_{A_\delta}
\left(\eta_{\kappa_N(r-)}(Y^{t,x,N,\delta,\cM}_{\kappa_N(r-)},z)
-\eta_{\kappa_N(r-)}(Y^{t',x',N,\delta,\cM}_{\kappa_N(r-)},z)\right)
\,\tilde{\pi}(dz,dr)\\
& \quad
+\int_{t'}^s\int_{A_\delta}
\left(\eta_{\kappa_N(r-)}(Y^{t,x,N,\delta,\cM}_{\kappa_N(r-)},z)
-\eta_{\kappa_N(r-)}(Y^{t',x',N,\delta,\cM}_{\kappa_N(r-)},z)\right)
\nu(dz)\,dr
\\
& \quad
-\int_{t'}^s\Bigg[\frac{\nu(A_\delta)}{\cM}\sum_{j=1}^{\cM}
\Big(
\eta_{\kappa_N(r-)}(Y^{t,x,N,\delta,\cM}_{\kappa_N(r-)},
V^{N,\delta,\cM}_{\frac{(\kappa_N(r-)-t)N}{T-t},j})
\\
& \quad
-\eta_{\kappa_N(r-)}(Y^{t',x',N,\delta,\cM}_{\kappa_N(r-)},
V^{N,\delta,\cM}_{\frac{(\kappa_N(r-)-t)N}{T-t},j})
\Big)
\Bigg]dr.
\end{align*}
Therefore, by Minkowski's inequality, H\"older's inequality, It\^o's isometry,
H\"older's inequality,
and the analogous argument to get \eqref{est delta Y}, we have for all
$s\in[t',T]$ that
\begin{align*}
e_s^{1/2}\leq & 
\left(\bE\left[\left\|Y^{t,x,N,\delta,\cM}_{t'}-x'\right\|^2\right]\right)^{1/2}
+L^{1/2}\left((s-t')\int_{t'}^s
\bE\left[\left\|Y^{t,x,N,\delta,\cM}_{\kappa_N(r)}
-Y^{t',x',N,\delta,\cM}_{\kappa_N(r)}\right\|^2\right]dr\right)^{1/2}\\
& 
+2L^{1/2}\left(\int_{t'}^s
\bE\left[\left\|Y^{t,x,N,\delta,\cM}_{\kappa_N(r)}
-Y^{t',x',N,\delta,\cM}_{\kappa_N(r)}\right\|^2\right]\,dr\right)^{1/2}\\
&
+\left(\nu(A_\delta)\cM^{-1}T\int_{t'}^s \bE\left[\int_{A_\delta}
\big\|\eta_{\kappa_N(r)}(Y^{t,x,N,\delta,\cM}_{\kappa_N(r)},z)-
\eta_{\kappa_N(r)}(Y^{t',x',N,\delta,\cM}_{\kappa_N(r)},z)\big\|^2\,
\nu(dz)\right]\,dr\right)^{1/2}.
\end{align*} 
Thus, by \eqref{assumption Lip mu sigma eta} and the fact that
$(\sum_{i=1}^4a_i)^2\leq 4\sum_{i=1}^4a_i^2$ for $a_i\in\bR$, $i=1,...,4$,
we have for all $s\in[t',T]$ that
\begin{align}
e_s 
\leq 4\bE\left[\big\|Y^{t,x,N,\delta,\cM}_{t'}-x'\big\|^2\right]
+4L(T+4)\int_{t'}^se_r\,dr
+4L\nu(A_\delta)\cM^{-1}T\int_{t'}^se_r\,dr.
\end{align}
Hence, by \eqref{estimate M} and 
Gr\"onwall's lemma, we get for all $s\in[t',T]$ that
\begin{equation}
e_s\leq 4\bE\left[\big\|Y^{t,x,N,\delta,\cM}_{t'}-x'\big\|^2\right]\cdot
\exp\big\{4LT\big(T+4+\nu(A_\delta)\cM^{-1}T\big)\big\}.                    \label{e M s 1}
\end{equation}
Furthermore, we notice for all $s\in[t,t']$ almost surely that
\begin{align*}
Y^{t,x,N,\delta,\cM}_{s}-x'=
&
x-x'+\int_t^{s}
\mu(\kappa_N(r-),Y^{t,x,N,\delta,\cM}_{\kappa_N(r-)})\,dr
+\int_{t}^{s}\sigma(\kappa_N(r-),Y^{t,x,N,\delta,\cM}_{\kappa_N(r-)})\,dW_r\\
&
+\int_t^{s}\int_{\bR^{d}}\eta_{\kappa_N(r-)}
(Y^{t,x,N,\delta,\cM}_{\kappa_N(r-)},z)
\,\,\tilde{\pi}(dz,dr)\\
&
+\int_t^s\int_{A_\delta}
\eta_{\kappa_N(r-)}(Y^{t,x,N,\delta,\cM}_{\kappa_N(r-)},z)\,\nu(dz)\,dr
\\
&
-\int_t^s\Big(\frac{\nu(A_\delta)}{\cM}\sum_{j=1}^{\cM}
\eta_{\kappa_N(r-)}(Y^{t,x,N,\delta,\cM}_{\kappa_N(r-)},
V^{N,\delta,\cM}_{\frac{(\kappa_N(r-)-t)N}{T-t},j})
\Big)dr.
\end{align*}
Then, by Minkowski's inequality, H\"older's inequality, It\^o's isometry,
\eqref{delta s s'},
\eqref{assumption Lip mu sigma eta}, and \eqref{assumption growth},
we obtain for all $s\in[t,t']$ that
\begin{align*}
&
\left(\sup_{u\in[t,s]}\bE\left[\left\|Y^{t,x,N,\delta,\cM}_{u}
-x'\right\|^2\right]\right)^{1/2}\\
&\leq
\|x-x'\|
+\left((s-t)\int_{t}^{s}\|\mu(\kappa_N(r),x')\|^2\,dr\right)^{1/2}\\
& \quad
+\left((s-t)\int_t^{s}
\bE\left[\left\|\mu(\kappa_N(r),Y^{t,x,N,\delta,\cM}_{\kappa_N(r)})
-\mu(\kappa_N(r),x')\right\|^2
\right]dr\right)^{1/2}\\
&
\quad+\left(\int_t^{s}\|\sigma(\kappa_N(r),x')\|_F^2\,dr\right)^{1/2}
+\left(\int_t^{s}
\bE\left[\left
\|\sigma(\kappa_N(r),Y^{t,x,N,\delta,\cM}_{\kappa_N(r)})
-\sigma(\kappa_N(r),x')\right\|_F^2\right]\,dr\right)^{1/2}\\
&
\quad+\left(\int_t^{s}\bE\left[\int_{\bR^{d}}
\left\|\eta_{\kappa_N(r)}(Y^{t,x,N,\delta,\cM}_{\kappa_N(r)},z)
-\eta_{\kappa_N(r)}(x',z)\right\|^2\nu(dz)\right]dr\right)^{1/2}\\
& \quad
+\left(\int_t^{s}\int_{\bR^{d}}\|\eta_{\kappa_N(r)}(x',z)\|^2
\,\nu(dz)\,dr\right)^{1/2}
\\
&
\quad +\left[2(s-t)^2\nu(A_\delta)\cM^{-1}L
\big(d^p+C_1(d^p+\|x\|^2)e^{\rho_1T}\big)\right]^{1/2}
\\
&
\leq \|x-x'\|
+\left((s-t)\int_{t}^{s}\|\mu(\kappa_N(r),x')\|^2\,dr\right)^{1/2}
+L^{1/2}\left((s-t)\int_t^{s}
\bE\left[\left\|Y^{t,x,N,\delta,\cM}_{\kappa_N(r)}
-x'\right\|^2\right]dr\right)^{1/2}\\
&
\quad+2L^{1/2}\left(\int_t^{s}
\bE\left[\left\|Y^{t,x,N,\delta,\cM}_{\kappa_N(r)}-x'\right\|^2\right]
\,dr\right)^{1/2}
+\left(\int_t^{s}\|\sigma(\kappa_N(r),x')\|_F^2\,dr\right)^{1/2}\\
& \quad
+\left(\int_t^{s}\int_{\bR^{d}}\|
\eta_{\kappa_N(r)}(x',z)\|^2\,\nu(dz)\,dr\right)^{1/2}
+\left[2(s-t)^2\nu(A_\delta)\cM^{-1}L
\big(d^p+C_1(d^p+\|x\|^2)e^{\rho_1T}\big)\right]^{1/2}.
\end{align*}

Hence, by \eqref{assumption growth} and the fact that
$(\sum_{i=1}^4a_i)^2\leq 4\sum_{i=1}^4a_i^2$ for $a_i\in\bR$, $i=1,...,4$, it holds for all $s\in[t,t']$ that
\begin{align*}
\sup_{u\in[t,s]}\bE\left[\big\|Y^{t,x,N,\delta,\cM}_{u}-x'\big\|^2\right]
&
\leq 4L(T+4)\int_t^{s}
\sup_{u\in[t,r]}\bE\left[\big\|Y^{t,x,N,\delta,\cM}_{u}-x'\big\|^2\right]dr
\\
& \quad
+4\Bigg[\|x-x'\|+\Bigg(|t'-t|\cdot\int_t^{t'}\|
\mu(\kappa_N(r),x')\|^2\,dr\Bigg)^{1/2}
\\
& \quad
+\Bigg(\int_t^{t'}\|\sigma(\kappa_N(r),x')\|_F^2\,dr\Bigg)^{1/2}
\\
& \quad
+\Bigg(\int_t^{t'}\int_{\bR^{d}}
\|\eta_{\kappa_N(r)}(x',z)\|^2\,\nu(dz)\,dr\Bigg)^{1/2}
\Bigg]^2\\
& \quad
+8L(t'-t)\nu(A_\delta)\cM^{-1}
\left[d^p+C_1(d^p+\|x\|^2)e^{\rho_1T}\right]\\
&
\leq 4L(T+4)\int_t^{s}
\sup_{u\in[t,r]}\bE\left[\big\|Y^{t,x,N,\delta,\cM}_{u}-x'\big\|^2\right]dr
\\
& \quad
+4(\|x-x'\|+|t'-t|+|t'-t|^{1/2})\big[1+3L^{1/2}d^{p/2}(1+\|x\|^2)^{1/2}\big]^2\\
& \quad
+8L(t'-t)\nu(A_\delta)\cM^{-1}
\left[d^p+C_1(d^p+\|x\|^2)e^{\rho_1T}\right].
\end{align*}
Thus, by \eqref{estimate M} and Gr\"onwall's lemma, we obtain that
\begin{align}
&
\sup_{u\in[t,t']}\bE\left[\big\|Y^{t,x,N,\delta,\cM}_{u}-x'\big\|^2\right]
\nonumber\\
&
\leq 
4e^{4LT(T+4)} \Big\{
(\|x-x'\|+|t'-t|+|t'-t|^{1/2})
\big[1+3L^{1/2}d^{p/2}(1+\|x\|^2)^{1/2}\big]^2\nonumber\\
& \quad
+2L(t'-t)\nu(A_\delta)\cM^{-1}
\left[d^p+C_1(d^p+\|x\|^2)e^{\rho_1T}\right]
\Big\}.                                                \label{sup Y x' t t'}                                                 
\end{align}
Finally, combining \eqref{e M s 1} and \eqref{sup Y x' t t'}
yields \eqref{sup t t'}. Thus, the proof of this lemma is completed.
\end{proof}

\begin{corollary}                                            \label{lemma prob 2}
Let $(t,x)\in[0,T]\times\bR^d$, $s\in[t,T]$,
and let $(t_k,s_k,x_k)_{k=1}^\infty$ be a sequence such that 
$(t_k,s_k,x_k)\in\Lambda\times\bR^d$ for all $k\in\bN$, 
and $\lim_{k\to\infty}(t_k,s_k,x_k)=(t,s,x)$.
Then it holds for for all $\varepsilon>0$,
$\delta\in(0,1)$, 
and $N,\cM\in\bN$ with $\cM\geq \delta^{-2}Kd^p$, that
\begin{equation}                                    \label{conv prob 2}
\lim_{k\to\infty}\mathbb{P}
\left(\left\|Y^{t_k,x_k,N,\delta,\cM}_{s_k}-Y^{t,x,N,\delta,\cM}_s\right\|
\geq \varepsilon\right)=0.
\end{equation}
\end{corollary}

\begin{proof}
By Lemma \ref{Lemma Y s s'} and Lemma \ref{Lemma e_s}, one can precisely follow
the proof of Corollary \ref{corollary prob 1} to obtain \eqref{conv prob 2}.
\end{proof}

\end{document}

%% file: table.txt
10 & Avg. Sol. & $12.6439$ & $12.8667$ & $12.7782$ & $12.8062$ & $12.8130$ \\ 
 & \textit{Std. Dev.} & \textit{0.7686} & \textit{0.3133} & \textit{0.1732} & \textit{0.0519} & \textit{0.0193} \\ 
 & Avg. Eval. & $4.23 \cdot 10^{4}$ & $3.56 \cdot 10^{5}$ & $5.03 \cdot 10^{6}$ & $9.88 \cdot 10^{7}$ & $2.45 \cdot 10^{9}$ \\ 
 & \textit{Avg. Time} & \textit{0.0016} & \textit{0.0141} & \textit{0.1534} & \textit{2.8302} & \textit{70.5218} \\ 
\hline 
50 & Avg. Sol. & $11.7451$ & $12.1087$ & $11.8201$ & $11.7833$ & $11.8102$ \\ 
 & \textit{Std. Dev.} & \textit{0.5888} & \textit{0.4380} & \textit{0.1114} & \textit{0.0414} & \textit{0.0115} \\ 
 & Avg. Eval. & $1.87 \cdot 10^{5}$ & $1.67 \cdot 10^{6}$ & $2.37 \cdot 10^{7}$ & $4.57 \cdot 10^{8}$ & $1.13 \cdot 10^{10}$ \\ 
 & \textit{Avg. Time} & \textit{0.0016} & \textit{0.0251} & \textit{0.3770} & \textit{5.9702} & \textit{151.1463} \\ 
\hline 
100 & Avg. Sol. & $11.3215$ & $11.4291$ & $11.4581$ & $11.4258$ & $11.4468$ \\ 
 & \textit{Std. Dev.} & \textit{0.3986} & \textit{0.3069} & \textit{0.1650} & \textit{0.0432} & \textit{0.0092} \\ 
 & Avg. Eval. & $3.76 \cdot 10^{5}$ & $3.17 \cdot 10^{6}$ & $4.66 \cdot 10^{7}$ & $9.06 \cdot 10^{8}$ & $2.24 \cdot 10^{10}$ \\ 
 & \textit{Avg. Time} & \textit{0.0079} & \textit{0.0533} & \textit{0.6849} & \textit{12.3840} & \textit{300.3673} \\ 
\hline 
500 & Avg. Sol. & $10.8580$ & $10.6897$ & $10.6524$ & $10.6949$ & $10.6960$ \\ 
 & \textit{Std. Dev.} & \textit{0.3835} & \textit{0.2409} & \textit{0.0701} & \textit{0.0405} & \textit{0.0069} \\ 
 & Avg. Eval. & $1.66 \cdot 10^{6}$ & $1.60 \cdot 10^{7}$ & $2.31 \cdot 10^{8}$ & $4.50 \cdot 10^{9}$ & $1.11 \cdot 10^{11}$ \\ 
 & \textit{Avg. Time} & \textit{0.0251} & \textit{0.2411} & \textit{3.3046} & \textit{63.6527} & \textit{1563.3035} \\ 
\hline 
1000 & Avg. Sol. & $10.6316$ & $10.4786$ & $10.3588$ & $10.4011$ & $10.4021$ \\ 
 & \textit{Std. Dev.} & \textit{0.4451} & \textit{0.3437} & \textit{0.0910} & \textit{0.0265} & \textit{0.0092} \\ 
 & Avg. Eval. & $3.54 \cdot 10^{6}$ & $3.26 \cdot 10^{7}$ & $4.68 \cdot 10^{8}$ & $9.00 \cdot 10^{9}$ & $2.23 \cdot 10^{11}$ \\ 
 & \textit{Avg. Time} & \textit{0.0549} & \textit{0.4985} & \textit{6.6654} & \textit{126.1326} & \textit{3039.6606} \\ 
\hline 
5000 & Avg. Sol. & $9.7696$ & $9.8894$ & $9.7807$ & $9.7843$ & $9.7854$ \\ 
 & \textit{Std. Dev.} & \textit{0.5765} & \textit{0.2825} & \textit{0.0763} & \textit{0.0235} & \textit{0.0105} \\ 
 & Avg. Eval. & $1.85 \cdot 10^{7}$ & $1.60 \cdot 10^{8}$ & $2.30 \cdot 10^{9}$ & $4.48 \cdot 10^{10}$ & $1.11 \cdot 10^{12}$ \\ 
 & \textit{Avg. Time} & \textit{0.2481} & \textit{2.2290} & \textit{31.9406} & \textit{616.2213} & \textit{15233.9473} \\ 
\hline 
10000 & Avg. Sol. & $9.8036$ & $9.4448$ & $9.5468$ & $9.5395$ & $9.5347$ \\ 
 & \textit{Std. Dev.} & \textit{0.3240} & \textit{0.1942} & \textit{0.0539} & \textit{0.0311} & \textit{0.0057} \\ 
 & Avg. Eval. & $3.77 \cdot 10^{7}$ & $3.25 \cdot 10^{8}$ & $4.61 \cdot 10^{9}$ & $8.98 \cdot 10^{10}$ & $2.22 \cdot 10^{12}$ \\ 
 & \textit{Avg. Time} & \textit{0.5087} & \textit{4.4832} & \textit{63.7070} & \textit{1234.3232} & \textit{30334.4023} \\ 
\hline 